\documentclass[12pt]{article}
\usepackage{mathrsfs}
\usepackage{amssymb}
\usepackage{color}

\usepackage{amsthm}
\usepackage{amsmath}
\usepackage{graphicx}
\allowdisplaybreaks

 \newtheorem{thm}{Theorem}[section]
 
 \newtheorem{lem}[thm]{Lemma}
 
 \theoremstyle{definition}
 
 \newtheorem{rem}[thm]{Remark}
 \numberwithin{equation}{section}

\newtheorem{theorem}{Theorem}[section]

\theoremstyle{definition}

\theoremstyle{remark}
\newtheorem{remark}[theorem]{Remark}

\begin{document}
\title{Global Well-posedness of Incompressible Elastodynamics in Three Dimensional Thin Domain}
\author{Yuan Cai
\footnote{Department of Mathematics, The Hong Kong University of Science and Technology, Clear Water Bay, Kowloon, Hong Kong,
 Email: maycai@ust.hk}
 \and
 Fan Wang
\footnote{Department of Mathematics, Southwest Jiaotong University, Chengdu, 611756, China,
Email: wangf767@swjtu.edu.cn}
 }

\date{}
\maketitle

\begin{abstract}
In this article, we prove global existence of classical solutions to the incompressible isotropic Hookean elastodynamics in three-dimensional thin domain $\Omega_\delta=\mathbb{R}^2\times [0,\delta]$  with periodic boundary condition.
This system is essentially two dimensional
quasilinear wave type equations.
Following the classical vector field theory and the generalized energy method of Klainerman,
we introduce the anisotropic weighted Sobolev type inequalities,
the anisotropic generalized energy and the anisotropic weighted $L^2$ norm adapted to the thin domain.
The main issue in the anisotropic generalized energy estimate is that,
the pressure gradient brings an extra singular $\delta^{-1}$ factor arising from the  $\partial^2_3$ derivative.
Based on the inherent cancellation structure, we introduce an auxiliary anisotropic generalized energy to overcome this difficulty.
 The ghost weight technique introduced by Alinhac \cite{Alinhac01a} and the strong null condition introduced by Lei \cite{Lei16}
play an important role for temporal decay estimates.
\end{abstract}

\maketitle





\section{Introduction}

This paper is concerned with the global existence of  classical solutions to
the Cauchy problem of incompressible nonlinear elastodynamics
in three dimensional thin domain $\Omega_\delta=\mathbb{R}^2\times [0,\delta]$  with periodic boundary condition.
For homogeneous, isotropic and hyperelastic materials, the motion of
the elastic fluids is determined by the following Lagrangian
functional of flow maps \cite{Lei16}:
\begin{eqnarray}\label{LaF}
\mathcal{L}(X; T, \Omega_\delta) &=&
\int_0^T\int_{\Omega_\delta}\big(\frac{1}{2}|\partial_tX(t, y)|^2 -
W(\nabla X(t, y))\\\nonumber && +\ p(t, y)\big[\det\big(\nabla X(t,
y)\big) - 1\big]\big)dydt.
\end{eqnarray}
Here $X(t, y)$ is the flow map, $W \in C^\infty(GL_3, \mathbb{R}_+)$ is the strain energy
function which depends only on the deformation tensor $\nabla X$,
 $p(t, y)$ is a Lagrangian
multiplier which forces the flow maps to be
incompressible.
Without loss of generality, we will only focus on
the typical Hookean case, in
which the strain energy functional is given by
$$W(\nabla X) = \frac{1}{2}|\nabla X|^2.$$
In this case, the Euler-Lagrangian equation of \eqref{LaF} takes
\begin{equation}\label{Elasto-Hook}
\partial_t^2 X - \Delta X = - (\nabla X)^{- \top}\nabla p.
\end{equation}
On the other hand,
according to the conservation law of mass, there holds
\begin{equation}\nonumber
\int_{\Omega}dy = \int_{\Omega_t}dX,\quad \Omega_t = \{X(t, y)|y \in
\Omega\}
\end{equation}
for any smooth bounded connected domain $\Omega$. This implies that
\begin{equation}\label{incom}
\det(\nabla X) \equiv 1.
\end{equation}
The above \eqref{incom} is equivalent to the incompressible
condition.

There are lots of works on the nonlinear wave equations and elastodynamics in the whole space, see for instance \cite{Alinhac10, H97, LZ17}.
However, the research in the thin domain is almost blank.
Hence, it's very interesting to study
the corresponding problem.
Our main goal in this paper is to establish the global existence result of \eqref{Elasto-Hook}-\eqref{incom} in the thin domain.

\subsection{A review of related results}
Let us first review a few highlights from the
existence theory of nonlinear wave equations and elastodynamics from various aspects.
For quasilinear wave equations in dimension three, and for small initial
data, John and Klainerman \cite{JohnKlainerman84} obtained the first non-trivial almost global existence result.
The global existence of nonlinear wave and elastic systems in dimension no bigger than three hinges on two basic assumptions: the smallness of the initial data and the null
condition of nonlinearities. The omission of either of these assumptions can lead to the break down of solutions in finite time \cite{John81, John84, T98}.
Under the null condition, the fundamental work of global existence
of classical solutions for 3D scalar quasilinear wave equation
was achieved by Klainerman \cite{Klainerman86} and by Christodoulou \cite{Christodoulou86} independently.
The existence question is more delicate in 2-D case due to the
much weaker time decay rate. Even
under the assumption of the null condition, quadratic nonlinearities
have at most critical time decay.
 A series of articles considered
the case of cubically nonlinear equations satisfying the null
condition, see for example \cite{G93, HK10, Katayama1, Katayama}. Alinhac
\cite{Alinhac01a} was the first to establish global existence for 2D
scalar wave equation with null bilinear forms, his argument combines
vector fields  with
the ghost energy method.

For compressible elastodynamic  systems, commonly referred as elastic waves in literature,
John \cite{John88} showed the existence of almost global solutions
for small displacement (see also \cite{KS96}).
The breakthrough work of global well-posedness is due to Sideris \cite{Sideris96, Sideris00} and also Agemi \cite{Agemi00}, under a nonresonance condition which is
physically consistent with the system.
With an
additional repulsive Poisson term, the global existence was established in \cite{HM17} which
allows a general form for the pressure.
For the incompressible elastodynamics,
the
global-well-posedness for the 3D case was obtained by Sideris and Thomases in \cite{ST05,
ST07}. We also refer to \cite{ST06} for a  unified treatment, and \cite{LW} for some
improvement on the uniform time-independent bound on the highest order energy.
The uniform bound of the highest-order energy of the 2D case was proved recently by Cai \cite{CY}.

The first non-trivial long time existence result was established
 by Lei, Sideris and Zhou
\cite{LSZ13}, in which the authors proved the almost global
existence for the 2D incompressible isotropic elastodynamics
 in Eulerian coordinate.
The global well-posedness was finally solved by Lei \cite{Lei16}
where he found that in Lagrangian coordinate, the incompressible elastodynamics
actually inherently satisfies the strong null condition.
Afterwards, Wang \cite{W17}
gave a new proof of this latter result using space-time resonance method and
a normal form transformation.

Now we recall some results on the well-posedness of the thin domain.
In the early 1990s, Hale and Raugel studied reaction diffusion equations and damped wave equations on thin domain, see
\cite{hyperbolic} and \cite{reaction}. Later on, Raugel and Sell  \cite{Raugel1, Raugel2} further studied the existence of strong solutions of the Navier-Stokes
equations on thin domain.
See also \cite{2Dper, DG, KZ1, KZ2, Z} for further results.
In all these works, a dissipative structure (with viscosity) is a key ingredient to study the long time existence.  Consequently, these arguments cannot be applied to study long time solutions for wave type equations in thin domain.

Before ending this subsection, let us mention some related works on
viscoelasticity, where there is viscosity in the momentum equation.
The global well-posedness with near equilibrium states was first
obtained in \cite{LLZ2005} (see also \cite{LZ05}). The 3-D case
was obtained independently in  \cite{CZ2006} (see also the thesis
\cite{LeiS}) and \cite{LeiLZ08}. The initial boundary value problem
was considered in \cite{LL2008}, the compressible case can be found
in \cite{QZ2010, HW2011}. For more results near equilibrium, readers
are referred to the nice review paper by Lin \cite{Lin} and other
works in {\cite{HX2010,ZF} as references. In \cite{Lei14} a
class of large solutions in two space dimensions are established via
the strain-rotation decomposition (which is based on earlier results
in \cite{Lei10} and \cite{Lei3}). In all of these works, the
initial data was restricted by the viscosity parameter. The work
\cite{Kessenich} was the first to establish global existence for 3-D
viscoelastic materials uniformly in the viscosity parameter.
The more difficult two dimensional global vanishing viscosity limit problem
was solved by Cai etc. \cite{CLLM}.

\subsection{Notations, theorem and key ideas}
Throughout this paper, we use the notation $\partial$ for space-time
derivatives:
$$\partial = (\partial_t, \partial_1, \partial_2, \partial_3),$$
$\nabla$ for space derivatives:
$$\nabla = (\partial_1, \partial_2, \partial_3),$$
and $\partial_v$ for horizontal derivative
$$\partial_v = (\partial_1, \partial_2).$$
We use $\langle a\rangle$ to denote
$$\langle a\rangle = \sqrt{1 + a^2},$$
and $[a]$ to denote
$$[a] = {\rm the\ biggest\ integer\ which\ is\ no\ more\ than\ |a| }.$$

We will consider the global well-posedness of incompressible elastodynamics \eqref{Elasto-Hook}-\eqref{incom}
near equilibrium.
Let us introduce the displacement variables as follows:
\begin{equation*} 
X(t, y) = y + Y(t, y).
\end{equation*}
Then the Hookean incompressible elastodynamics \eqref{Elasto-Hook} are written as follows
\begin{equation} \label{B-9}
\partial_t^2Y - \Delta Y = -\nabla p
- (\nabla Y)^{ \top}(\partial_t^2 - \Delta )Y, 
\end{equation}
and  the incompressible condition \eqref{incom} becomes
\begin{align}\label{Struc-1}
\nabla\cdot Y
&= -\frac12 [({\textrm{tr}} \nabla Y)^2- \textrm{tr} (\nabla Y)^2]- \det(\nabla Y)\\
&= -\frac12 (\nabla_i Y^i \nabla_j Y^j- \nabla_i Y^j \nabla_j Y^i)- \det(\nabla Y). \nonumber
\end{align}
Here and in what follows, summation convention over repeated indices will always be assumed.
Moreover, we use the following convention
$$(\nabla Y)_{ij} = \frac{\partial Y^i}{\partial y^j}.$$

Now let us introduce the generalized derivatives and various anisotropic generalized energy.
The rotational operator is defined as follows
\begin{align*}
\Omega=(\Omega_1,\Omega_2,\Omega_3)=y\wedge \nabla.
\end{align*}
Note that we have the following decomposition of gradient operator along
radial and rotational direction:
\begin{equation}\label{grad-decomp}
\nabla=\frac{y}{r}\partial_r-\frac{y}{r^2} \wedge\Omega.
\end{equation}
The angular momentum operators are defined as follows
\begin{align*}
\widetilde{\Omega}=\Omega_i I+U_i,\quad i=1,2,3
\end{align*}
with
\begin{align*}
U_1=
\left[
\begin{matrix}
0 &0 &0\\
0 &0 &1\\
0 &-1&0
\end{matrix}
\right],
\quad
U_2=\left[
\begin{matrix}
0 &0 &-1\\
0 &0 &0\\
1 &0 &0
\end{matrix}
\right],
\quad
U_3=
\left[
\begin{matrix}
0 &1 &0\\
-1&0 &0\\
0 &0 &0
\end{matrix}
\right].
\end{align*}
The scaling operator which is defined as follows:
\begin{align*}
S=t\partial_t+r\partial_r,
\end{align*}
where more useful vector field is the following modified scaling operator
\begin{align*}
\widetilde{S}=(S-1).
\end{align*}

For any  vector $Y$ and scalar $p$, we make the following
conventions:
\begin{equation}\nonumber
\begin{cases}
\widetilde{\Omega} Y \triangleq \Omega Y + UY, \quad
\widetilde{\Omega} p
\triangleq \Omega p,\quad \widetilde{\Omega} Y^j = (\widetilde{\Omega} Y)^j,\\[-4mm]\\
\widetilde{S} Y \triangleq (S - 1) Y, \quad \widetilde{S} p
\triangleq S p,\quad \widetilde{S} Y^j = (\widetilde{S} Y)^j.
\end{cases}
\end{equation}
Let $\Gamma$ be any operator of the following set
\begin{equation}\nonumber
\{\partial_t, \partial_1, \partial_2, \partial_3, \widetilde{\Omega}_1, \widetilde{\Omega}_2, \widetilde{\Omega}_3,
\widetilde{S}\}.
\end{equation}
In order to separate the $\partial_3$ derivative with the others, we introduce the following notation conventions.
For any multi-index $\alpha = \{\alpha_1, \alpha_2, \alpha_3,
\alpha_4, \alpha_5, \alpha_6, \alpha_7, \alpha_8 \} \in \mathbb{N}^8$,
let  $\alpha=(h, a)$, $h\in\mathbb{N}$ and $a=(a_1, a_2, a_3, a_4, a_5, a_6, a_7)\in \mathbb{N}^7$.
$Z^a=\partial_t^{a_1} \partial_1^{a_2} \partial_2^{a_3}
\widetilde{\Omega}^{a_4}_1 \widetilde{\Omega}^{a_5}_2 \widetilde{\Omega}^{a_6}_3
\widetilde{S}^{a_7}$. Moreover,
\begin{align*}
\Gamma^\alpha=\partial_3^h Z^a.
\end{align*}
We also use the notation $\beta=(l, b)$, $\gamma=(m, c)$, $\iota=(n, d)$
where $\beta,\ \gamma,\ \iota\in \mathbb{N}^8$,
$l,\ m,\ n\ \in \mathbb{N}$ and $b,\ c,\ d\ \in \mathbb{N}^7$ and
\begin{align*}
\Gamma^\beta=\partial_3^l Z^b,\quad
\Gamma^\gamma=\partial_3^m Z^c,\quad
\Gamma^\iota=\partial_3^n Z^d.
\end{align*}

The anisotropic generalized energy is introduced as follows:
\begin{equation}\nonumber
\mathcal{E}_{\kappa} = \sum_{|\alpha| \leq \kappa-1} \delta^{2(h-\frac 12)}
\| \partial \Gamma^\alpha Y\|_{L^2(\Omega_\delta)}^2.
\end{equation}
In order to close the energy due the difficulty caused by the anisotropy, we need to introduce the
functional energy
\begin{equation}\nonumber
\mathcal{F}_{\kappa} = \sum_{|\alpha| \leq \kappa-2}
\delta^{2(h-\frac 12)} \| \partial_3 \partial \Gamma^\alpha Y\|_{L^2(\Omega_\delta)}^2.
\end{equation}
Inspired by Lei \cite{Lei16}, we define the weighted $L^2$ generalized energy  as follows
\begin{equation*} 
\mathcal{X}_{\kappa} =  \sum_{|\alpha| \leq \kappa - 2}
\delta^{2(h-\frac 12)} \Big(\int_{\Omega_\delta} ( t - r)^2 |\partial^2\Gamma^\alpha
Y|^2dy + \int_{r > 2t}t^2
|\partial^2\Gamma^\alpha Y|^2dy\Big).
\end{equation*}
To describe the space of the initial data, we introduce
\begin{equation}\nonumber
\Lambda = \{\nabla, r\partial_r, \Omega\}.
\end{equation}
For $\Lambda^\alpha$ with $\alpha\in \mathbb{N}^7$.
Let $\alpha=(h, a)$, $h\in\mathbb{N}$ and $a=(a_1, a_2, a_3, a_4, a_5, a_6)\in \mathbb{N}^6$,
$ A^a=\partial_1^{a_1} \partial_2^{a_2}
\widetilde{\Omega}^{a_3}_1 \widetilde{\Omega}^{a_4}_2 \widetilde{\Omega}^{a_5}_3
(r\partial_r)^{a_6}$. Moreover,
\begin{align*}
\Lambda^\alpha=\partial_3^h A^a.
\end{align*}

Now we are ready to state the main result in this paper as follows:
\begin{thm}\label{GlobalW}
Let $W(\nabla X) = \frac{1}{2}|\nabla X|^2$ be an isotropic Hookean strain energy
function.
Let $\kappa \ge
15$, $M_0 > 0$ and $0<\mu<\frac 18$ be two given
constants.
Suppose that $Y_0$ satisfy the structural constraint
\eqref{Struc-1} and $(Y_0, v_0)$ satisfies
\begin{eqnarray}\nonumber
&&\mathcal{E}_{\kappa}^{\frac{1}{2}}(0) = \sum_{|\alpha| \leq \kappa
- 1}\big(
  \delta^{h-\frac12}\|\nabla\Lambda^{\alpha}Y_0\|_{L^2} +
  \delta^{h-\frac12}\|\Lambda^{\alpha}v_0\|_{L^2}\big) \leq
  M_0,\\\nonumber
&&\mathcal{E}_{\kappa - 3}^{\frac{1}{2}}(0) = \sum_{|\alpha| \leq
\kappa - 4}
  \big(\delta^{h-\frac12}\|\nabla\Lambda^{\alpha}Y_0\|_{L^2} +
  \delta^{h-\frac12}\|\Lambda^{\alpha}v_0\|_{L^2}\big) \leq
  \epsilon,\\\nonumber
&&\mathcal{F}_{\kappa - 3}^{\frac{1}{2}}(0) = \sum_{|\alpha| \leq
\kappa - 5}
  \big(\delta^{h-\frac12}\|\partial_3\nabla\Lambda^{\alpha}Y_0\|_{L^2} +
  \delta^{h-\frac12}\|\partial_3\Lambda^{\alpha}v_0\|_{L^2}\big) \leq
  \epsilon.
\end{eqnarray}
There exists a positive constant $\epsilon_0 < e^{- M_0}$ which
depends only on $\kappa$, $M_0$ and $\mu$, such that, if $\epsilon
\leq \epsilon_0$, then the system of incompressible  isotropic Hookean
elastodynamics \eqref{B-9}-\eqref{Struc-1} with the following initial data
$$Y(0, y) = Y_0(y),\quad \partial_tY(0, y) = v_0(y)$$ has a unique
global classical solution  in the thin domain $\Omega_\delta=\mathbb{R}^2\times [0,\delta]$  with periodic boundary condition, and
\begin{eqnarray}\nonumber
\mathcal{E}_{\kappa - 3}^{\frac{1}{2}}(t)
+\mathcal{F}_{\kappa - 3}^{\frac{1}{2}}(t) \leq \epsilon
\exp\big\{C_0^2M_0\big\},\quad \mathcal{E}_{\kappa}^{\frac{1}{2}}(t)
\leq C_0M_0(1 + t)^{\mu}
\end{eqnarray}
for some $C_0 > 1$ uniformly in $t$ and $\delta$.
\end{thm}

There are mainly three ingredients in proving Theorem \ref{GlobalW}.
Firstly, we establish a series of anisotropic weighted Sobolev type inequalities using generalized derivative in the thin domain.
Those inequalities are dimensionless in terms of the thickness parameter $\delta$. According to the scaling in $x_3$ direction in the thin domain, when taking $L^2$ norm in $x_3$ variable, a $\delta^{-\frac12}$ should be applied, when taking a $\partial_3$ derivative, a $\delta$ should be applied.
The definition of generalized energy $\mathcal{E}_{\kappa}$ also follows this rule except the generalized derivative.
Secondly, we can not directly adapt the Klainerman-Sobolev inequality in the thin domain due to the absence of Lorentz invariance.
This is compensated by introducing the anisotropic weighted $L^2$ norm $\mathcal{X}_{\kappa}$.
Thirdly, due to the quasilinearity, in the anisotropic generalized energy estimate, the pressure
gradient brings an extra singular $\delta^{-1}$ factor arising from the $\partial_3^2$ derivative.
To overcome this difficulty, we introduce an auxiliary anisotropic generalized energy $\mathcal{F}_{\kappa}$, which always contains an extra $\partial_3$ derivative but apply no extra $\delta$ factor. The estimate of $\mathcal{F}_{\kappa}$ prompts us to control the pressure with twice traditional derivative (see Section \ref{EPWTD}).
Our observation is that the inherent cancellation structure gives us some freedom to move one derivative.
Thus the generalized energy estimate can be closed. 

Now let us discuss some other strategies related in this paper. Inspired by \cite{Lei16},
for the lower order energy estimate, instead of estimating the unknowns,
we turn to estimate its curl-free part since divergence-free part is quadratic nonlinearity and thus negligible.
This gives us convenience by deleting the pressure due to the application of the curl operator.
On the other hand, the incompressible elastodynamics in three dimensional thin domain is essentially two dimensional that the unknowns exhibit $\langle t\rangle^{-\frac12}$ supercritical decay in time. Hence we need to use the inherent strong null structure and the ghost weight energy to enhance the temporal decay.

This paper is organized as follows: In Section
\ref{Equations}, we will present the basic
properties of the system, introduce various notations for vector fields and anisotropic energy,
and also outline the proof of Theorem \ref{GlobalW}. In Section \ref{ASI}, we will prove some
weighted anisotropic Sobolev inequalities and weighted $L^\infty$ estimate.
In Section \ref{EEITD}, we will study the elliptic estimates in the thin domain.
In Section \ref{ESQ}, we give the estimate for good derivatives and lay down a preliminary
step for estimating weighted anisotropic $L^2$ energy. Then we will
explore the estimate for strong null form in Section \ref{WE}, and
at the end of that section we give the estimate for the weighted
$L^2$ energy.
Section \ref{EPWTD} is devoted to the second derivative estimate of pressure.
In Section \ref{HEE} we present the highest order
generalized energy estimate. Then we perform the lower order
generalized energy estimate in Section \ref{LEE}.

\section{Commutation and energy estimate ansatz}\label{Equations}

In this section, we begin by applying various vector fields onto the system.
Though we are considering the elastodynamics in thin domain,
the commutation properties of the vector fields are exactly the same to the case in the whole space.
We refer to Sideris \cite{Sideris00} for the case in $\mathbb{R}^3$ and Lei \cite{Lei16} for the case in $\mathbb{R}^2$.
Here, we will give a
sketch of the application of the vector fields and indicate the differences with the classical theory.

We first apply the rotational operators.
Through direct calculation, one has
\begin{equation}\label{2-2}
\begin{cases}
(\partial_t^2-\Delta) \widetilde{\Omega} Y
= -(\nabla\widetilde{\Omega} Y)^{\top}(\partial_t^2-\Delta) Y
 -(\nabla Y)^{\top}(\partial_t^2-\Delta) \widetilde{\Omega} Y
 - \nabla \widetilde{\Omega} p, \\[-4mm]\\
\nabla\cdot \widetilde{\Omega} Y
 = -\frac12 \big(\partial_i\widetilde{\Omega} Y^i \partial_j Y^j -
  \partial_iY^j\partial_j\widetilde{\Omega} Y^i\big)
  -\frac12 \big(\partial_i Y^i \partial_j\widetilde{\Omega} Y^j -
  \partial_i \widetilde{\Omega} Y^j\partial_j Y^i\big) \\[-4mm]\\
\qquad\qquad\ - \left|
\begin{matrix}
 \partial_1 \widetilde{\Omega} Y^1 &\partial_2 \widetilde{\Omega} Y^1 & \partial_3 \widetilde{\Omega} Y^1\\
 \partial_1 Y^2& \partial_2  Y^2& \partial_3 Y^2\\
\partial_1 Y^3&\partial_2 Y^3&\partial_3Y^3
\end{matrix} \right|
- \left|
\begin{matrix}
 \partial_1 Y^1 &\partial_2  Y^1 & \partial_3  Y^1\\
 \partial_1 \widetilde{\Omega} Y^2& \partial_2 \widetilde{\Omega} Y^2& \partial_3 \widetilde{\Omega} Y^2\\
\partial_1 Y^3&\partial_2 Y^3&\partial_3Y^3
\end{matrix} \right|
\\[-4mm]\\
\qquad\qquad\ - \left|
\begin{matrix}
 \partial_1 Y^1 &\partial_2  Y^1 & \partial_3  Y^1\\
 \partial_1 Y^2& \partial_2  Y^2& \partial_3 Y^2\\
\partial_1 \widetilde{\Omega} Y^3&\partial_2 \widetilde{\Omega} Y^3&\partial_3\widetilde{\Omega} Y^3
\end{matrix} \right|.
\end{cases}
\end{equation}

Next, we apply the scaling operator. 
Similarly, through direct calculation, we have
\begin{equation}\label{2-3}
\begin{cases}
(\partial_t^2 - \Delta)\widetilde{S}Y
=-(\nabla\widetilde{S}Y)^{\top}(\partial_t^2 - \Delta) Y
-(\nabla Y)^{\top}(\partial_t^2 - \Delta)\widetilde{S} Y
- \nabla Sp,\\[-4mm]\\
\nabla\cdot \widetilde{S} Y= -\frac12 \big(\partial_i\widetilde{S} Y^i \partial_j Y^j -
  \partial_iY^j\partial_j\widetilde{S} Y^i\big)
  -\frac12 \big(\partial_i Y^i \partial_j\widetilde{S} Y^j -
  \partial_i \widetilde{S} Y^j\partial_j Y^i\big) \\[-4mm]\\
 \qquad\qquad\ - \left|
\begin{matrix}
 \partial_1 \widetilde{S} Y^1 &\partial_2 \widetilde{S} Y^1 & \partial_3 \widetilde{S} Y^1\\
 \partial_1 Y^2& \partial_2  Y^2& \partial_3 Y^2\\
\partial_1 Y^3&\partial_2 Y^3&\partial_3Y^3
\end{matrix} \right|
- \left|
\begin{matrix}
 \partial_1 Y^1 &\partial_2  Y^1 & \partial_3  Y^1\\
 \partial_1 \widetilde{S} Y^2& \partial_2 \widetilde{S} Y^2& \partial_3 \widetilde{S} Y^2\\
\partial_1 Y^3&\partial_2 Y^3&\partial_3Y^3
\end{matrix} \right|
\\[-4mm]\\
\qquad\qquad\ - \left|
\begin{matrix}
 \partial_1 Y^1 &\partial_2  Y^1 & \partial_3  Y^1\\
 \partial_1 Y^2& \partial_2  Y^2& \partial_3 Y^2\\
\partial_1 \widetilde{S} Y^3&\partial_2 \widetilde{S} Y^3&\partial_3\widetilde{S} Y^3
\end{matrix} \right|.
\end{cases}
\end{equation}
Using \eqref{2-2} and
\eqref{2-3} inductively, we have
\begin{eqnarray}\label{Elasticity-D}
(\partial_t^2 - \Delta) \Gamma^\alpha Y
 = -\nabla\Gamma^\alpha p-\ \sum_{\beta + \gamma
 = \alpha} C_{\alpha}^\beta(\nabla\Gamma^\beta Y)^{\top}(\partial_t^2 - \Delta)\Gamma^\gamma Y,
\end{eqnarray}
together with the constraint
\begin{align}\label{Struc-2}
\nabla\cdot \Gamma^\alpha Y
 =& -\frac12 \sum_{\beta + \gamma =\alpha} C_{\alpha}^\beta
\big(\partial_i\Gamma^\beta Y^i \partial_j\Gamma^\gamma Y^j -
  \partial_i\Gamma^\gamma Y^j\partial_j\Gamma^\beta Y^i\big)\\\nonumber
  &-\sum_{\beta+\gamma+\iota=\alpha}C_\alpha  ^\beta C_{\alpha-\beta}  ^\gamma \left|
\begin{matrix}
 \partial_1\Gamma^\beta Y^1 &\partial_2 \Gamma^\beta Y^1 & \partial_3 \Gamma^\beta Y^1\\
 \partial_1\Gamma ^\gamma Y^2& \partial_2 \Gamma ^\gamma Y^2& \partial_3\Gamma ^\gamma Y^2\\
\partial_1\Gamma ^\iota Y^3&\partial_2\Gamma^\iota Y^3&\partial_3\Gamma^\iota Y^3
\end{matrix} \right|.
\end{align}
The expression \eqref{Struc-2} keeps the cancellation structure in \eqref{Struc-1}.
This  will be of importance for the future argument.

Now we present the main step of the proof of the Theorem \ref{GlobalW}  as follows: for initial data
satisfying the constraints in Theorem \ref{GlobalW}, we will prove
that
\begin{eqnarray}\nonumber
\mathcal{E}_{\kappa}^\prime(t) \leq \frac{C_0}{4}\langle
t\rangle^{-1}\mathcal{E}_{\kappa}(t)
\big(\mathcal{E}_{\kappa-3}^{\frac{1}{2}}(t)+\mathcal{F}_{\kappa-3}^{\frac{1}{2}}(t)
\big),
\end{eqnarray}
which is given in \eqref{ES-2}
 and
\begin{eqnarray}\nonumber
\mathcal{E}_{\kappa - 3}^\prime(t) \leq \frac{C_0}{4}\langle
t\rangle^{-\frac{3}{2}} \mathcal{E}^{\frac12}_{\kappa -3}(t)
\mathcal{E}^{\frac12}_{\kappa}(t)
\big(\mathcal{E}_{\kappa-3}^{\frac{1}{2}}(t)+\mathcal{F}_{\kappa-3}^{\frac{1}{2}}(t)\big)
\end{eqnarray}
which is given in \eqref{LOEE-1} and
\begin{eqnarray}\nonumber
\mathcal{F}_{\kappa - 3}^\prime(t) \leq \frac{C_0}{4}\langle
t\rangle^{-\frac{3}{2}} \mathcal{F}^{\frac{1}{2}}_{\kappa -3}(t)
\mathcal{E}^{\frac12}_{\kappa}(t)
\big(\mathcal{E}_{\kappa-3}^{\frac{1}{2}}(t)+\mathcal{F}_{\kappa-3}^{\frac{1}{2}}(t)\big)
\end{eqnarray}
which is given in \eqref{B-3},
 for some absolute
positive constant $C_0$ depending only on $\kappa$.
Thus we can combine to conclude the following two differential inequalities:
\begin{eqnarray}\nonumber
\mathcal{E}_{\kappa}^\prime(t) \leq \frac{C_0}{4}\langle
t\rangle^{-1}\mathcal{E}_{\kappa}(t)
\big(\mathcal{E}_{\kappa-3}(t)+\mathcal{F}_{\kappa-3}(t)
\big)^{\frac{1}{2}}
\end{eqnarray}
and
\begin{eqnarray}\nonumber
\big(\mathcal{E}_{\kappa - 3}+\mathcal{F}_{\kappa - 3}\big)^\prime(t) \leq
\frac{C_0}{4}\langle t\rangle^{-\frac{3}{2}}
\mathcal{E}^{\frac12}_{\kappa}(t)
\big(\mathcal{E}_{\kappa-3}(t)+\mathcal{F}_{\kappa-3}(t)\big).
\end{eqnarray}
Then it is easy to show that the
bounds for $\mathcal{E}_{\kappa - 3}^{\frac{1}{2}}$ and
$\mathcal{E}_{\kappa}^{\frac{1}{2}}$ given in the theorem can never
be reached, which will yield the global existence result and
complete the proof of Theorem \ref{GlobalW}.
For the details, we refer to \cite{Lei16}.

So from now on our main goal is to show the three \textit{a
priori} differential inequalities mentioned above. The highest order
one will be done in Section \ref{HEE} and the lower order one will
be done in Section \ref{LEE}. Since we can make $C_0$ arbitrarily
large and $\mu$ arbitrarily small, we will always assume that
$\mathcal{E}_{\kappa - 3}^{\frac{1}{2}} \ll 1$ and
$\mathcal{F}_{\kappa - 3}^{\frac{1}{2}} \ll 1$ in the remaining of
this paper.

Finally, we clarify that the notation $X\lesssim Y$ is to indicate $X\leq CY$ for some constant $C> 0$.
Without specification, the constant may depends on $\kappa$, but not depend on time $t$ and $\delta$.
\section{Weighted anisotropic Sobolev inequalities}\label{ASI}

We first study the weighted anisotropic Sobolev type inequalities in thin domain.
\begin{lem}\label{Sob1}
For all $g(x_3)\in H^1(\mathbb{T}^\delta)$
with $\mathbb{T}^\delta=[0,\delta]$, there holds
\begin{align*}
\| g \|_{L^\infty(\mathbb{T}^\delta)}\lesssim
\delta^{-\frac12}\| g \|_{L^2(\mathbb{T}^\delta)}
+\delta^{\frac12}\| \partial_3 g \|_{L^2(\mathbb{T}^\delta)},
\end{align*}
provided the right hand side is finite.

\end{lem}
\begin{proof}
Define
\begin{align*}
\tilde{g}(x_3)=g(\delta x_3)\ \textrm{in}\ [0,1].
\end{align*}
By Sobolev imbedding, there holds
\begin{align*}
\| \tilde{g} \|_{L^\infty([0,1])} \lesssim
\| \tilde{g} \|_{L^2([0,1])}+ \| \partial_3\tilde{g} \|_{L^2([0,1])}.
\end{align*}
Consequently,
\begin{align*}
\| g \|_{L^\infty(\mathbb{T}^\delta)}
=\| \tilde{g} \|_{L^\infty([0,1])}
&\lesssim
\| \tilde{g} \|_{L^2([0,1])}+ \| \partial_3\tilde{g} \|_{L^2([0,1])}\\
&= \delta^{-\frac 12} \|g \|_{L^2(\mathbb{T}^\delta)}
+ \delta^{\frac 12}\| \partial_3 g \|_{L^2(\mathbb{T}^\delta)} .
\end{align*}
This gives the desired result of this lemma.
\end{proof}

\begin{lem}\label{Sob2}
For all $f(x)\in H^3(\Omega_\delta)$
with $\Omega_\delta=\mathbb{R}^2\times \mathbb{T}^\delta$, there holds
\begin{align*}
\| f \|_{L^\infty(\Omega_\delta)}\lesssim
\sum_{|a|\leq 2}\delta^{-\frac12}\| \partial_v^a f \|_{L^2(\Omega_\delta)}
+\sum_{|a|\leq 2}\delta^{\frac12}\| \partial_v^a \partial_3f \|_{L^2(\Omega_\delta)},
\end{align*}
provided the right hand side is finite.
\end{lem}
\begin{proof}
Define $\tilde{f}=f(x_v,\delta x_3)$ in
$\Omega=\mathbb{R}^2\times [0,1], x_v=(x_1, x_2)$. By the Sobolev imbedding in $\mathbb{R}^2$, there holds
\begin{align*}
\tilde{f}(x_v,x_3) \lesssim \sum_{|a|\leq 2}\| \partial_v^a \tilde{f}(\cdot, x_3)\|_{L^2(\mathbb{R}^2)}.
\end{align*}
Then combined with Lemma \ref{Sob1} yields the desired inequality.
\end{proof}

Next we study the weighted Sobolev-type inequalities in thin domain.
They will be used to prove the decay of solutions in $L^\infty$ norm.
\begin{lem}\label{KS-A}
For all $f\in H^3(\Omega_\delta)$, there hold
\begin{align}
\label{K-S-1}
r_v^{\frac12} |f(x)|\lesssim
&\delta^{-\frac12}\sum_{a=0,1} \big( \|\partial_v \Omega^a_3 f\|_{L^2(\Omega_\delta)}
+\|\Omega^a_3 f\|_{L^2(\Omega_\delta)} \big)  \nonumber\\
&+\delta^{\frac12}\sum_{a=0,1} \big( \|\partial_3\partial_v \Omega^a_3 f\|_{L^2(\Omega_\delta)}+\|\partial_3\Omega^a_3 f\|_{L^2(\Omega_\delta)} \big)
,\\
\label{K-S-2}
 r_v^{\frac12} |(t-r)f(x)|&\lesssim
\delta^{-\frac12}\sum_{a=0,1}\big(\|(t-r)\partial_v\Omega^a_3f\|_{L^2(\Omega_\delta)}
+\|(t-r) \Omega^a_3f\|_{L^2(\Omega_\delta)} \big) \nonumber \\
&\quad+\delta^{\frac12}\sum_{a=0,1} \big( \|( t-r) \partial_3\partial_v\Omega^a_3f\|_{L^2(\Omega_\delta)}
+\|(t-r) \partial_3\Omega^a_3f \|_{L^2(\Omega_\delta)} \big),\\
\label{K-S-4}
 r_v^{\frac12} |t-r|^{\frac12}|f(x)|&\lesssim
\delta^{-\frac12}\sum_{a=0,1}\big(\|(t-r)\partial_v\Omega^a_3f\|_{L^2(\Omega_\delta)}
+\|\Omega^a_3f\|_{L^2(\Omega_\delta)} \big) \nonumber \\
&\quad+\delta^{\frac12}\sum_{a=0,1} \big( \|( t-r) \partial_3\partial_v\Omega^a_3f\|_{L^2(\Omega_\delta)}
+\|\partial_3\Omega^a_3f \|_{L^2(\Omega_\delta)} \big),\\
\label{K-S-3}
t \| f\|_{L^\infty(r\leq t/2)}
&\lesssim \delta^{-\frac12}\sum_{|a|\le2}\| (t-r) \partial^a_v f\|_{L^2(\Omega_\delta)}
+\delta^{\frac12}\sum_{|a|\leq 2}\| (t-r) \partial_3\partial^a_v f\|_{L^2(\Omega_\delta)},
\end{align}
provided the right-hand side is finite.
\end{lem}
\begin{remark}
It can be roughly seen from Lemma \ref{KS-A} that the decay
of solution in $L^\infty$ norm is $t^{-\frac12}$
assuming the weighted $L^2$ norm $\mathcal{X}_{\kappa}$ is bounded.
This suggests the problem is essentially two dimensional case.
\end{remark}

\begin{remark}
On the right hand side of the above inequalities,
only the traditional derivative and rotational derivative $\Omega_3=x_1\partial_2-x_2\partial_1$ are involved, we do not use $\Omega_1$ and $\Omega_2$.
This also corresponds to the fact that the decay comes from the dispersion of waves along $x_v=(x_1, x_2)$ direction.
\end{remark}

For the three-dimensional isotropic version of the weighted Sobolev type inequalities,
we refer to \cite{Sideris00, ST07}.
For the two-dimensional isotropic version,
we refer to \cite{LSZ13, Lei16}.
\begin{proof}
We only need to prove this lemma for $f \in
C_0^2(\mathbb{R}^2)$. For the general case,
one can use a completion argument.
Let
\begin{align*}
x_v=(x_1, x_2),\
r_v=(x_1^2 + x_2^2)^{\frac12},\
\omega_v=\frac{x_v}{r_v}.
\end{align*}
Employing Sobolev imbedding 
$H^1(\mathbb{S}^1)\hookrightarrow L^\infty(\mathbb{S}^1)$
in horizontal direction,
we write
\begin{align*}
r_v |f(r_v\omega_v, x_3)|^2
&\lesssim \sum_{a=0,1}\int_{\mathbb{S}^1} r_v |\Omega^a_3 f(r_v \omega_v, x_3)|^2d\sigma_v\\
&=- \sum_{a=0,1}\int_{\mathbb{S}^1}d\sigma_v
\int_{r_v}^\infty r_v \partial_{\rho_v} [|\Omega^a_3 f(\rho_v \omega_v, x_3)|^2]d \rho_v\\
&\lesssim \sum_{a=0,1}\int_{\mathbb{S}^1}
\int_{r_v}^\infty|\Omega^a_3f(\rho_v \omega_v, x_3)|
|\partial_{\rho_v}\Omega^a_3 f(\rho_v \omega_v, x_3)| \rho_v d \rho_v d\sigma_v\\
&\lesssim \sum_{a=0,1}\large(\|\partial_v \Omega^a_3 f(\cdot,x_3)\|_{L^2(\mathbb{R}^2)}^2
+\|\Omega^a_3 f(\cdot,x_3)\|_{L^2(\mathbb{R}^2)}^2 \large).
\end{align*}
Combined with Lemma \ref{Sob1} yields \eqref{K-S-1}.

The proof of \eqref{K-S-2} is similar.
By Sobolev imbedding $H^1(\mathbb{S}^1)\hookrightarrow L^\infty(\mathbb{S}^1)$,
one has
\begin{align*}
&r_v (t-r)^2|f(r_v\omega_v,x_3)|^2\lesssim
\sum_{a=0,1}\int_{\mathbb{S}^1}
r_v(t-r)^2
|\Omega^a_3 f(r_v \omega_v, x_3)|^2 d\sigma_v\\
&= -\sum_{a=0,1}\int_{\mathbb{S}^1} d\sigma_v
  \int_{r_v}^\infty
  r_v\partial_{\rho_v}[(t-\rho)^2 |\Omega^a_3 f(\rho_h\omega_v,x_3)|^2] d\rho_v\\
&\lesssim \sum_{a=0,1}\int_{\mathbb{S}^1}\int_{r_v}^\infty
[(t-\rho)^2  |\Omega^a_3 f(\rho_v\omega_v,x_3)||\partial_{\rho_v}
  \Omega^a_3f(\rho_v \omega_v,x_3)|+(t-\rho)
  |\Omega^a_3f(\rho_v\omega_v,x_3)|^2] \rho_v d\rho_v d\sigma\\
&\lesssim \sum_{a=0,1}\int_{\mathbb{S}^1}\int_{r_v}^\infty
[(t-\rho)^2  |\partial_{\rho_v} \Omega^a_3f(\rho_v\omega_v,x_3)|^2
  +(t-\rho)^2 |\Omega^a_3f(\rho_v\omega_v,x_3)|^2] \rho_v d\rho_v d\sigma\\
&\lesssim \sum_{a=0,1}\large(\|(t-r) \partial_v\Omega^a_3f(\cdot,x_3)\|_{L^2(\mathbb{R}^2)}^2
+\|( t-r) \Omega^a_3f(\cdot,x_3)\|_{L^2(\mathbb{R}^2)}^2\large).
\end{align*}
Combined with Lemma \ref{Sob1} yields \eqref{K-S-2}.

The proof of \eqref{K-S-4} is similar to \eqref{K-S-2}.
By Sobolev imbedding $H^1(\mathbb{S}^1)\hookrightarrow L^\infty(\mathbb{S}^1)$,
one has
\begin{align*}
&r_v|t-r||f(r_v\omega_v,x_3)|^2\lesssim
\sum_{a=0,1}\int_{\mathbb{S}^1}
r_v|t-r|
|\Omega^a_3 f(r_v \omega_v, x_3)|^2 d\sigma_v\\
&= -\sum_{a=0,1}\int_{\mathbb{S}^1} d\sigma_v
  \int_{r_v}^\infty
  r_v\partial_{\rho_v}[|t-\rho| |\Omega^a_3 f(\rho_v\omega_v,x_3)|^2] d\rho_v\\
&\lesssim \sum_{a=0,1}\int_{\mathbb{S}^1}\int_{r_v}^\infty
[|t-\rho|  |\Omega^a_3 f(\rho_v\omega_v,x_3)||\partial_{\rho_v}
  \Omega^a_3f(\rho_v \omega_v,x_3)|+
  |\Omega^a_3f(\rho_v\omega_v,x_3)|^2] \rho_v d\rho_v d\sigma\\
&\lesssim \sum_{a=0,1}\int_{\mathbb{S}^1}\int_{r_v}^\infty
[|t-\rho|  |\partial_{\rho_v} \Omega^a_3f(\rho_v\omega_v,x_3)|^2
  + |\Omega^a_3f(\rho_v\omega_v,x_3)|^2] \rho_v d\rho_v d\sigma\\
&\lesssim \sum_{a=0,1}\large(\|(t-r) \partial_v\Omega^a_3f(\cdot,x_3)\|_{L^2(\mathbb{R}^2)}^2
+\|\Omega^a_3f(\cdot,x_3)\|_{L^2(\mathbb{R}^2)}^2\large).
\end{align*}
Combined with Lemma \ref{Sob1} yields \eqref{K-S-4}.

Now we turn our attention to \eqref{K-S-3}.
Let $\eta\in C_0^\infty$, satisfy $\eta(s)=1$ for $s\le 1$, $\eta(s)=0$ for $s\ge3/2$.
Note that one has $t/4 \leq t-r$ on $\mbox{supp}\;\eta(r/( t/2) )$.  Thus, for $r\leq t/2 $, we deduce by Sobolev imbedding $H^2(\mathbb{R}^2)\hookrightarrow L^\infty(\mathbb{R}^2)$ on horizontal direction that
\begin{align*}
t|f(x_v,x_3)|
&=  t \eta(r/ (t/2) )|f(x_v,x_3)|\\[-4mm]\\
&\lesssim \|t\eta(r/ ( t/2 )f(\cdot,x_3)\|_{H^2(\mathbb{R}^2)}\\[-4mm]\\
&\lesssim \sum_{|a|\le2}\|(t-r)\partial^a_v f(\cdot,x_3)\|_{L^2(\mathbb{R}^2)}.
\end{align*}
Further employing Lemma \ref{Sob1} gives \eqref{K-S-3}.
This completes the proof of the lemma.
\end{proof}

Next we state a lemma which concerns the decay of solutions away from the light cone.
Before then, we first state a decay lemma in $\mathbb{R}^2$.
\begin{lem} \label{sob}
Assume $f\in H^2(\mathbb{R}^2)$.
There holds
\begin{equation*}
\|f\|_{L^\infty(\mathbb{R}^2)}
\lesssim \|\partial_v f\|_{L^2(\mathbb{R}^2)}
\ln^{\frac{1}{2}}\big(e+ t\big) + \frac{1}{1 + t}\big(\|f\|_{L^2(\mathbb{R}^2)} +
\|\partial_v^2f\|_{L^2(\mathbb{R}^2)}\big).
\end{equation*}
\end{lem}
This lemma comes from \cite{Lei16} (Lemma 3.1).
\begin{proof}
We first introduce the Littlewood-Paley theory on $\mathbb{R}^2$. Let $\phi$ be a
smooth function supported in $\{\tau \in \mathbb{R}^+: \frac{3}{4}
\leq \tau \leq \frac{8}{3}\}$ such that
\begin{equation}\nonumber
\sum_{j\in \mathbb{Z}}\phi(2^{-j}\tau) = 1.
\end{equation}
For $f \in L^2(\mathbb{R}^2)$, we set
\begin{equation}\nonumber
\Delta_j f = \mathcal{F}^{-1}\big(\phi(2^{-j}|\xi|)\mathcal{F}(f)\big)
\end{equation}
and
\begin{equation}\nonumber
S_{j}f = \sum_{- \infty < k \leq j - 1,\ k \in \mathbb{Z}}\Delta_kf.
\end{equation}
Here $\mathcal{F}$ denotes the usual Fourier transformation on $\mathbb{R}^2$ and
$\mathcal{F}^{-1}$ the inverse Fourier transformation on $\mathbb{R}^2$.

By the following well-known Bernstein inequality on $\mathbb{R}^2$:
\begin{eqnarray}\nonumber
\|\Delta_j f\|_{L^\infty(\mathbb{R}^2)} \lesssim 2^{j}\|\Delta_j f\|_{L^2(\mathbb{R}^2)},\quad
\|S_j f\|_{L^\infty(\mathbb{R}^2)} \lesssim 2^{j}\|S_j f\|_{L^2(\mathbb{R}^2)},
\end{eqnarray}
one has
\begin{eqnarray}\nonumber
\|f\|_{L^\infty(\mathbb{R}^2)} &=
& \|\sum_j\Delta_j f\|_{L^\infty(\mathbb{R}^2)}\\\nonumber &\lesssim&
  2^{- N}\|S_{- N}f\|_{L^2(\mathbb{R}^2)} + \sum_{- N \leq j \leq
  N }2^j\|\Delta_jf\|_{L^2(\mathbb{R}^2)} + \sum_{j \geq N +
  1}2^j\|\Delta_jf\|_{L^2(\mathbb{R}^2)}\\\nonumber
&\lesssim& 2^{- N}\|f\|_{L^2(\mathbb{R}^2)} + \sqrt{2N}\|\partial_v f\|_{L^2(\mathbb{R}^2)}
+ 2^{- N}\|\partial_v^2f\|_{L^2(\mathbb{R}^2)}.
\end{eqnarray}
Choosing $N = \frac{\ln(e + t)}{\ln 2}$, one has
\begin{equation*}
\|f\|_{L^\infty(\mathbb{R}^2)} \lesssim \|\partial_v f\|_{L^2(\mathbb{R}^2)}\ln^{\frac{1}{2}}\big(e
+ t\big) + \frac{1}{1 + t}\big(\|f\|_{L^2(\mathbb{R}^2)} +
\|\partial_v^2f\|_{L^2(\mathbb{R}^2)}\big).
\end{equation*}
This finishes the proof of the lemma.
\end{proof}

\begin{lem}\label{K-S-O}
Let $t\geq 1$, $f\in H^4(\Omega_\delta)$.
There holds
\begin{align*}
&t \ln^{-\frac12}(e+t)\|  \nabla f\|_{L^\infty(r\leq t/2)} \\[-4mm]\\\nonumber
&\lesssim
\delta^{-\frac12} \|(t-r)\nabla^2 f\|_{L^2(\Omega_\delta)}
+\delta^{\frac12} \|(t-r)\partial_3\partial_v \nabla f\|_{L^2(\Omega_\delta)} \\[-4mm]\\\nonumber
&\quad +\delta^{-\frac12} \sum_{|a|\leq 2}\|\nabla Z^a f\|_{L^2(\Omega_\delta)}
+\delta^{\frac12} \sum_{|a|\leq 2}\|\nabla^2 Z^a f\|_{L^2(\Omega_\delta)}.
\end{align*}
\end{lem}
\begin{proof}
By Lemma \ref{Sob1} and  Lemma \ref{sob}, there holds
\begin{align} \label{so}
\|f\|_{L^\infty(\Omega_\delta)}
\lesssim
& \ln^{\frac{1}{2}}(e+ t)\big(\delta^{-\frac12}\|\partial_v f\|_{L^2(\Omega_\delta)}
+\delta^{\frac12} \|\partial_3\partial_v f\|_{L^2(\Omega_\delta)})
\\\nonumber
&+ \frac{1}{1 + t}\delta^{-\frac12}(\|f\|_{L^2(\Omega_\delta)} +\|\partial^2_v f\|_{L^2(\Omega_\delta)}) \\\nonumber
&+\frac{1}{1 + t}\delta^{\frac12}(\| \partial_3f\|_{L^2(\Omega_\delta)} +\|\partial_3\partial^2_vf\|_{L^2(\Omega_\delta)}).
\end{align}
Next let us choose a radial cutoff function $\varphi \in
C_0^\infty(\Omega_\delta)$ which satisfies
\begin{equation}\nonumber
\varphi = \begin{cases}1,\quad {\rm if}\ r \geq\frac{3}{4} \\
0,\quad {\rm if}\ r \leq \frac{2}{3}\end{cases},\quad |\nabla\varphi| \leq 100.
\end{equation}
For each fixed $t \geq 1$, let $\varphi^t(y) = \varphi(y/t)$. Clearly, one has
$$\varphi^t(y) \equiv 1\ \ {\rm for}\ r\geq\frac{3t}{4}
,\quad \varphi^t(y) \equiv 0\ \ {\rm
for}\ r \leq \frac{2t}{3}$$ and $$|\nabla\varphi^t(y)| \leq
100\ t^{-1} \lesssim \langle t\rangle^{-1}.$$
Then applying \eqref{so} to $(1-\varphi^t) f$ immediately yields the lemma.
\end{proof}

\section{Elliptic estimate in thin domain}\label{EEITD}
Next, we study the elliptic estimates in the thin domain.
\begin{lem}\label{OP1}
For $p(x)$ defined on the $\Omega_\delta$ satisfying
\begin{align*}
\Delta p=\partial_v f,
\end{align*}
with periodic boundary condition. Then
there holds that
\begin{align*}
&\|\partial_v p \|_{L^2(\Omega_\delta)} \leq \| f \|_{L^2(\Omega_\delta)}, \\
&\|\partial_3 p \|_{L^2(\Omega_\delta)} \leq \| f \|_{L^2(\Omega_\delta)}.
\end{align*}
\end{lem}
\begin{proof}
Define
\begin{align*}
\tilde{p}(x_v, x_3)=p(x_v,\delta x_3),\quad
\tilde{f}(x_v, x_3)=f(x_v,\delta x_3) \quad
\mathrm{in}\quad \Omega_1=\mathbb{R}^2\times \mathbb{T}.
\end{align*}
Then
\begin{align*}
\partial_3 \tilde{p}=\delta \partial_3 p.
\end{align*}
Let $\Delta_\delta=\partial_1^2+\partial_2^2 +\delta^{-2} \partial_3^2$.
Then
\begin{align*}
\Delta_\delta \tilde{p}= \partial_v \tilde{f}.
\end{align*}
Note that
\begin{align*}
\|\tilde{p}\|_{L^2(\Omega_1)}= \delta^{-\frac12}  \|p\|_{L^2(\Omega_\delta)},
\quad \|\partial_3\tilde{p}\|_{L^2(\Omega_1)}= \delta^{\frac12}  \|\partial_3p\|_{L^2(\Omega_\delta)}.
\end{align*}
Solving pressure gives
\begin{align*}
\partial_v \tilde{p}= \partial_v \Delta_\delta^{-1} \partial_v \tilde{f},\quad
\partial_3 \tilde{p}= \partial_3 \Delta_\delta^{-1} \partial_v \tilde{f}.
\end{align*}
Note the multipliers are $\frac{\xi_v^2}{\xi_v^2+\delta^{-2}\xi_3^2}$
and $\frac{\xi_v \xi_3}{\xi_v^2+\delta^{-2}\xi_3^2}=
\delta \frac{\xi_v\cdot \delta^{-1}\xi_3}{\xi_v^2+\delta^{-2}\xi_3^2}$, respectively, where $\xi_v=(\xi_1, \xi_2)\in \mathbb{R}^2, \xi_3\in\mathbb{Z}\backslash\{0\} $. Hence
\begin{align*}
\|\partial_v\tilde{p} \|_{L^2(\Omega_1)}\leq \|\tilde{f}\|_{L^2(\Omega_1)},\quad
\|\partial_3\tilde{p} \|_{L^2(\Omega_1)}\leq \delta \|\tilde{f}\|_{L^2(\Omega_1)}.
\end{align*}
This gives that
\begin{align*}
&\| \partial_v p \|_{L^2(\Omega_\delta)}=\delta^{\frac12} \| \partial_v\tilde{p} \|_{L^2(\Omega_1)}
\leq  \delta^{\frac12}  \| \tilde{f} \|_{L^2(\Omega_1)}= \| f \|_{L^2(\Omega_\delta)},\\
\quad
&\| \partial_3p \|_{L^2(\Omega_\delta)}=\delta^{-\frac12} \| \partial_3\tilde{p} \|_{L^2(\Omega_1)}
\leq  \delta^{\frac12}  \| \tilde{f} \|_{L^2(\Omega_1)}= \| f \|_{L^2(\Omega_\delta)}.
\end{align*}
\end{proof}
\begin{lem}\label{OP2}
For $p(x)$ defined on the $\Omega_\delta$ satisfying
\begin{align*}
\ \Delta p=\partial_3 g,
\end{align*}
with periodic boundary condition. Then
there holds that
\begin{align*}
&\|\partial_v p \|_{L^2(\Omega_\delta)} \leq  \| g \|_{L^2(\Omega_\delta)}, \\
&\|\partial_3 p \|_{L^2(\Omega_\delta)} \leq  \| g \|_{L^2(\Omega_\delta)} .
\end{align*}
\end{lem}
\begin{proof}
Define
\begin{align*}
\tilde{p}(x_v, x_3)=p(x_v,\delta x_3),\quad
\tilde{g}(x_v, x_3)=g(x_v,\delta x_3) \quad \mathrm{in}\quad \mathbb{R}^2\times \mathbb{T}.
\end{align*}
Then
\begin{align*}
\partial_3 \tilde{p}=\delta \partial_3 p.
\end{align*}
Let $\Delta_\delta=\partial_1^2+\partial_2^2 +\delta^{-2} \partial_3^2$.
Then
\begin{align*}
\Delta_\delta \tilde{p}= \partial_v \tilde{f},\quad
\Delta_\delta \tilde{p}= \delta^{-1}\partial_3 \tilde{g}.
\end{align*}
Note that
\begin{align*}
\|\tilde{p}\|_{L^2(\Omega_1)}= \delta^{-\frac12}  \|p\|_{L^2(\Omega_\delta)},
\quad \|\partial_3\tilde{p}\|_{L^2(\Omega_1)}= \delta^{\frac12}  \|\partial_3p\|_{L^2(\Omega_\delta)}.
\end{align*}
Solving pressure gives
\begin{align*}
\partial_v \tilde{p}= \delta^{-1} \partial_v \Delta_\delta^{-1} \partial_3 \tilde{g},\quad
\partial_3 \tilde{p}= \delta^{-1} \partial_3 \Delta_\delta^{-1} \partial_3 \tilde{g}.
\end{align*}
Note the multipliers are $\frac{\xi_v \cdot \delta^{-1}\xi_3}{\xi_v^2+\delta^{-2}\xi_3^2}$
and $\frac{\delta^{-1} \xi_3^2}{\xi_v^2+\delta^{-2}\xi_3^2}=
\delta \frac{\delta^{-2} \xi_3^2}{\xi_v^2+\delta^{-2}\xi_3^2}$, respectively, hence
\begin{align*}
\|\partial_v\tilde{p} \|_{L^2(\Omega_1)}\leq \|\tilde{g}\|_{L^2(\Omega_1)},\quad
\|\partial_3\tilde{p} \|_{L^2(\Omega_1)}\leq \delta \|\tilde{g}\|_{L^2(\Omega_1)}.
\end{align*}
This gives that
\begin{align*}
&\| \partial_v p \|_{L^2(\Omega_\delta)}=\delta^{\frac12} \| \partial_v\tilde{p} \|_{L^2(\Omega_1)}
\leq  \delta^{\frac12}  \| \tilde{g} \|_{L^2(\Omega_1)}= \| g \|_{L^2(\Omega_\delta)},\\
\quad
&\| \partial_3p \|_{L^2(\Omega_\delta)}=\delta^{-\frac12} \| \partial_3\tilde{p} \|_{L^2(\Omega_1)}
\leq  \delta^{\frac12}  \| \tilde{g} \|_{L^2(\Omega_1)}= \| g \|_{L^2(\Omega_\delta)}.
\end{align*}
\end{proof}

\section{Estimate of the special quantity}\label{ESQ}
In this section, we give a preliminary estimate for the weighted $L^2$
generalized energy norm $\mathcal{X}_{\kappa}$.
\begin{lem}\label{WE-1}
Let $Y$ be defined on $\Omega_\delta=\mathbb{R}^2\times [0,\delta]$  satisfying periodic boundary condition. Then
there holds
\begin{equation}\label{WW1}
\mathcal{X}_2^{\frac 12}
\lesssim  \delta^{-\frac12}\Big(\|\partial Y \|_{L^2(\Omega_\delta)}+\|\partial Z Y \|_{L^2(\Omega_\delta)}
  + t \|(\partial_t^2 - \Delta)Y\|_{L^2(\Omega_\delta)}\Big).
\end{equation}
\end{lem}
\begin{rem}
On the right hand side of \eqref{WW1},
we use the generalized derivative $Z$
involving no $\partial_3$ derivative. This observation plays an important role in closing the energy estimate.
\end{rem}

\begin{proof}
First of all, using the decomposition for gradient operator \eqref{grad-decomp}
and the expression of Laplacian in spherical coordinate
\begin{equation}\nonumber
\Delta = \partial_r^2 + \frac{2}{r}\partial_r +
\frac{1}{r^2}\Delta_{\mathbb{S}^2}.
\end{equation}
One derives that
\begin{eqnarray}\label{C-5}
|\Delta Y - \partial_{r}^2Y|
\leq \frac{2|\partial_rY|}{r} +
\frac{|\Delta_{\mathbb{S}^2} Y|}{r^2}
\lesssim \frac{|\nabla Y| +
|\nabla \Omega Y|}{r}.
\end{eqnarray}
Let us further write
\begin{eqnarray}\nonumber
(t^2 - r^2)\Delta Y
&=& - t^2(\partial_t^2 - \Delta) Y - r^2(\Delta
  Y - \partial_r^2Y) + t^2\partial_t^2Y - r^2\partial_r^2Y\\\nonumber
&=& - t^2(\partial_t^2 - \Delta) Y - r^2(\Delta
  Y - \partial_r^2Y)\\\nonumber
&&+\ (t\partial_t - r\partial_r)(t\partial_t +
  r\partial_r - 1)Y.
\end{eqnarray}
Hence, using \eqref{C-5}, one has
\begin{eqnarray}\label{C-6}
|(t - r)\Delta Y| &\lesssim& t|(\partial_t^2 - \Delta) Y| + r|\Delta
  Y - \partial_r^2Y| + |\partial SY| + |\partial Y|\\\nonumber
&\lesssim& |\partial Z Y| + |\partial Y| + t|(\partial_t^2 - \Delta) Y|.
\end{eqnarray}

Next, using \eqref{C-6}, integration by parts and the periodic boundary condition, one immediately
has
\begin{eqnarray}\nonumber
&&\|(t - r)\partial_i\partial_j Y\|_{L^2(\Omega_{\delta})}^2
= \int_{\Omega_\delta}  (t - r)^2\partial_i \partial_j Y\partial_i\partial_j Ydy\\\nonumber
&&= 2\int_{\Omega_\delta}(t - r)\omega_i\partial_j
  Y\partial_i\partial_j Ydy - \int_{\Omega_\delta}(t - r)^2\partial_j
  Y\Delta\partial_j Ydy\\\nonumber
&&= 2\int_{\Omega_\delta}(t - r)\omega_i\partial_j
  Y\partial_i\partial_j Ydy - 2\int_{\Omega_\delta}(t - r)\omega_j\partial_j
  Y\Delta Ydy\\\nonumber
&&\quad +\ \int_{\Omega_\delta}(t - r)^2\Delta Y\Delta
  Ydy\\\nonumber
&&\leq 8\|\nabla Y\|_{L^2(\Omega_\delta)}^2  + \frac{1}{2}
  \|(t - r)\partial_i\partial_jY\|_{L^2(\Omega_\delta)}^2
  + \int_{\Omega_\delta}(t - r)^2\Delta Y\Delta Ydy.
\end{eqnarray}
Using \eqref{C-6}, one obtains that
\begin{eqnarray}\label{C-7}
\|(t - r)\partial_i\partial_j Y\|_{L^2(\Omega_\delta)} \lesssim
\|\partial Y\|_{L^2(\Omega_\delta)} +\|\partial ZY\|_{L^2(\Omega_\delta)}+ t \|(\partial_t^2 -
\Delta)Y\|_{L^2(\Omega_\delta)}.
\end{eqnarray}

To estimate $|(t - r)\partial_{t}\nabla Y|$, let us first write
\begin{eqnarray}\nonumber
(t - r)\partial_t\partial_rY &=& - (\partial_t -
  \partial_r)(t\partial_t + r\partial_r - 1)Y
  + t\partial_t^2Y - r\partial_r^2Y\\\nonumber
&=& - (\partial_t - \partial_r)(t\partial_t +
  r\partial_r - 1)Y + t(\partial_t^2 - \Delta)Y\\\nonumber
&&+\ (t - r)\Delta Y + r(\Delta Y - \partial_r^2Y).
\end{eqnarray}
Then using \eqref{C-5} and \eqref{C-6}, one has
\begin{eqnarray}\label{C-10}
|(t - r)\partial_t\partial_rY| \lesssim |\partial Y|+|\partial Z Y| +
t|(\partial_t^2 - \Delta) Y|.
\end{eqnarray}
Consequently, we have
\begin{eqnarray}\nonumber
|(t - r)\partial_t\partial_j Y| &\leq & \big|(t - r)
  \partial_t\omega_j\partial_r Y| + |(t -
  r) r^{-1} \partial_t \Omega Y\big|\\\nonumber
&\leq& |(t - r)\omega_j\partial_t\partial_r Y| + \big|r^{-1}
  (t\partial_t + r\partial_r)\Omega
  Y - \Omega(\partial_t +
  \partial_r) Y\big|\\\nonumber
&\lesssim& |(t - r)\partial_t\partial_r Y| +
  r^{-1}|\Omega SY| + |\Omega
  \partial_tY| + |\Omega
  (\omega\cdot\nabla Y)|\\\nonumber
&\lesssim& |(t - r) \partial_t\partial_r Y| +
  r^{-1}|\Omega SY| +
  |\Omega \partial_tY|
 +\ | \nabla Y| + |
  \Omega \nabla Y|\\\nonumber
&\lesssim& |\partial Y| +|\partial ZY| + t|(\partial_t^2 - \Delta) Y|,
\end{eqnarray}
which gives
\begin{eqnarray}\label{C-8}
\|(t - r)\partial_t\partial_j Y\|_{L^2(\Omega_\delta)} \lesssim
\|\partial Y\|_{L^2(\Omega_\delta)} +\|\partial ZY\|_{L^2(\Omega_\delta)} + t \|(\partial_t^2 -
\Delta)Y\|_{L^2(\Omega_\delta)}.
\end{eqnarray}

At last, if $r \leq 2 t$, we deduce using \eqref{C-6} that
\begin{eqnarray}\nonumber
|(t - r)\partial_t^2 Y| &\leq& \big|(t - r)
  (\partial_t^2 - \Delta)Y\big| + |(t -
  r) \Delta Y|\\\nonumber
&\lesssim& |(t - r)
  (\partial_t^2 - \Delta)Y| + |\partial ZY| + |\partial Y|
  + t|(\partial_t^2 - \Delta) Y|\\\nonumber
&\lesssim& |\partial Z Y| + |\partial Y|
  +  |t(\partial_t^2 - \Delta) Y|,
\end{eqnarray}
and if $r > 2 t$, we deduce that
\begin{eqnarray}\nonumber
t |\partial_t^2 Y| \leq t |(\partial_t^2 - \Delta)Y\big| + t|\Delta Y| .
\end{eqnarray}
Hence, by \eqref{C-7}, we have
\begin{eqnarray}\label{C-9}
&&\int_{r \leq 2 t}( t - r)^2
|\partial_t^2  Y|^2dy + \int_{r > 2t}t^2|\partial_t^2
  Y|^2dy\\\nonumber
&&\lesssim t^2\|(\partial_t^2 - \Delta)Y\|_{L^2(\Omega_\delta)}^2 +
  \| |\partial Y\|_{L^2(\Omega_\delta)} +\| |\partial ZY\|_{L^2(\Omega_\delta)} + \int_{r > 2t}
  t^2|\Delta Y|^2dy\\\nonumber
&&\lesssim t^2\|(\partial_t^2 - \Delta)Y\|_{L^2(\Omega_\delta)}^2 +
  \|\partial Y\|_{L^2(\Omega_\delta)} +\|\partial ZY\|_{L^2(\Omega_\delta)}.
\end{eqnarray}
Then the lemma follows from \eqref{C-7}, \eqref{C-8} and
\eqref{C-9}.
\end{proof}
At the end of this section, let us estimate good derivatives
$\omega_j\partial_t + \partial_j$.
\begin{lem}\label{GoodDeri}
For $t\leq 3r $, there holds
\begin{equation}\label{WW2}
t |\omega_j\partial_t\partial Y + \partial_j\partial Y| \lesssim
|\partial Y| + |\partial  Z Y| + t|(\partial_t^2 - \Delta) Y|.
\end{equation}
\end{lem}
\begin{rem}
Similar to the above  Lemma \ref{WE-1}.
On the right hand side of \eqref{WW2},
we use the generalized derivative $Z$
involving no $\partial_3$ derivative.
\end{rem}

\begin{proof}
Let us first calculate that
\begin{eqnarray}\nonumber
&&t(\partial_t + \partial_r)(\partial_t - \partial_r)Y\\\nonumber
&&= t(\partial_t^2 - \Delta)Y + t(\Delta - \partial_r^2)Y\\\nonumber
&&= t(\partial_t^2 - \Delta)Y + \frac{t}{r}\big(2\partial_rY +
  \frac{\Delta_{\mathbb{S}^2} Y}{r}\big).
\end{eqnarray}
Consequently, we have
\begin{eqnarray}\label{G-1}
&&t|(\partial_t + \partial_r)(\partial_t -
\partial_r)Y|\\\nonumber
&&\lesssim t|(\partial_t^2 - \Delta)Y| + \frac{t}{r}\big(|\nabla Y| +
   \frac{|(x_i\partial_j - x_j\partial_i) \Omega
   Y|}{r}\big)\\\nonumber
&&\lesssim t|(\partial_t^2 - \Delta)Y| + \frac{t}{r}\big(|\nabla Y|
  + |\nabla\Omega Y|\big).
\end{eqnarray}
Next, using \eqref{C-5}, \eqref{C-6}, \eqref{C-10} and \eqref{G-1},
we calculate that
\begin{eqnarray}\nonumber
&&t|(\partial_t + \partial_r)\partial_rY|\\\nonumber &&\leq
\frac{t}{2}|(\partial_t +
\partial_r)(\partial_t -
  \partial_r)Y| + \frac{t}{2}|(\partial_t + \partial_r)(\partial_t +
  \partial_r)Y|\\\nonumber
&&\lesssim \frac{t}{2}|(\partial_t + \partial_r)(\partial_t -
  \partial_r)Y| + \frac{1}{2}|S(\partial_t +
  \partial_r)Y| + \frac{1}{2}|(t - r)\partial_r(\partial_t +
  \partial_r)Y|\\\nonumber
&&\lesssim \frac{t}{2}|(\partial_t + \partial_r)(\partial_t -
  \partial_r)Y| + \frac{1}{2}|S(\partial_t +
  \partial_r)Y|\\\nonumber
&&\quad +\ \frac{1}{2}|(t - r)\partial_{tr}^2Y| + \frac{1}{2}|(t -
  r)(\Delta - \partial_r^2)Y| + \frac{1}{2}|(t - r)\Delta
  Y|\\\nonumber
&&\lesssim t|(\partial_t^2 - \Delta)Y| + |\partial Y| + |\partial Z Y|.
\end{eqnarray}
Hence, we have
\begin{eqnarray}\label{G-2}
&&t|(\omega_j\partial_t + \partial_j)\partial_kY|\\\nonumber
&&= \big|t\omega_j(\partial_t + \partial_r)\partial_kY +
  \frac{t}{r} \Big(\frac{x}{r}\wedge\Omega\Big)_j
  \partial_kY\big|\\\nonumber
&&\lesssim |t(\partial_t + \partial_r)\partial_rY| +
  \frac{t}{r}\big(|\Omega \partial_kY| + |\partial_r\Omega Y|\big)\\\nonumber
&&\lesssim t|(\partial_t^2 - \Delta)Y| + |\partial Y| + |\partial Z Y|.
\end{eqnarray}

At last, let us calculate that
\begin{eqnarray}\nonumber
t|(\partial_t + \partial_r)\partial_tY| &\leq& t|(\partial_t +
  \partial_r)(\partial_t - \partial_r)Y| + t|(\partial_t +
  \partial_r)\partial_rY|\nonumber
\end{eqnarray}
which together with \eqref{G-1} and \eqref{G-2} gives that
\begin{eqnarray}\label{G-3}
&&t|(\omega_j\partial_t + \partial_j)\partial_tY|\\\nonumber
&&\lesssim t|(\partial_t^2 - \Delta)Y| + |\partial Y| + |\partial Z Y|.
\end{eqnarray}
Then the lemma follows from \eqref{G-2} and \eqref{G-3}.
\end{proof}

\section{Estimate of the $L^2$ weighted norm}\label{WE}

Now we are going to estimate the $L^2$ weighted generalized energy
$\mathcal{X}_{\kappa}$. First of all, we prove the following lemma
which says that the rescaled $L^2$ norm of $\nabla\Gamma^\alpha p$ and
$(\partial_t^2 - \Delta)\Gamma^\alpha Y$, which involve the
$(|\alpha| + 2)$-th order derivatives of unknowns, can be bounded by
certain quantities which only involve $(|\alpha| + 1)$-th order
derivatives of unknowns. This surprising result is
based on the inherent special structures of nonlinearities in the system.
We refer to \cite{Lei16} of Lemma 4.1 for the two dimensional isotropic version.

\begin{lem}\label{SN-1}
Suppose $\|\nabla Y\|_{L^\infty(\Omega_\delta)} \ll 1 $.  There holds
\begin{equation}\nonumber
\delta^{h-\frac12}\|\nabla\Gamma^\alpha p\|_{L^2(\Omega_\delta)} +
\delta^{h-\frac12} \|(\partial_t^2 -\Delta)\Gamma^\alpha Y\|_{L^2(\Omega_\delta)} \lesssim \Pi(\alpha, 2)
\end{equation}
where
\begin{eqnarray*} 
\Pi(\alpha, 2) =&& \sum_{\beta + \gamma = \alpha,\ \beta \neq \alpha} \Pi_1
+\sum_{\beta + \gamma = \alpha,\ |\beta| \leq|\gamma|}\Pi_2
  +\sum_{\beta + \gamma = \alpha,\ |\beta| > |\gamma|}\Pi_3 \\\nonumber
&& +\sum_{\beta+\gamma+\iota=\alpha,\ \beta\neq \alpha} \Pi_4
+\sum_{\beta+\gamma+\iota=\alpha,\ |\beta|\geq |\alpha|/3}  \Pi_5,
\end{eqnarray*}
where $\Pi_1$, $\Pi_2$ and $\Pi_3$ are given by
\begin{align*} 
&\Pi_1=\delta^{h-\frac12} \big\|
  |(\partial_t^2- \Delta)\Gamma^\beta Y| |\nabla\Gamma^\gamma Y|\big\|_{L^2(\Omega_\delta)},\\[-4mm]\nonumber\\\nonumber
&\Pi_2 =\delta^{h-\frac12}\|\nabla(- \Delta)^{-1}(-\partial_j,\partial_i)  \cdot
\big((\partial_i, \partial_j)
\partial_t \Gamma^\beta Y^i \partial_t\Gamma^\gamma Y^j
- (\partial_i, \partial_j) \partial_k\Gamma^\beta Y^i \partial_k \Gamma^\gamma Y^j \big)\|_{L^2(\Omega_\delta)},
\\[-4mm]\nonumber\\\nonumber
& \Pi_3 =\delta^{h-\frac12} \|\nabla(- \Delta)^{-1}(\partial_i, \partial_j) \cdot
\big( \partial_t \Gamma^\beta Y^i  (-\partial_j,\partial_i)\partial_t\Gamma^\gamma Y^j
-\partial_k\Gamma^\beta Y^i (-\partial_j,\partial_i)\partial_k \Gamma^\gamma Y^j\big)\|_{L^2(\Omega_\delta)},
\end{align*}
and $\Pi_4$, $\Pi_5$ are given by
\begin{align*} 
&\Pi_4=\delta^{h-\frac12}\big\|
|(\partial_t^2 - \Delta) \Gamma^\beta Y|\cdot |\nabla\Gamma ^\gamma Y|\cdot |\nabla\Gamma ^\iota Y|
\big\|_{L^2(\Omega_\delta)},\\[-4mm]\nonumber\\\nonumber
&\Pi_5 =\delta^{h-\frac12}
\big\| \partial\Gamma^\beta Y\big\|_{L^2(\Omega_\delta)}
\big\||\partial\nabla\Gamma ^\gamma Y\nabla\Gamma^\iota Y|
+|\nabla\Gamma ^\gamma Y\partial\nabla\Gamma^\iota Y|
\big\|_{L^\infty(\Omega_\delta)}.
\end{align*}
\end{lem}
\begin{proof}
We begin by writing \eqref{Elasticity-D} as follows
\begin{eqnarray}\nonumber
- \nabla\Gamma^\alpha p = (\partial_t^2 - \Delta)
\Gamma^\alpha Y + \sum_{\beta + \gamma = \alpha}C_\alpha^\beta (\nabla \Gamma^\beta Y)^{\top}(\partial_t^2 -
\Delta)\Gamma^\gamma Y.
\end{eqnarray}
Applying the divergence operator to the above equation and then
applying the operator $\nabla(- \Delta)^{-1}$, we obtain that
\begin{eqnarray*}
\nabla\Gamma^\alpha p &=& \sum_{\beta + \gamma = \alpha,\
  \gamma \neq \alpha}\nabla(- \Delta)^{-1}\nabla\cdot[C_\alpha^\beta(\nabla\Gamma^\beta Y)^\top(\partial_t^2
   - \Delta)\Gamma^\gamma Y]\\\nonumber
&&+\ \nabla(- \Delta)^{-1}\nabla\cdot[(\nabla Y)^\top(\partial_t^2 -
  \Delta)\Gamma^\alpha Y]\\[-4mm]\nonumber\\\nonumber
&&+\ \nabla(- \Delta)^{-1}\nabla\cdot(\partial_t^2 -
  \Delta)\Gamma^\alpha Y.
\end{eqnarray*}
Using the $L^2$ boundedness of the Riesz transform in Lemma \ref{OP1} and Lemma \ref{OP2},
one has
\begin{eqnarray}\label{D-1}
&&\delta^{h-\frac12}\|\nabla\Gamma^\alpha p\|_{L^2(\Omega_\delta)}\\\nonumber
&\lesssim& \delta^{h-\frac12}\sum_{\beta + \gamma = \alpha,\
  \gamma \neq \alpha}\|(\nabla\Gamma^\beta Y)^\top(\partial_t^2
   - \Delta)\Gamma^\gamma Y)\|_{L^2(\Omega_\delta)}\\\nonumber
&&+\ \delta^{h-\frac12}\|(\nabla Y)^\top(\partial_t^2 - \Delta)\Gamma^\alpha Y\|_{L^2(\Omega_\delta)}
  + \delta^{h-\frac12}\|\nabla(- \Delta)^{-1}\nabla\cdot(\partial_t^2
  - \Delta)\Gamma^\alpha Y\|_{L^2(\Omega_\delta)}.
\end{eqnarray}
As we will see later, the first two terms on the right hand side of \eqref{D-1} will be
treated directly. In what follows, we focus our mind on the calculation of
\begin{eqnarray}\label{D-2}
\delta^{h-\frac12} \|\nabla(- \Delta)^{-1}\nabla\cdot(\partial_t^2
  - \Delta)\Gamma^\alpha Y\|_{L^2(\Omega_\delta)}.
\end{eqnarray}
Using \eqref{Struc-2}, we compute that
\begin{align*} 
&\nabla\cdot (\partial_t^2
  - \Delta)\Gamma^\alpha Y \\\nonumber
&= -\frac12 (\partial_t^2 - \Delta)\sum_{\beta + \gamma = \alpha}
  C_\alpha^\beta\big(\partial_i \Gamma^\beta  Y^i \partial_j\Gamma^\gamma  Y^j
  - \partial_i\Gamma^\gamma  Y^j \partial_j\Gamma^\beta  Y^i\big)\\\nonumber
&\quad -(\partial_t^2 - \Delta)
\sum_{\beta+\gamma+\iota=\alpha}C_\alpha  ^\beta C_{\alpha-\beta}  ^\gamma
\left|
\begin{matrix}
 \partial_1\Gamma^\beta Y^1 &\partial_2 \Gamma^\beta Y^1 & \partial_3 \Gamma^\beta Y^1\\
 \partial_1\Gamma ^\gamma Y^2& \partial_2 \Gamma ^\gamma Y^2& \partial_3\Gamma ^\gamma Y^2\\
\partial_1\Gamma ^\iota Y^3&\partial_2\Gamma^\iota Y^3&\partial_3\Gamma^\iota Y^3
\end{matrix}
\right|\\\nonumber
&=J_1+J_2.
\end{align*}
Thus \eqref{D-2} are reduced to the computation of the quadratic nonlinear term
\begin{equation}\label{D-4}
\delta^{h-\frac12} \|\nabla(- \Delta)^{-1}J_1\|_{L^2(\Omega_\delta)}
\end{equation}
and the cubic term
\begin{equation}\label{D-5}
\delta^{h-\frac12} \|\nabla(- \Delta)^{-1}J_2\|_{L^2(\Omega_\delta)}.
\end{equation}

We first take care of the quadratic nonlinearity $J_1$.
Note the structural identity:
\begin{align*}
-\big(\partial_i\Gamma^\beta Y^i \partial_j\Gamma^\gamma Y^j -
  \partial_i\Gamma^\gamma Y^j\partial_j\Gamma^\beta Y^i\big)
=(\partial_i, \partial_j) \Gamma^\beta Y^i \cdot (-\partial_j,\partial_i)\Gamma^\gamma Y^j .
\end{align*}
$J_1$ can be computed as follows:
\begin{align}\label{D-6}
&\frac12 (\partial_t^2-\Delta)
\sum_{\beta + \gamma =\alpha} C_{\alpha}^\beta
(\partial_i, \partial_j) \Gamma^\beta Y^i \cdot (-\partial_j,\partial_i)\Gamma^\gamma Y^j \\ \nonumber
&=\frac12
\sum_{\beta + \gamma =\alpha} C_{\alpha}^\beta
(\partial_i, \partial_j) (\partial_t^2-\Delta) \Gamma^\beta Y^i \cdot (-\partial_j,\partial_i)\Gamma^\gamma Y^j \\ \nonumber
&\quad+\frac12
\sum_{\beta + \gamma =\alpha} C_{\alpha}^\beta
(\partial_i, \partial_j) \Gamma^\beta Y^i \cdot (-\partial_j,\partial_i) (\partial_t^2-\Delta)\Gamma^\gamma Y^j \\ \nonumber
&\quad+
\sum_{\beta + \gamma =\alpha} C_{\alpha}^\beta
(\partial_i, \partial_j) \partial_t \Gamma^\beta Y^i \cdot (-\partial_j,\partial_i)\partial_t\Gamma^\gamma Y^j \\ \nonumber
&\quad -
\sum_{\beta + \gamma =\alpha} C_{\alpha}^\beta
(\partial_i, \partial_j) \partial_k\Gamma^\beta Y^i \cdot (-\partial_j,\partial_i)\partial_k \Gamma^\gamma Y^j \\\nonumber
&=\frac12 (\partial_i, \partial_j) \cdot
\sum_{\beta + \gamma =\alpha} C_{\alpha}^\beta
 (\partial_t^2-\Delta) \Gamma^\beta Y^i  (-\partial_j,\partial_i)\Gamma^\gamma Y^j \\ \nonumber
&\quad+\frac12 (-\partial_j,\partial_i) \cdot
\sum_{\beta + \gamma =\alpha} C_{\alpha}^\beta
(\partial_i, \partial_j) \Gamma^\beta Y^i  (\partial_t^2-\Delta)\Gamma^\gamma Y^j\\\nonumber
&\quad+(-\partial_j,\partial_i)  \cdot
\sum_{\beta + \gamma =\alpha, |\beta|\leq |\gamma|} C_{\alpha}^\beta
\big((\partial_i, \partial_j)
\partial_t \Gamma^\beta Y^i \partial_t\Gamma^\gamma Y^j
- (\partial_i, \partial_j) \partial_k\Gamma^\beta Y^i \partial_k \Gamma^\gamma Y^j \big) \\ \nonumber
&\quad + (\partial_i, \partial_j) \cdot
\sum_{\beta + \gamma =\alpha, |\beta|> |\gamma|} C_{\alpha}^\beta
\big( \partial_t \Gamma^\beta Y^i  (-\partial_j,\partial_i)\partial_t\Gamma^\gamma Y^j
-\partial_k\Gamma^\beta Y^i (-\partial_j,\partial_i)\partial_k \Gamma^\gamma Y^j\big).
\end{align}
Inserting $\eqref{D-6}$ into \eqref{D-4}, one has
\begin{align*}
&\delta^{h-\frac12}\|\nabla(- \Delta)^{-1} J_1\|_{L^2(\Omega_\delta)}
\\[-4mm] \nonumber\\\nonumber
&\lesssim \Pi(\alpha, 2)+ \delta^{h-\frac12} \big\|
  |(\partial_t^2- \Delta)\Gamma^\alpha Y| |\nabla Y|\big\|_{L^2(\Omega_\delta)}.
\end{align*}

Next, let us compute $J_2$.
\begin{eqnarray*} 
&&J_2=-(\partial_t^2 - \Delta)
\sum_{\beta+\gamma+\iota=\alpha}C_\alpha  ^\beta C_{\alpha-\beta}  ^\gamma
\left|
\begin{matrix}
 \partial_1\Gamma^\beta Y^1 &\partial_2 \Gamma^\beta Y^1 & \partial_3 \Gamma^\beta Y^1\\
 \partial_1\Gamma ^\gamma Y^2& \partial_2 \Gamma ^\gamma Y^2& \partial_3\Gamma ^\gamma Y^2\\
\partial_1\Gamma ^\iota Y^3&\partial_2\Gamma^\iota Y^3&\partial_3\Gamma^\iota Y^3
\end{matrix}
\right|.
\end{eqnarray*}
Expanding the determinant, one has
\begin{eqnarray*}
\left|
\begin{matrix}
 \partial_1 f &\partial_2 f & \partial_3 f\\
 \partial_1 g &\partial_2 g & \partial_3 g\\
 \partial_1 h &\partial_2 h & \partial_3 h
\end{matrix}
\right|
&&=\partial_1 f \left|
\begin{matrix}
 \partial_2 g & \partial_3 g\\
 \partial_2 h & \partial_3 h
\end{matrix}
\right|
-\partial_2 f
\left|
\begin{matrix}
 \partial_1 g  & \partial_3 g\\
 \partial_1 h  & \partial_3 h
\end{matrix}
\right|
+\partial_3 f
\left|
\begin{matrix}
 \partial_1 g &\partial_2 g \\
 \partial_1 h &\partial_2 h
\end{matrix}
\right|
\\ \nonumber
&&=\nabla f\cdot
\left(
\left|
\begin{matrix}
 \partial_2 g & \partial_3 g\\
 \partial_2 h & \partial_3 h
\end{matrix}
\right|,
-\left|
\begin{matrix}
 \partial_1 g  & \partial_3 g\\
 \partial_1 h  & \partial_3 h
\end{matrix}
\right|,
\left|
\begin{matrix}
 \partial_1 g &\partial_2 g \\
 \partial_1 h &\partial_2 h
\end{matrix}
\right|
\right).
\end{eqnarray*}
In order to simplify the computations, we introduce the three dimensional curl operator on
two dimensional vector maps as follows:
\begin{align*}
\nabla\times (g,h)=
\left(
\left|
\begin{matrix}
 \partial_2 g & \partial_3 g\\
 \partial_2 h & \partial_3 h
\end{matrix}
\right|,
-\left|
\begin{matrix}
 \partial_1 g  & \partial_3 g\\
 \partial_1 h  & \partial_3 h
\end{matrix}
\right|,
\left|
\begin{matrix}
 \partial_1 g &\partial_2 g \\
 \partial_1 h &\partial_2 h
\end{matrix}
\right|
\right).
\end{align*}
It's easy to check the following property holds.
\begin{align*}
&\nabla \cdot \nabla \times (f,h)=0,\\
&\nabla f\cdot \nabla\times (g,h)=\nabla\cdot \big(f \nabla\times (g,h) \big),\\
& \left|
\begin{matrix}
 \partial_1 f &\partial_2 f & \partial_3 f\\
 \partial_1 g &\partial_2 g & \partial_3 g\\
 \partial_1 h &\partial_2 h & \partial_3 h
\end{matrix}
\right|=\nabla f\cdot \nabla\times (g,h)
=\nabla g\cdot \nabla\times (h,f)
=\nabla h\cdot \nabla\times (f,g).
\end{align*}
Now we calculate using the curl operator as follows:
\begin{eqnarray}\label{D-8}
J_2&&=-(\partial_t^2 - \Delta)
\sum_{\beta+\gamma+\iota=\alpha}C_\alpha  ^\beta C_{\alpha-\beta}  ^\gamma
\left|
\begin{matrix}
 \partial_1\Gamma^\beta Y^1 &\partial_2 \Gamma^\beta Y^1 & \partial_3 \Gamma^\beta Y^1\\
 \partial_1\Gamma ^\gamma Y^2& \partial_2 \Gamma ^\gamma Y^2& \partial_3\Gamma ^\gamma Y^2\\
\partial_1\Gamma ^\iota Y^3&\partial_2\Gamma^\iota Y^3&\partial_3\Gamma^\iota Y^3
\end{matrix}
\right|
\\\nonumber
&&=-\sum_{\beta+\gamma+\iota=\alpha}C_\alpha  ^\beta C_{\alpha-\beta}  ^\gamma
\nabla (\partial_t^2 - \Delta) \Gamma^\beta Y^1 \cdot \nabla\times(\Gamma ^\gamma Y^2, \Gamma ^\iota Y^3)
\\ \nonumber
&&\quad -
\sum_{\beta+\gamma+\iota=\alpha}C_\alpha  ^\beta C_{\alpha-\beta}  ^\gamma
\nabla (\partial_t^2 - \Delta) \Gamma ^\gamma Y^2 \cdot \nabla\times( \Gamma ^\iota Y^3, \Gamma^\beta Y^1)
\\ \nonumber
&&\quad -\sum_{\beta+\gamma+\iota=\alpha}C_\alpha^\beta C_{\alpha-\beta}^\gamma
\nabla (\partial_t^2 - \Delta)\Gamma^\iota Y^3  \cdot \nabla\times
(\Gamma^\beta Y^1,\Gamma ^\gamma Y^2)\\ \nonumber
&&\quad -2\sum_{\beta+\gamma+\iota=\alpha}C_\alpha^\beta C_{\alpha-\beta}^\gamma
\nabla\partial_t\Gamma^\beta Y^1  \cdot \partial_t\nabla\times (\Gamma ^\gamma Y^2, \Gamma^\iota Y^3)\\ \nonumber
&&\quad +2\sum_{\beta+\gamma+\iota=\alpha}C_\alpha^\beta C_{\alpha-\beta}^\gamma
\nabla\partial_k\Gamma^\beta Y^1  \cdot \partial_k\nabla\times (\Gamma ^\gamma Y^2, \Gamma^\iota Y^3).
 \nonumber
\end{eqnarray}
The first three terms on the right hand side of \eqref{D-8} can be organized as follows:
\begin{eqnarray}\label{D-9}
&&-\nabla\cdot \sum_{\beta+\gamma+\iota=\alpha}C_\alpha  ^\beta C_{\alpha-\beta}^\gamma
(\partial_t^2 - \Delta) \Gamma^\beta Y^1 \nabla\times(\Gamma ^\gamma Y^2, \Gamma ^\iota Y^3)
\\ \nonumber
&&\quad -\nabla\cdot
\sum_{\beta+\gamma+\iota=\alpha}C_\alpha  ^\beta C_{\alpha-\beta}  ^\gamma
(\partial_t^2 - \Delta) \Gamma ^\gamma Y^2  \nabla\times(\Gamma^\beta Y^1 , \Gamma ^\iota Y^3)
\\ \nonumber
&&\quad -\nabla \cdot\sum_{\beta+\gamma+\iota=\alpha}C_\alpha  ^\beta C_{\alpha-\beta}  ^\gamma
 (\partial_t^2 - \Delta)\Gamma ^\iota Y^3  \nabla\times
(\Gamma^\beta Y^1,\Gamma ^\gamma Y^2) .
\end{eqnarray}
The remaining two lines of \eqref{D-8} are divided into three types:
$|\beta|\geq |\alpha|/3$, $|\gamma|\geq |\alpha|/3$
and $|\iota|\geq |\alpha|/3$.
Due the symmetry, their treatment is the same. Hence
we only present the details for $|\beta|\geq |\alpha|/3$.
\begin{eqnarray}\label{D-10}
&&-2\sum_{\beta+\gamma+\iota=\alpha,\ |\beta|\geq |\alpha|/3}
C_\alpha^\beta C_{\alpha-\beta}^\gamma
\nabla\partial_t\Gamma^\beta Y^1  \cdot \partial_t\nabla\times (\Gamma ^\gamma Y^2, \Gamma^\iota Y^3)\\ \nonumber
&&+2\sum_{\beta+\gamma+\iota=\alpha,\ |\beta|\geq |\alpha|/3}
C_\alpha^\beta C_{\alpha-\beta}^\gamma
\nabla\partial_k\Gamma^\beta Y^1  \cdot \partial_k\nabla\times (\Gamma ^\gamma Y^2, \Gamma^\iota Y^3)
\\ \nonumber
&&=-2\nabla \cdot\sum_{\beta+\gamma+\iota=\alpha,\ |\beta|\geq |\alpha|/3}
C_\alpha^\beta C_{\alpha-\beta}^\gamma
\partial_t\Gamma^\beta Y^1  \partial_t\nabla\times (\Gamma ^\gamma Y^2, \Gamma^\iota Y^3)\\ \nonumber
&&\quad +2\nabla\cdot\sum_{\beta+\gamma+\iota=\alpha,\ |\beta|\geq |\alpha|/3}
C_\alpha^\beta C_{\alpha-\beta}^\gamma
\partial_k\Gamma^\beta Y^1 \partial_k\nabla\times (\Gamma ^\gamma Y^2, \Gamma^\iota Y^3).
\end{eqnarray}
Inserting \eqref{D-9} and \eqref{D-10} into \eqref{D-8}, \eqref{D-5} will be bounded as follows:
\begin{align*}
&\delta^{h-\frac12}\|\nabla(- \Delta)^{-1} J_2\|_{L^2(\Omega_\delta)} \\ \nonumber
&\lesssim
\delta^{h-\frac12} \sum_{\beta+\gamma+\iota=\alpha} \big\|
|(\partial_t^2 - \Delta) \Gamma^\beta Y|\cdot |\nabla\Gamma ^\gamma Y|\cdot |\nabla\Gamma ^\iota Y|
\big\|_{L^2(\Omega_\delta)}
\\ \nonumber
&\quad + \delta^{h-\frac12}\sum_{\beta+\gamma+\iota=\alpha,\ |\beta|\geq |\alpha|/3}
\| \partial\Gamma^\beta Y\|_{L^2(\Omega_\delta)}
\big\|
|\partial\nabla\Gamma ^\gamma Y|\cdot|\nabla\Gamma^\iota Y|
+|\nabla\Gamma ^\gamma Y|\cdot|\partial\nabla\Gamma^\iota Y|
\big\|_{L^\infty(\Omega_\delta)}\\\nonumber
&\lesssim
\delta^{h-\frac12}
\big\|(\partial_t^2 - \Delta) \Gamma^\alpha Y|\cdot |\nabla Y|^2
\big\|_{L^2(\Omega_\delta)}+\Pi(\alpha,2).
\end{align*}

Collecting all the above calculation, \eqref{D-1} becomes
\begin{align}\label{D-11}
\delta^{h-\frac12}\|\nabla\Gamma^\alpha p\|_{L^2(\Omega_\delta)}
 \lesssim
& \Pi(\alpha, 2) +\delta^{h-\frac12} \big\|
  |(\partial_t^2- \Delta)\Gamma^\alpha Y| |\nabla Y|\big\|_{L^2(\Omega_\delta)}\\[-4mm]\nonumber
  \\\nonumber
&+\delta^{h-\frac12} \big\|
|(\partial_t^2 - \Delta) \Gamma^\alpha Y|\cdot |\nabla Y|^2
\big\|_{L^2(\Omega_\delta)}.
\end{align}
Combined with \eqref{Elasticity-D} yields
\begin{align*}
\delta^{h-\frac12}\|(\partial_t^2- \Delta)\Gamma^\alpha Y\|_{L^2(\Omega_\delta)}
 \lesssim
& \Pi(\alpha, 2) +\delta^{h-\frac12}
\|(\partial_t^2- \Delta)\Gamma^\alpha Y\|_{L^2(\Omega_\delta)} \|\nabla Y \|_{L^\infty(\Omega_\delta)}\\[-4mm]\nonumber
  \\\nonumber
&+\delta^{h-\frac12} \|(\partial_t^2 - \Delta) \Gamma^\alpha Y\|_{L^2(\Omega_\delta)}
\|\nabla Y\|^2_{L^\infty(\Omega_\delta)}.
\end{align*}
Due to the assumption that $\|\nabla Y\|_{L^\infty(\Omega_\delta)} \ll 1$.
The last two terms are absorbed by the left hand side, we obtain
\begin{align}\label{D-13}
\delta^{h-\frac12}\|(\partial_t^2- \Delta)\Gamma^\alpha Y\|_{L^2(\Omega_\delta)}
 \lesssim \Pi(\alpha, 2) .
\end{align}
Inserting the above \eqref{D-13} into \eqref{D-11} yields
\begin{align*}
\delta^{h-\frac12}\| \nabla\Gamma^\alpha p\|_{L^2(\Omega_\delta)}
 \lesssim \Pi(\alpha, 2) .
\end{align*}
This ends the proof of the lemma.
\end{proof}

Next, we give a direct estimate of the scaled
$\|\nabla\Gamma^\alpha p\|_{L^2} $ and
$\|(\partial_t^2 -\Delta)\Gamma^\alpha Y\|_{L^2}$.
This will be used in the energy estimate.
\begin{lem}\label{SNB-1}
Let $\kappa\geq 15$. Suppose $\mathcal{E}_{\kappa-3}+\mathcal{F}_{\kappa-3} \ll 1 $.  Then there holds
\begin{align}\nonumber
&\sum_{|\alpha|\leq \kappa-4}\big( \delta^{h-\frac12}\|\nabla\Gamma^\alpha p\|_{L^2(\Omega_\delta)} +
\delta^{h-\frac12} \|(\partial_t^2 -\Delta)\Gamma^\alpha Y\|_{L^2(\Omega_\delta)} \big) \lesssim
\mathcal{E}_{\kappa-3}+\mathcal{F}_{\kappa-3}, \\\nonumber
&\sum_{|\alpha|\leq \kappa-1}\big( \delta^{h-\frac12}\|\nabla\Gamma^\alpha p\|_{L^2(\Omega_\delta)} +
\delta^{h-\frac12} \|(\partial_t^2 -\Delta)\Gamma^\alpha Y\|_{L^2(\Omega_\delta)} \big) \lesssim
\mathcal{E}_{\kappa}^{\frac12} (\mathcal{E}_{\kappa-3}^{\frac12}+\mathcal{F}_{\kappa-3}^{\frac12} ) .
\end{align}
\end{lem}
\begin{proof}
We first claim that there holds the following fact:
\begin{align}\label{PS}
&\delta^{h-\frac12}\|\nabla\Gamma^\alpha p\|_{L^2(\Omega_\delta)} +
\delta^{h-\frac12} \|(\partial_t^2 -\Delta)\Gamma^\alpha Y\|_{L^2(\Omega_\delta)}  \\[-4mm]\nonumber\\\nonumber
&\lesssim
\sum_{|\beta|<|\alpha| }
\delta^{l-\frac12}\| (\partial_t^2-\Delta)\Gamma^\beta Y\|_{L^2(\Omega_\delta)}
(\mathcal{E}_{[\alpha/2]+4}^{\frac 12}+\mathcal{E}_{[\alpha/2]+4}) \\\nonumber
&+\mathcal{E}_{|\alpha|+1}^{\frac12}(\mathcal{E}_{[\alpha/2]+5}^{\frac12}+\mathcal{F}_{[\alpha/2]+5}^{\frac12})
+\mathcal{E}_{|\alpha|+1}^{\frac12}(\mathcal{E}_{[\alpha/2]+5}+\mathcal{F}_{[\alpha/2]+5})\\\nonumber
&+ (\mathcal{E}_{[\alpha/2]+5}^{\frac12}+\mathcal{F}_{[\alpha/2]+5}^{\frac12})
\mathcal{E}_{|\alpha|+1}^{\frac12}\mathcal{E}_{[\alpha/2]+4}^{\frac12}.
\end{align}

In terms of Lemma \ref{SN-1}, we need to estimate $\Pi(\alpha,2)$.
For $\Pi_1$, by Lemma \ref{Sob2}, there holds
\begin{align*}
\sum_{\beta + \gamma = \alpha,\ \beta \neq \alpha} \Pi_1
\lesssim
& \delta^{h-\frac12} \sum_{\beta + \gamma = \alpha,\ |\gamma|\leq |\beta|<|\alpha| }
\| (\partial_t^2-\Delta)\Gamma^\beta Y\|_{L^2(\Omega_\delta)} \| \nabla \Gamma^\gamma Y\|_{L^\infty(\Omega_\delta)} \\\nonumber
&+\delta^{h-\frac12}\sum_{\beta + \gamma = \alpha,\ |\beta|\leq |\gamma| }
\| (\partial_t^2-\Delta)\Gamma^\beta Y\|_{L^\infty(\Omega_\delta)} \| \nabla \Gamma^\gamma Y\|_{L^2(\Omega_\delta)}  \\\nonumber
\lesssim
&\sum_{|\beta|<|\alpha| }
\delta^{l-\frac12}\| (\partial_t^2-\Delta)\Gamma^\beta Y\|_{L^2(\Omega_\delta)}
\mathcal{E}_{[\alpha/2]+4}^{\frac 12} \\\nonumber
&+(\mathcal{E}_{[\alpha/2]+5}^{\frac12}+\mathcal{F}_{[\alpha/2]+5}^{\frac12})  \mathcal{E}_{|\alpha|+1}^{\frac12}.
\end{align*}
For $\Pi_4$, by Lemma \ref{Sob2}, there holds
\begin{align*}
\sum_{\beta+\gamma+\iota=\alpha,\ \beta\neq \alpha} \Pi_4
\lesssim
&\sum_{\beta + \gamma+\iota = \alpha,\ [\alpha/2]\leq |\beta|<|\alpha| }
\delta^{h-\frac12}\big\|
|(\partial_t^2 - \Delta) \Gamma^\beta Y|\cdot |\nabla\Gamma ^\gamma Y|\cdot |\nabla\Gamma ^\iota Y|
\big\|_{L^2(\Omega_\delta)} \\\nonumber
&+\sum_{\beta + \gamma+\iota = \alpha,\ |\beta|\leq [\alpha/2]}
\delta^{h-\frac12}\big\|
|(\partial_t^2 - \Delta) \Gamma^\beta Y|\cdot |\nabla\Gamma ^\gamma Y|\cdot |\nabla\Gamma ^\iota Y|
\big\|_{L^2(\Omega_\delta)} \\\nonumber
\lesssim &\sum_{|\beta|<|\alpha| }
\delta^{l-\frac12}\|(\partial_t^2 - \Delta) \Gamma^\beta Y\|_{L^2(\Omega_\delta)}
\mathcal{E}_{[\alpha/2]+4} \\\nonumber
&+ (\mathcal{E}_{[\alpha/2]+5}^{\frac12}+\mathcal{F}_{[\alpha/2]+5}^{\frac12})
\mathcal{E}_{|\alpha|+1}^{\frac12}\mathcal{E}_{[\alpha/2]+4}^{\frac12}.
\end{align*}

Now we take care of $\Pi_2$ and $\Pi_3$. Due to their similarity,
we only present the estimate details for $\Pi_2$.
By Lemma \ref{Sob2}, there holds
\begin{align} \nonumber
 \sum_{\beta + \gamma = \alpha,\ |\beta| \leq|\gamma|} \Pi_2
&\leq \sum_{\beta + \gamma = \alpha,\ |\beta| \leq|\gamma|}  \delta^{h-\frac12}\big\|
|\nabla \partial \Gamma^\beta Y|\cdot  |\partial \Gamma^\gamma Y| \big\|_{L^2(\Omega_\delta)}
\\[-4mm]\nonumber\\\nonumber
&\lesssim
(\mathcal{E}_{[\alpha/2]+5}^{\frac{1}{2} }+\mathcal{F}_{[\alpha/2]+5}^{\frac{1}{2} }) \mathcal{E}_{|\alpha|+1}^{\frac{1}{2} }.
\end{align}
$\Pi_5$ will be bounded by
\begin{align*}
\sum_{\beta+\gamma+\iota=\alpha,\ |\beta|\geq |\alpha|/3}\Pi_5
\lesssim
\mathcal{E}_{|\alpha|+1}^{\frac{1}{2} }
(\mathcal{E}_{[\alpha/2]+5}^{\frac{1}{2} }+\mathcal{F}_{[\alpha/2]+5}^{\frac{1}{2} })
(\mathcal{E}_{[\alpha/2]+5}^{\frac{1}{2} }+\mathcal{F}_{[\alpha/2]+5}^{\frac{1}{2} }) .
\end{align*}
Thus \eqref{PS} is proved.

Now for $\kappa\geq 15$, $|\alpha|\leq \kappa-4$, one has $[\alpha/2]+5\leq \kappa-3$.
Hence \eqref{PS} becomes
\begin{align}\nonumber
&\sum_{|\alpha|\leq \kappa-4} \big(\delta^{h-\frac12}\|\nabla\Gamma^\alpha p\|_{L^2(\Omega_\delta)} +
\delta^{h-\frac12} \|(\partial_t^2 -\Delta)\Gamma^\alpha Y\|_{L^2(\Omega_\delta)} \big)\\\nonumber
& \lesssim\sum_{|\beta|<\kappa-5 }
\delta^{l-\frac12}\| (\partial_t^2-\Delta)\Gamma^\beta Y\|_{L^2(\Omega_\delta)} \mathcal{E}_{\kappa-3}^{\frac 12} \\\nonumber
&\quad +\mathcal{E}_{\kappa-3}^{\frac 12}
(\mathcal{E}_{\kappa-3}^{\frac 12}+\mathcal{F}_{\kappa-3}^{\frac 12}).
\end{align}
 Absorbing the first term on the right hand side yields the first inequality.

Next for $\kappa\geq 15$, $|\alpha|\leq \kappa-1$, one has $[\alpha/2]+5\leq \kappa-3$.
Hence \eqref{PS} yields
\begin{align}\nonumber
&\sum_{|\alpha|\leq \kappa-1} \big(\delta^{h-\frac12}\|\nabla\Gamma^\alpha p\|_{L^2(\Omega_\delta)} +
\delta^{h-\frac12} \|(\partial_t^2 -\Delta)\Gamma^\alpha Y\|_{L^2(\Omega_\delta)} \big)\\\nonumber
& \lesssim\sum_{|\beta|<\kappa-2 }
\delta^{l-\frac12}\| (\partial_t^2-\Delta)\Gamma^\beta Y\|_{L^2(\Omega_\delta)} \mathcal{E}_{\kappa-3}^{\frac 12} \\\nonumber
&\quad +\mathcal{E}_{\kappa}^{\frac 12}
(\mathcal{E}_{\kappa-3}^{\frac 12}+\mathcal{F}_{\kappa-3}^{\frac 12}).
\end{align}
Absorbing the first term on the right hand side yields the second inequality.
\end{proof}

Next, we estimate the pressure with the $t$-factor.
\begin{lem}\label{SN-3}
Let $\kappa \geq 15$, $t\geq 1$. There exists $0<\eta\ll 1$ such that
if $\mathcal{E}_{\kappa - 3}+\mathcal{F}_{\kappa-3} \leq \eta$. Then
 there hold
\begin{equation}\nonumber
t\sum_{|\alpha|\leq \kappa-5}
\delta^{(h-\frac12)}\big(\|\nabla\Gamma^{\alpha} p\|_{L^2(\Omega_\delta)}
+  \|(\partial_t^2 - \Delta)\Gamma^{\alpha} Y\|_{L^2(\Omega_\delta)}\big)
\lesssim \mathcal{E}_{ \kappa - 3}^{\frac{1}{2}}
\big(\mathcal{X}_{\kappa - 3}^{\frac{1}{2}} +\mathcal{E}_{\kappa - 3}^{\frac{1}{2}}
+ \mathcal{F}_{\kappa-3}^{\frac{1}{2}}\big)
\end{equation}
and
\begin{equation}\nonumber
t\sum_{|\alpha|\leq \kappa-2}
\delta^{(h-\frac12)} \big(\|\nabla\Gamma^{\alpha} p\|_{L^2(\Omega_\delta)} +
 \|(\partial_t^2 - \Delta)\Gamma^{\alpha} Y\|_{L^2(\Omega_\delta)}\big)
\lesssim \mathcal{E}_{\kappa}^{\frac{1}{2}}\big(
\mathcal{X}_{ \kappa-3}^{\frac{1}{2}} +
\mathcal{E}_{ \kappa-3}^{\frac{1}{2}} +
\mathcal{F}_{ \kappa-3}^{\frac{1}{2}}\big).
\end{equation}
\end{lem}
\begin{proof}
In terms of Lemma \ref{SN-1}, we need to control $t\Pi(\alpha,2)$
\begin{eqnarray*}
\Pi(\alpha, 2) =&& \sum_{\beta + \gamma = \alpha,\ \beta \neq \alpha} \Pi_1
+\sum_{\beta + \gamma = \alpha,\ |\beta| \leq|\gamma|}\Pi_2
  +\sum_{\beta + \gamma = \alpha,\ |\beta| > |\gamma|}\Pi_3 \\\nonumber
&& +\sum_{\beta+\gamma+\iota=\alpha,\ \beta\neq \alpha} \Pi_4
+\sum_{\beta+\gamma+\iota=\alpha,\ |\beta|\geq |\alpha|/3}  \Pi_5,
\end{eqnarray*}
where $\Pi_1$, $\Pi_2$, $\Pi_3$, $\Pi_4$ and $\Pi_5$ are given in Lemma \ref{SN-1}.

We first take care of $\Pi_1$ $\Pi_2$ and $\Pi_3$.
Due to the similarity between $\Pi_2$ and $\Pi_3$, we only give the estimate details for
$\Pi_2$ but omit the ones for $\Pi_3$.
For the expression inside $\Pi_2$, they can be organized as follows
\begin{align*}
&(-\partial_j,\partial_i)  \cdot
\big[(1-\varphi^t)
\big((\partial_i, \partial_j)
\partial_t \Gamma^\beta Y^i \partial_t\Gamma^\gamma Y^j
- (\partial_i, \partial_j) \partial_k\Gamma^\beta Y^i \partial_k \Gamma^\gamma Y^j \big)\big]
\\[-4mm] \nonumber\\\nonumber
&\quad+(-\partial_j,\partial_i)  \cdot
\big[\varphi^t
\big((\partial_i, \partial_j)
\partial_t \Gamma^\beta Y^i \partial_t\Gamma^\gamma Y^j
- (\partial_i, \partial_j) \partial_k\Gamma^\beta Y^i \partial_k \Gamma^\gamma Y^j \big)\big]
\\[-4mm] \nonumber\\\nonumber
&=(-\partial_j,\partial_i)  \cdot
\big[(1-\varphi^t)
\big((\partial_i, \partial_j)
\partial_t \Gamma^\beta Y^i \partial_t\Gamma^\gamma Y^j
- (\partial_i, \partial_j) \partial_k\Gamma^\beta Y^i \partial_k \Gamma^\gamma Y^j \big)\big]
\\[-4mm] \nonumber\\\nonumber
&\quad+(-\partial_j,\partial_i)  \cdot
\big[\varphi^t (\omega_k(\omega_k\partial_t+\partial_k))(\partial_i, \partial_j)
\Gamma^\beta Y^i \partial_t\Gamma^\gamma Y^j \big] \\[-4mm] \nonumber\\\nonumber
&\quad-  (\partial_i, \partial_j)\cdot
\big[ \partial_k  \Gamma^\beta Y^i (-\partial_j,\partial_i)(\varphi^t (\omega_k\partial_t+\partial_k)\Gamma^\gamma Y^j )\big] ,
\end{align*}
where $\varphi^t$ is cutoff function introduced in Lemma \ref{K-S-O}.
Note that here, the good derivatives always contain an extra derivative. This gives us the so-called
strong null condition.
By the $L^2$ boundedness of Riesz transform on thin domain in Lemma \ref{OP1} and Lemma \ref{OP2},
The quadratic nonlinearity $\Pi_1$ and $\Pi_2$ can be bounded by
\begin{align*}
& \delta^{h-\frac12}\sum_{\beta + \gamma =\alpha,\ \beta\neq \alpha}
\big\| |(\partial_t^2-\Delta) \Gamma^\beta Y|\cdot |\nabla\Gamma^\gamma Y|
 \big\|_{L^2(\Omega_\delta)}\\ \nonumber
&+ \delta^{h-\frac12}\sum_{\beta + \gamma =\alpha,\ |\beta|\leq |\gamma|} \|
(1-\varphi^t) \big((\partial_i, \partial_j)
\partial_t \Gamma^\beta Y^i \partial_t\Gamma^\gamma Y^j
- (\partial_i, \partial_j) \partial_k\Gamma^\beta Y^i \partial_k \Gamma^\gamma Y^j \big) \|_{L^2(\Omega_\delta)}
\\ \nonumber
&+ \delta^{h-\frac12}\sum_{\beta + \gamma =\alpha,\ |\beta|\leq |\gamma|} \|
\varphi^t (\omega_k(\omega_k\partial_t+\partial_k))(\partial_i, \partial_j)
\Gamma^\beta Y^i \partial_t\Gamma^\gamma Y^j  \|_{L^2(\Omega_\delta)}
\\ \nonumber
&+ \delta^{h-\frac12}\sum_{\beta + \gamma =\alpha,\ |\beta|\leq |\gamma|}\|
 \partial_k  \Gamma^\beta Y^i (-\partial_j,\partial_i)(\varphi^t (\omega_k\partial_t+\partial_k)\Gamma^\gamma Y^j )\|_{L^2(\Omega_\delta)}\\ \nonumber
&\lesssim   \delta^{h-\frac12}\sum_{\beta + \gamma =\alpha,\ \beta\neq \alpha}\big\|
 |(\partial_t^2-\Delta) \Gamma^\beta Y|\cdot  |\nabla\Gamma^\gamma Y|
 \big\|_{L^2(\Omega_\delta)}\\ \nonumber
&\quad + \delta^{h-\frac12}\sum_{\beta + \gamma =\alpha,\ |\beta|\leq |\gamma|} \big\|
(1-\varphi^t)
|\nabla \partial \Gamma^\beta Y| \cdot |\partial\Gamma^\gamma Y| \big\|_{L^2(\Omega_\delta)}
\\ \nonumber
&\quad + \delta^{h-\frac12}\sum_{k} \sum_{\beta + \gamma =\alpha,\ |\beta|\leq |\gamma|} \big\|
\varphi^t |(\omega_k\partial_t+\partial_k)\nabla
\Gamma^\beta Y| \cdot|\partial_t\Gamma^\gamma Y| \big\|_{L^2(\Omega_\delta)}
\\ \nonumber
&\quad + \delta^{h-\frac12}\sum_{k}\sum_{\beta + \gamma =\alpha,\ |\beta|\leq |\gamma|} \big\|
 |\nabla \Gamma^\beta Y| \cdot \nabla(\varphi^t (\omega_k\partial_t+\partial_k)\Gamma^\gamma Y )\big\|_{L^2(\Omega_\delta)}.
\end{align*}
By Lemma \ref{GoodDeri}, the above are further estimated as follows.
\begin{align}\label{E-1}
& \quad  \delta^{h-\frac12}\sum_{\beta + \gamma =\alpha,\ |\gamma|\leq |\beta|} \|
 (\partial_t^2-\Delta) \Gamma^\beta Y  \|_{L^2(\Omega_\delta)}
 \| \partial\Gamma^\gamma Y\|_{L^\infty(\Omega_\delta)}\\ \nonumber
&\quad + \delta^{h-\frac12}\sum_{\beta + \gamma =\alpha,\ |\beta|\leq |\gamma|}
 \|(\partial_t^2-\Delta)\Gamma^\beta Y\|_{L^\infty(\Omega_\delta)}
\| \partial\Gamma^\gamma Y  \|_{L^2(\Omega_\delta)}
\\ \nonumber
&\quad + \delta^{h-\frac12}\sum_{\beta + \gamma =\alpha,\ |\beta|\leq |\gamma|}
\|(1-\varphi^t)\nabla \partial \Gamma^\beta Y \|_{L^\infty(\Omega_\delta)}
\cdot\|\partial\Gamma^\gamma Y\|_{L^2(\Omega_\delta)}
\\ \nonumber
&\quad + t^{-1}\delta^{h-\frac12}\sum_{k}\sum_{\beta + \gamma =\alpha,\ |\beta|\leq |\gamma|} \|
\varphi^t (|\nabla\Gamma^\beta Y| + |\nabla Z\Gamma^\beta Y|)
\|_{L^\infty(\Omega_\delta)} \cdot\|\partial\Gamma^\gamma Y\|_{L^2(\Omega_\delta)}
\\ \nonumber
&\quad + \delta^{h-\frac12}\sum_{\beta + \gamma =\alpha,\ |\beta|\leq |\gamma|} \big\|
 |\nabla \Gamma^\beta Y| \cdot \nabla\varphi^t \partial\Gamma^\gamma Y \big\|_{L^2(\Omega_\delta)}\\ \nonumber
&\quad + \delta^{h-\frac12}\sum_{k}\sum_{\beta + \gamma =\alpha,\ |\beta|\leq |\gamma|} \big\|
 |\nabla \Gamma^\beta Y| \cdot\varphi^t \nabla\omega_k\partial_t\Gamma^\gamma Y \|_{L^2(\Omega_\delta)}. \nonumber
\end{align}
Note the index $\alpha=(h,a),\ \beta=(l,b),\ \gamma=(m,c)$, $h=l+m$.
The first two lines of \eqref{E-1} are estimated by
\begin{align*}
& \quad \sum_{[\alpha/2] \leq|\beta|\leq |\alpha|}
\delta^{l-\frac12}\|
 (\partial_t^2-\Delta) \Gamma^\beta Y \|_{L^2(\Omega_\delta)}
 \mathcal{E}_{[\alpha/2]+4}^{\frac12} \\ \nonumber
&\quad + \sum_{|\beta| \leq [\alpha/2]+3}
\delta^{l-\frac12}\|(\partial_t^2-\Delta)\Gamma^\beta Y\|_{L^2(\Omega_\delta)}
\mathcal{E}_{|\alpha|+1}^{\frac12}.
\end{align*}
The third  line of \eqref{E-1} correspond to the the domain away from the light cone.
By Lemma \ref{KS-A}, they are bounded as follows
\begin{align*}
& \sum_{\beta + \gamma =\alpha,\ |\beta|\leq |\gamma|}
\delta^{l}\|(1-\varphi^t)\nabla \partial \Gamma^\beta Y \|_{L^\infty(\Omega_\delta)}
\cdot\delta^{m-\frac12}\|\partial\Gamma^\gamma Y\|_{L^2(\Omega_\delta)}
\\ \nonumber
&\lesssim
t^{-1} \sum_{\beta + \gamma =\alpha,\ |\beta|\leq |\gamma|}
(\mathcal{X}_{|\beta|+5}^{\frac12}
+\mathcal{E}_{|\beta|+4}^{\frac12}+\mathcal{F}_{|\beta|+4}^{\frac12})
\mathcal{E}_{|\gamma|+1}^{\frac12}\\\nonumber
&\lesssim t^{-1}
(\mathcal{X}_{[\alpha/2]+5}^{\frac12}
+ \mathcal{E}_{[\alpha/2]+4}^{\frac12}+\mathcal{F}_{[\alpha/2]+4}^{\frac12})
\mathcal{E}_{|\alpha|+1}^{\frac12}.
\end{align*}
The remaining three lines of \eqref{E-1} are directly bounded by
\begin{align*}
t^{-1} \mathcal{E}_{|\alpha|+1}^{\frac12}\mathcal{E}_{[\alpha/2]+5}^{\frac12}.
\end{align*}
We summarize the estimate for $\Pi_1$ and $\Pi_2$ to deduce that
\begin{align*}
&t\sum_{\beta + \gamma = \alpha,\ \beta \neq \alpha} \Pi_1
+t\sum_{\beta + \gamma = \alpha,\ |\beta| \leq|\gamma|}\Pi_2 \\
&\lesssim\sum_{[\alpha/2]\leq |\beta|\leq |\alpha|}
\delta^{l-\frac12}t\|
 (\partial_t^2-\Delta) \Gamma^\beta Y \|_{L^2(\Omega_\delta)}
\mathcal{E}_{[\alpha/2]+4}^{\frac12}\\ \nonumber
&\quad + \sum_{|\beta|\leq [\alpha/2]+3}
\delta^{l-\frac12}t\|(\partial_t^2-\Delta)\Gamma^\beta Y\|_{L^2(\Omega_\delta)}
\mathcal{E}_{|\alpha|+1}^{\frac12}
\\ \nonumber
&\quad
+\mathcal{E}_{|\alpha|+1}^{\frac12}
\big(\mathcal{X}_{[\alpha/2]+5}^{\frac12}
+\mathcal{E}_{[\alpha/2]+5}^{\frac12} +
\mathcal{F}_{[\alpha/2]+4}^{\frac12}\big).
\end{align*}

Next we deal with the cubic term $\Pi_4$ and $\Pi_5$.
Note that relation among the index: $\alpha=(h,a),\ \beta=(l,b),\ \gamma=(m,c)$, $\iota=(n,d)$, $h=l+m+n$.
By Lemma \ref{Sob2} and Lemma \ref{KS-A}, there holds
\begin{align}
& \sum_{\beta+\gamma+\iota=\alpha,\ \beta\neq \alpha} \Pi_4
 +\sum_{\beta+\gamma+\iota=\alpha,|\beta|\geq |\alpha|/3}  \Pi_5 \\\nonumber
&\lesssim \sum_{\beta+\gamma+\iota=\alpha,\ |\alpha|/2\leq|\beta|<|\alpha|}
 \delta^{l-\frac12}\|
(\partial_t^2 - \Delta) \Gamma^\beta Y\|_{L^2(\Omega_\delta)}
\delta^{m}\| \nabla\Gamma ^\gamma Y\|_{L^\infty(\Omega_\delta)}
\delta^{n}\|\nabla\Gamma ^\iota Y\|_{L^\infty(\Omega_\delta)}
\\ \nonumber
&+\sum_{\beta+\gamma+\iota=\alpha,\ |\gamma|\geq |\alpha|/2}
 \delta^{l}\|
(\partial_t^2 - \Delta) \Gamma^\beta Y\|_{L^\infty(\Omega_\delta)}
\delta^{m-\frac12}\| \nabla\Gamma ^\gamma Y\|_{L^2(\Omega_\delta)}
\delta^{n}\|\nabla\Gamma ^\iota Y\|_{L^\infty(\Omega_\delta)}
\\ \nonumber
&\quad + \sum_{\beta+\gamma+\iota=\alpha,\ |\beta|\geq |\alpha|/3}
\delta^{l-\frac12}\| \partial\Gamma^\beta Y\|_{L^2(\Omega_\delta)}
\delta^{m+n}\big\|
|\partial\nabla\Gamma ^\gamma Y|\cdot|\nabla\Gamma^\iota Y|
+|\nabla\Gamma ^\gamma Y|\cdot|\partial\nabla\Gamma^\iota Y|
\big\|_{L^\infty(\Omega_\delta)} \\ \nonumber
&\lesssim
\sum_{|\alpha|/2\leq |\beta|<|\alpha|}
 \delta^{l-\frac12} \|
(\partial_t^2 - \Delta) \Gamma^\beta Y\|_{L^2(\Omega_\delta)}
\mathcal{E}_{[\alpha/2]+4}
\\ \nonumber
&\quad+\sum_{|\beta|\leq [\alpha/2]+3}
 \delta^{l-\frac12}\|
(\partial_t^2 - \Delta) \Gamma^\beta Y\|_{L^2(\Omega_\delta)}
\mathcal{E}_{|\alpha|+1}^{\frac12} \mathcal{E}_{[\alpha/2]+4}^{\frac12}\\ \nonumber
&\quad+t^{-1}\mathcal{E}_{|\alpha|+1}^{\frac12} (\mathcal{E}_{[\alpha/2]+4}^{\frac12}
+ \mathcal{X}_{[\alpha/2]+4}^{\frac12}
+ \mathcal{F}_{[\alpha/2]+4}^{\frac12} ) \mathcal{E}_{[\alpha/2]+4}^{\frac12} .
\end{align}
Finally we summarize that
\begin{align*}\nonumber
&  t\delta^{(h-\frac12)}\|\nabla\Gamma^{\alpha} p\|_{L^2(\Omega_\delta)}
+ t\delta^{(k-\frac12)} \|(\partial_t^2 - \Delta)\Gamma^{\alpha} Y\|_{L^2(\Omega_\delta)}
 \\[-4mm] \\ \nonumber
&\lesssim
t\sum_{|\alpha|/2\leq |\beta|<|\alpha|}
 \delta^{l-\frac12} \|
(\partial_t^2 - \Delta) \Gamma^\beta Y\|_{L^2(\Omega_\delta)}
\big(\mathcal{E}_{[\alpha/2]+4}+ \mathcal{E}_{[\alpha/2]+4}^{\frac12}\big)
\\ \nonumber
&\quad+t\sum_{|\beta|\leq [\alpha/2]+3}
 \delta^{l-\frac12}\|
(\partial_t^2 - \Delta) \Gamma^\beta Y\|_{L^2(\Omega_\delta)}
\mathcal{E}_{|\alpha|+1}^{\frac12} (\mathcal{E}_{[\alpha/2]+4}^{\frac12}+1)\\ \nonumber
&\quad+\mathcal{E}_{|\alpha|+1}^{\frac12}
\big(\mathcal{X}_{[\alpha/2]+5}^{\frac12}
+\mathcal{E}_{[\alpha/2]+5}^{\frac12} +
\mathcal{F}_{[\alpha/2]+5}^{\frac12}\big)
\big(\mathcal{E}_{[\alpha/2]+5}^{\frac12}+1\big).
\end{align*}

For $\kappa\geq 15$, $|\alpha|\leq \kappa-5$, one has $[\alpha/2]+5\leq \kappa-3$.
Hence we have
\begin{align*}
&t\sum_{|\alpha|\leq \kappa-5}
\big(\delta^{(h-\frac12)}\|\nabla\Gamma^{\alpha} p\|_{L^2(\Omega_\delta)}
+ \delta^{(k-\frac12)} \|(\partial_t^2 - \Delta)\Gamma^{\alpha} Y\|_{L^2(\Omega_\delta)}\big) \\ \nonumber
&\lesssim
t\sum_{|\alpha|\leq \kappa-5}\delta^{(h-\frac12)} \|(\partial_t^2 - \Delta)\Gamma^{\alpha} Y\|_{L^2(\Omega_\delta)}
\mathcal{E}_{ \kappa-3}^{\frac{1}{2}}
+\mathcal{E}_{ \kappa-3}^{\frac{1}{2}}
\big(\mathcal{X}_{\kappa-3}^{\frac{1}{2}} +\mathcal{E}_{\kappa-3}^{\frac{1}{2}}
+ \mathcal{F}_{\kappa-3}^{\frac{1}{2}}\big).
\end{align*}
Due the assumption  $\mathcal{E}_{ \kappa-3}\ll 1$,
the first term will absorbed by the left hand. This yields the first inequality of the lemma.

Next, for $\kappa\geq 15$, $|\alpha|\leq \kappa-2$, one has $[\alpha/2]+5\leq \kappa-3$.
Hence
\begin{align*}
&t\sum_{|\alpha|\leq \kappa-2}
\big(\delta^{(h-\frac12)}\|\nabla\Gamma^{\alpha} p\|_{L^2(\Omega_\delta)}
+ \delta^{(h-\frac12)} \|(\partial_t^2 - \Delta)\Gamma^{\alpha} Y\|_{L^2(\Omega_\delta)}\big) \\ \nonumber
&\lesssim
t\sum_{|\alpha|\leq \kappa-2}\delta^{(h-\frac12)} \|(\partial_t^2 - \Delta)\Gamma^{\alpha} Y\|_{L^2(\Omega_\delta)}
\mathcal{E}_{ \kappa-3}^{\frac{1}{2}}
+t\sum_{|\alpha|\leq \kappa-5}\delta^{(h-\frac12)} \|(\partial_t^2 - \Delta)\Gamma^{\alpha} Y\|_{L^2(\Omega_\delta)}
\mathcal{E}_{ \kappa}^{\frac{1}{2}} \\
&\quad
+\mathcal{E}_{ \kappa }^{\frac{1}{2}}
\big(\mathcal{X}_{\kappa-3}^{\frac{1}{2}} +\mathcal{E}_{\kappa-3}^{\frac{1}{2}}
+ \mathcal{F}_{\kappa-3}^{\frac{1}{2}}\big)
+\mathcal{X}_{ \kappa }^{\frac{1}{2}}\mathcal{E}_{\kappa-3}^{\frac{1}{2}}.
\end{align*}
Due the assumption that $\mathcal{E}_{ \kappa-3}\ll 1$,
the first term will absorbed by the left hand. Then combined with the
first inequality of the lemma will yields the second inequality.
\end{proof}

We are ready to state the following lemma
that the weighted $L^2$ norm will be bounded by the
weighted energy under some smallness assumptions.
\begin{lem}\label{WE-2}
Let $\kappa \geq 15$. There exists $0<\eta\ll 1$ such that
if $\mathcal{E}_{\kappa-3}+ \mathcal{F}_{\kappa-3} \leq \eta$, then there holds
\begin{equation}\nonumber
\mathcal{X}_{\kappa-3} \lesssim \mathcal{E}_{\kappa - 3}+\mathcal{F}_{\kappa - 3} ,\quad
\mathcal{X}_{\kappa} \lesssim \mathcal{E}_{\kappa}.
\end{equation}
\end{lem}
\begin{proof}
Applying Lemma \ref{WE-1} and Lemma \ref{SN-3}, one has
\begin{eqnarray}\nonumber
\mathcal{X}_{\kappa - 3} &\lesssim& \mathcal{E}_{\kappa -3} + t^2\sum_{|\alpha|\leq \kappa-5}\delta^{2h-1}
\|(\partial_t^2 -\Delta)\Gamma^{\alpha}Y\|^2_{L^2(\Omega_\delta)}\\\nonumber
&\lesssim& \mathcal{E}_{\kappa -3}
+ \mathcal{E}_{\kappa -3}\big(\mathcal{E}_{\kappa -3}
 + \mathcal{X}_{\kappa -3}+\mathcal{F}_{\kappa -3}\big).
\end{eqnarray}
This gives the first inequality of the lemma by noting the
assumption.

Next, applying Lemma \ref{WE-1} and Lemma \ref{SN-3} once more, one
has
\begin{eqnarray}\nonumber
\mathcal{X}_{\kappa} &\lesssim& \mathcal{E}_{
  \kappa}+ t^2\sum_{|\alpha|\leq \kappa-2}\delta^{2h-1}\|(\partial_t^2 -\Delta)\Gamma^{\alpha}Y\|^2_{L^2(\Omega_\delta)}\\\nonumber
&\lesssim& \mathcal{E}_{\kappa} +
  \mathcal{E}_{\kappa}\big(\mathcal{E}_{\kappa -3} + \mathcal{X}_{\kappa -3}
  +\mathcal{F}_{\kappa -3} \big) .
\end{eqnarray}
Then the second estimate of the lemma follows from the first one and
the assumption.
\end{proof}

Combining Lemma \ref{SN-3} and Lemma \ref{WE-2},
there holds the following lemma.
\begin{lem}\label{WE-3}
Let $\kappa \geq 15$, $t\geq 1$. There exists $0<\eta\ll 1$ such that
if $\mathcal{E}_{\kappa-3}+ \mathcal{F}_{\kappa-3} \leq \eta$, then there holds
\begin{equation}\nonumber
t\sum_{|\alpha|\leq \kappa-5}
\delta^{(h-\frac12)}\big(\|\nabla\Gamma^{\alpha} p\|_{L^2(\Omega_\delta)}
+  \|(\partial_t^2 - \Delta)\Gamma^{\alpha} Y\|_{L^2(\Omega_\delta)}\big)
\lesssim \mathcal{E}_{ \kappa - 3}^{\frac{1}{2}}
\big(\mathcal{E}_{\kappa - 3}^{\frac{1}{2}}
+ \mathcal{F}_{\kappa-3}^{\frac{1}{2}}\big)
\end{equation}
and
\begin{equation}\nonumber
t\sum_{|\alpha|\leq \kappa-2}
\delta^{(h-\frac12)} \big(\|\nabla\Gamma^{\alpha} p\|_{L^2(\Omega_\delta)} +
 \|(\partial_t^2 - \Delta)\Gamma^{\alpha} Y\|_{L^2(\Omega_\delta)}\big)
\lesssim \mathcal{E}_{\kappa}^{\frac{1}{2}}\big(
\mathcal{E}_{ \kappa-3}^{\frac{1}{2}} +
\mathcal{F}_{ \kappa-3}^{\frac{1}{2}}\big).
\end{equation}
\end{lem}

\section{Second derivative estimate of pressure}\label{EPWTD}
This section is devoted to the estimate of $\|\partial_3\nabla\Gamma^{\alpha} p\|_{L^2(\Omega_\delta)}$.
We remark that a direct estimate of $\|\partial_3\nabla\Gamma^{\alpha} p\|_{L^2(\Omega_\delta)}$
is insufficient to close the energy estimate due to the appearance of $\partial_3^3 \Gamma^\alpha Y$.
Instead, we are going to estimate $\| \nabla^2\Gamma^{\alpha} p\|_{L^2(\Omega_\delta)}$.
\begin{lem}\label{Pr3}
Let $\kappa \geq 15$. There exists $0<\eta\ll 1$ such that
if $\mathcal{E}_{\kappa-3}+\mathcal{F}_{\kappa-3} \leq \eta$.
Then for $|\alpha|\leq \kappa-5$, there holds
\begin{equation}\nonumber
\delta^{(h-\frac12)} \|\nabla^2\Gamma^{\alpha} p\|_{L^2(\Omega_\delta)}
+ \delta^{(h-\frac12)}  \|\nabla(\partial_t^2 - \Delta)\Gamma^{\alpha} Y\|_{L^2(\Omega_\delta)}
\lesssim \sum_{|\alpha|\leq \kappa-5}\Phi(\alpha,2),
\end{equation}
where
\begin{align*}
\Phi(\alpha,2)=
&\delta^{(h-\frac12)}\sum_{\beta + \gamma = \alpha}\big\| |\nabla^2\Gamma^\beta Y|
 \cdot| (\partial_t^2 - \Delta)\Gamma^\gamma Y|\big\|_{L^2(\Omega_\delta)} \\\nonumber
& + \delta^{(h-\frac12)}\sum_{k}\sum_{\beta + \gamma =\alpha}\big\| |\nabla\partial \Gamma^\beta Y| \cdot |(\omega_k\partial_t+\partial_k)\nabla\Gamma^\gamma Y| \big\|_{L^2(\Omega_\delta)}   \\ \nonumber
&+\delta^{(h-\frac12)}\sum_{\beta+\gamma+\iota=\alpha} \big\|
|\nabla\partial \Gamma^\beta Y|  \cdot |\partial \nabla\times (\Gamma ^\gamma Y, \Gamma^\iota Y)|
\big\|_{L^2(\Omega_\delta)}.
\end{align*}
\end{lem}
\begin{proof}
Applying the divergence operator to \eqref{Elasticity-D} and then
applying the operator $\nabla^2(- \Delta)^{-1}$, we obtain that
\begin{eqnarray}\nonumber
\nabla^2\Gamma^\alpha p &=& \sum_{\beta + \gamma = \alpha}
\nabla^2(- \Delta)^{-1}\nabla\cdot[C_\alpha^\beta
(\nabla\Gamma^\beta Y)^\top(\partial_t^2
   - \Delta)\Gamma^\gamma Y]\\\nonumber
&&+\ \nabla^2(- \Delta)^{-1}\nabla\cdot(\partial_t^2 -
  \Delta)\Gamma^\alpha Y.
\end{eqnarray}
Using the $L^2$ boundedness of the Riesz transform in thin domain in Lemma \ref{OP1} and Lemma \ref{OP2},
one has
\begin{align}\label{F-1}
&\|\nabla^2\Gamma^\alpha p\|_{L^2(\Omega_\delta)}\\[-4mm]\nonumber\\\nonumber
&\lesssim
 \sum_{\beta + \gamma = \alpha}
  \|\nabla\big[(\nabla\Gamma^\beta Y)^\top(\partial_t^2
   - \Delta)\Gamma^\gamma Y\big]\|_{L^2(\Omega_\delta)}
+ \|\nabla\cdot(\partial_t^2- \Delta)\Gamma^\alpha Y\|_{L^2(\Omega_\delta)} \\ \nonumber
&\lesssim
\sum_{\beta + \gamma = \alpha}\big\| |\nabla\Gamma^\beta Y|
 \cdot|\nabla(\partial_t^2 - \Delta)\Gamma^\gamma Y|\big\|_{L^2(\Omega_\delta)}
+ \sum_{\beta + \gamma = \alpha}\big\| |\nabla^2\Gamma^\beta Y|
 \cdot| (\partial_t^2 - \Delta)\Gamma^\gamma Y|\big\|_{L^2(\Omega_\delta)} \\\nonumber
&\quad +  \|\nabla\cdot(\partial_t^2- \Delta)\Gamma^\alpha Y\|_{L^2(\Omega_\delta)}.
\end{align}
It remains to compute the last line of \eqref{F-1}.
Using \eqref{Struc-2}, we have
\begin{eqnarray}\label{F-2}
&&\nabla\cdot (\partial_t^2
  - \Delta)\Gamma^\alpha Y \\\nonumber
&&= -\frac12 (\partial_t^2 - \Delta)\sum_{\beta + \gamma = \alpha}
  C_\alpha^\beta\big(\partial_i \Gamma^\beta  Y^i \partial_j\Gamma^\gamma  Y^j
  - \partial_i\Gamma^\gamma  Y^j \partial_j\Gamma^\beta  Y^i\big)\\\nonumber
&&\quad -(\partial_t^2 - \Delta)
\sum_{\beta+\gamma+\iota=\alpha}C_\alpha  ^\beta C_{\alpha-\beta}  ^\gamma
\left|
\begin{matrix}
 \partial_1\Gamma^\beta Y^1 &\partial_2 \Gamma^\beta Y^1 & \partial_3 \Gamma^\beta Y^1\\
 \partial_1\Gamma ^\gamma Y^2& \partial_2 \Gamma ^\gamma Y^2& \partial_3\Gamma ^\gamma Y^2\\
\partial_1\Gamma ^\iota Y^3&\partial_2\Gamma^\iota Y^3&\partial_3\Gamma^\iota Y^3
\end{matrix}
\right|.
\end{eqnarray}
We first compute the first line on the right hand side of \eqref{F-2}.
Note the structural identity
\begin{align*}
-\big(\partial_i\Gamma^\beta Y^i \partial_j\Gamma^\gamma Y^j -
  \partial_i\Gamma^\gamma Y^j\partial_j\Gamma^\beta Y^i\big)
=(\partial_i, \partial_j) \Gamma^\beta Y^i \cdot (-\partial_j,\partial_i)\Gamma^\gamma Y^j .
\end{align*}
We write
\begin{align}\label{F-3}
&\frac12 (\partial_t^2-\Delta)
\sum_{\beta + \gamma =\alpha} C_{\alpha}^\beta
(\partial_i, \partial_j) \Gamma^\beta Y^i \cdot (-\partial_j,\partial_i)\Gamma^\gamma Y^j \\ \nonumber
&=\frac12
\sum_{\beta + \gamma =\alpha} C_{\alpha}^\beta
(\partial_i, \partial_j) (\partial_t^2-\Delta) \Gamma^\beta Y^i \cdot (-\partial_j,\partial_i)\Gamma^\gamma Y^j \\ \nonumber
&\quad+\frac12
\sum_{\beta + \gamma =\alpha} C_{\alpha}^\beta
(\partial_i, \partial_j) \Gamma^\beta Y^i \cdot (-\partial_j,\partial_i) (\partial_t^2-\Delta)\Gamma^\gamma Y^j \\ \nonumber
&\quad+
\sum_{\beta + \gamma =\alpha} C_{\alpha}^\beta
(\partial_i, \partial_j) \partial_t \Gamma^\beta Y^i \cdot (-\partial_j,\partial_i)\partial_t\Gamma^\gamma Y^j \\ \nonumber
&\quad -
\sum_{\beta + \gamma =\alpha} C_{\alpha}^\beta
(\partial_i, \partial_j) \partial_k\Gamma^\beta Y^i \cdot (-\partial_j,\partial_i)\partial_k \Gamma^\gamma Y^j .
\end{align}
 The last two lines of \eqref{F-3} are organized as follows to show the strong null condition.
\begin{align*}
&\partial_t(\partial_i, \partial_j)
\Gamma^\beta Y^i \partial_t(-\partial_j, \partial_i)\Gamma^\gamma Y^j
- \partial_k(\partial_i, \partial_j)\Gamma^\beta Y^i \partial_k(-\partial_j, \partial_i) \Gamma^\gamma Y^j
\\[-4mm]\\
&=(\omega_k(\omega_k\partial_t+\partial_k))(\partial_i, \partial_j)
\Gamma^\beta Y^i \partial_t(-\partial_j, \partial_i) \Gamma^\gamma Y^j \\[-4mm]\\
&\quad- \partial_k(\partial_i, \partial_j) \Gamma^\beta Y^i (\omega_k\partial_t+\partial_k)(-\partial_j, \partial_i)\Gamma^\gamma Y^j.
\end{align*}
Consequently,
we have
\begin{align}\label{F-4}
&\|\frac12 (\partial_t^2-\Delta)
\sum_{\beta + \gamma =\alpha} C_{\alpha}^\beta
(\partial_i, \partial_j) \Gamma^\beta Y^i \cdot (-\partial_j,\partial_i)\Gamma^\gamma Y^j\|_{L^2(\Omega_\delta)} \\ \nonumber
&\lesssim \sum_{\beta + \gamma =\alpha}\big\|
|\nabla(\partial_t^2-\Delta) \Gamma^\beta Y| \cdot |\nabla\Gamma^\gamma Y| \big\|_{L^2(\Omega_\delta)}
\\ \nonumber
& \quad+\sum_{k} \sum_{\beta + \gamma =\alpha}\big\| |\nabla\partial \Gamma^\beta Y| \cdot |(\omega_k\partial_t+\partial_k)\nabla\Gamma^\gamma Y| \big\|_{L^2(\Omega_\delta)}.
\end{align}

Next, we compute the last line of \eqref{F-2}.
Using the curl operator introduced previously, we have
\begin{eqnarray}\label{F-6}
&&-(\partial_t^2 - \Delta)
\sum_{\beta+\gamma+\iota=\alpha}C_\alpha  ^\beta C_{\alpha-\beta}  ^\gamma
\left|
\begin{matrix}
 \partial_1\Gamma^\beta Y^1 &\partial_2 \Gamma^\beta Y^1 & \partial_3 \Gamma^\beta Y^1\\
 \partial_1\Gamma ^\gamma Y^2& \partial_2 \Gamma ^\gamma Y^2& \partial_3\Gamma ^\gamma Y^2\\
\partial_1\Gamma ^\iota Y^3&\partial_2\Gamma^\iota Y^3&\partial_3\Gamma^\iota Y^3
\end{matrix}
\right|
\\\nonumber
&&=-\sum_{\beta+\gamma+\iota=\alpha}C_\alpha  ^\beta C_{\alpha-\beta}  ^\gamma
\nabla (\partial_t^2 - \Delta) \Gamma^\beta Y^1 \cdot \nabla\times(\Gamma ^\gamma Y^2, \Gamma ^\iota Y^3)
\\ \nonumber
&&\quad -
\sum_{\beta+\gamma+\iota=\alpha}C_\alpha  ^\beta C_{\alpha-\beta}  ^\gamma
\nabla (\partial_t^2 - \Delta) \Gamma ^\gamma Y^2 \cdot \nabla\times( \Gamma ^\iota Y^3, \Gamma^\beta Y^1)
\\ \nonumber
&&\quad -\sum_{\beta+\gamma+\iota=\alpha}C_\alpha^\beta C_{\alpha-\beta}^\gamma
\nabla (\partial_t^2 - \Delta)\Gamma^\iota Y^3  \cdot \nabla\times
(\Gamma^\beta Y^1,\Gamma ^\gamma Y^2)\\ \nonumber
&&\quad -2\sum_{\beta+\gamma+\iota=\alpha}C_\alpha^\beta C_{\alpha-\beta}^\gamma
\nabla\partial_t\Gamma^\beta Y^1  \cdot \partial_t\nabla\times (\Gamma ^\gamma Y^2, \Gamma^\iota Y^3)\\ \nonumber
&&\quad +2\sum_{\beta+\gamma+\iota=\alpha}C_\alpha^\beta C_{\alpha-\beta}^\gamma
\nabla\partial_k\Gamma^\beta Y^1  \cdot \partial_k\nabla\times (\Gamma ^\gamma Y^2, \Gamma^\iota Y^3).
 \nonumber
\end{eqnarray}
Taking the $L^2$ norm \eqref{F-6}, we deduce that
\begin{align}\label{F-8}
&\Big\|(\partial_t^2 - \Delta)
\sum_{\beta+\gamma+\iota=\alpha}C_\alpha  ^\beta C_{\alpha-\beta}  ^\gamma
\left|
\begin{matrix}
 \partial_1\Gamma^\beta Y^1 &\partial_2 \Gamma^\beta Y^1 & \partial_3 \Gamma^\beta Y^1\\
 \partial_1\Gamma ^\gamma Y^2& \partial_2 \Gamma ^\gamma Y^2& \partial_3\Gamma ^\gamma Y^2\\
\partial_1\Gamma ^\iota Y^3&\partial_2\Gamma^\iota Y^3&\partial_3\Gamma^\iota Y^3
\end{matrix}
\right|\Big\|_{L^2(\Omega_\delta)}
\\\nonumber
&\lesssim \sum_{\beta+\gamma+\iota=\alpha}\big\|
|\nabla (\partial_t^2 - \Delta) \Gamma^\beta Y| \cdot |\nabla\times(\Gamma^\gamma Y, \Gamma^\iota Y)|
\big\|_{L^2(\Omega_\delta)}
\\ \nonumber
&+\sum_{\beta+\gamma+\iota=\alpha} \big\|
|\nabla\partial \Gamma^\beta Y|  \cdot |\partial \nabla\times (\Gamma ^\gamma Y, \Gamma^\iota Y)|
\big\|_{L^2(\Omega_\delta)}.\nonumber
\end{align}
Inserting \eqref{F-4} and \eqref{F-8} into \eqref{F-1},
 we obtain that
\begin{align}\nonumber
\|\nabla^2\Gamma^\alpha p\|_{L^2(\Omega_\delta)}
\lesssim &\sum_{\beta + \gamma =\alpha}\big\|
|\nabla(\partial_t^2-\Delta) \Gamma^\beta Y| \cdot |\nabla\Gamma^\gamma Y| \big\|_{L^2(\Omega_\delta)} \\\nonumber
&+ \sum_{\beta + \gamma = \alpha}\big\| |\nabla^2\Gamma^\beta Y|
 \cdot| (\partial_t^2 - \Delta)\Gamma^\gamma Y|\big\|_{L^2(\Omega_\delta)}\\ \nonumber
& + \sum_{k}\sum_{\beta + \gamma =\alpha}\big\| |\nabla\partial \Gamma^\beta Y| \cdot |(\omega_k\partial_t+\partial_k)\nabla\Gamma^\gamma Y| \big\|_{L^2(\Omega_\delta)}   \\ \nonumber
&+\sum_{\beta+\gamma+\iota=\alpha}\big\|
|\nabla (\partial_t^2 - \Delta) \Gamma^\beta Y| \cdot |\nabla\times(\Gamma^\gamma Y, \Gamma^\iota Y)|
\big\|_{L^2(\Omega_\delta)}
\\ \nonumber
&+\sum_{\beta+\gamma+\iota=\alpha} \big\|
|\nabla\partial \Gamma^\beta Y|  \cdot |\partial \nabla\times (\Gamma ^\gamma Y, \Gamma^\iota Y)|
\big\|_{L^2(\Omega_\delta)} .
\end{align}
Hence
\begin{align}\label{F-9}
&\delta^{h-\frac12}\|\nabla(\partial_t^2-\Delta)\Gamma^\alpha Y\|_{L^2(\Omega_\delta)}
+\delta^{h-\frac12}\|\nabla^2\Gamma^\alpha p\|_{L^2(\Omega_\delta)}
\\[-4mm]\nonumber\\ \nonumber
&\lesssim \Phi(\alpha,2)+ \sum_{\beta + \gamma =\alpha}
\delta^{h-\frac12}\big\| |\nabla(\partial_t^2-\Delta) \Gamma^\beta Y| \cdot |\nabla\Gamma^\gamma Y| \big\|_{L^2(\Omega_\delta)}  \\ \nonumber
&+\sum_{\beta+\gamma+\iota=\alpha}\delta^{h-\frac12}\big\|
|\nabla(\partial_t^2 - \Delta) \Gamma^\beta Y| \cdot |\nabla\times(\Gamma^\gamma Y, \Gamma^\iota Y)|
\big\|_{L^2(\Omega_\delta)}.
\end{align}
Recall $\kappa\geq 15$, $|\alpha|\leq \kappa-5$. By Lemma \ref{Sob2}, there holds
\begin{align}\label{F-10}
&\sum_{\beta + \gamma =\alpha}
\delta^{h-\frac12}\big\| |\nabla(\partial_t^2-\Delta) \Gamma^\beta Y| \cdot |\nabla\Gamma^\gamma Y| \big\|_{L^2(\Omega_\delta)}  \\ \nonumber
&\lesssim
\sum_{\beta + \gamma =\alpha,\ |\beta|\leq|\gamma|}
\delta^{l} \|\nabla(\partial_t^2-\Delta) \Gamma^\beta Y\|_{L^\infty(\Omega_\delta)}
\mathcal{E}_{|\gamma|+1}^{\frac 12}
 \\ \nonumber
&+\sum_{\beta + \gamma =\alpha,\ |\beta|>|\gamma|}
\delta^{l-\frac12} \| \nabla(\partial_t^2-\Delta) \Gamma^\beta Y\|_{L^2(\Omega_\delta)}
\mathcal{E}_{|\gamma|+4}^{\frac 12}  \\ \nonumber
&\lesssim \sum_{|\alpha|\leq \kappa-5}
\delta^{h-\frac12} \|\nabla(\partial_t^2-\Delta) \Gamma^\alpha Y\|_{L^2(\Omega_\delta)}
\mathcal{E}_{\kappa-3}^{\frac 12}.
\end{align}
Similarly,
\begin{align}\label{F-11}
&\sum_{\beta+\gamma+\iota=\alpha}\delta^{h-\frac12}\big\|
|\nabla(\partial_t^2 - \Delta) \Gamma^\beta Y| \cdot |\nabla\times(\Gamma^\gamma Y, \Gamma^\iota Y)|
\big\|_{L^2(\Omega_\delta)}
\\ \nonumber
&\lesssim \sum_{|\alpha|\leq \kappa-5}
\delta^{h-\frac12} \|\nabla(\partial_t^2-\Delta) \Gamma^\alpha Y\|_{L^2(\Omega_\delta)}
\mathcal{E}_{\kappa-3}.
\end{align}
Inserting \eqref{F-10} and \eqref{F-11} into \eqref{F-9}
 would imply that
\begin{align*}
&\sum_{|\alpha|\leq \kappa-5}
\big(\delta^{h-\frac12}\|\nabla (\partial_t^2-\Delta)\Gamma^\alpha Y\|_{L^2(\Omega_\delta)}
+\delta^{h-\frac12}\| \nabla^2\Gamma^\alpha p\|_{L^2(\Omega_\delta)}\big)
\\ \nonumber
&\lesssim \sum_{|\alpha|\leq \kappa-5}\Phi(\alpha,2)+
\sum_{|\alpha|\leq \kappa-5}
\delta^{h-\frac12} \|\nabla(\partial_t^2-\Delta) \Gamma^\alpha Y\|_{L^2(\Omega_\delta)}
(\mathcal{E}_{\kappa-3}^{\frac12}+\mathcal{E}_{\kappa-3}).
\end{align*}
Due to the assumption $\mathcal{E}_{\kappa-3}\ll 1$,
the last term is absorbed by the left hand side.
Thus the lemma is proved.
\end{proof}

\begin{lem}\label{SNB-2}
Let $\kappa \geq 15$. There exists $0< \eta\ll 1$ such that
if $\mathcal{E}_{\kappa-3} \leq \eta$. Then for $|\alpha|\leq \kappa-5$, there holds
\begin{equation}\nonumber
\sum_{|\alpha|\leq \kappa-5}
\big(\delta^{(h-\frac12)} \|\nabla\partial_3\Gamma^{\alpha} p\|_{L^2(\Omega_\delta)}
+ \delta^{(h-\frac12)}  \|(\partial_t^2 - \Delta)\partial_3\Gamma^{\alpha} Y\|_{L^2(\Omega_\delta)} \big)
\lesssim 
\mathcal{E}_{\kappa - 3} + \mathcal{F}_{\kappa -3}.
\end{equation}
\end{lem}
\begin{proof}
By Lemma \ref{Pr3}, we need to control $\Phi(\alpha,2)$
for $|\alpha|\leq \kappa-5$.
By Lemma \ref{Sob2}, one deduces that
\begin{align*}
\Phi(\alpha,2)
\lesssim
&\delta^{(h-\frac12)}\sum_{\beta + \gamma = \alpha}\big\| \nabla^2\Gamma^\beta Y|
 \cdot|\partial^2 \Gamma^\gamma Y|\big\|_{L^2(\Omega_\delta)} \\\nonumber
&+\delta^{(h-\frac12)}\sum_{\beta+\gamma+\iota=\alpha} \big\|
|\nabla\partial \Gamma^\beta Y|  \cdot |\partial \nabla\times (\Gamma ^\gamma Y, \Gamma^\iota Y)|
\big\|_{L^2(\Omega_\delta)}\\\nonumber
&\lesssim (\mathcal{E}_{\kappa-3}+\mathcal{F}_{\kappa-3})(1+\mathcal{E}_{\kappa-3}^{\frac 12}).
\end{align*}
Then the lemma follows from the assumption $\mathcal{E}_{\kappa-3} \leq \eta\ll 1$.
\end{proof}

Next we consider the case with $t$-factor.
\begin{lem}
Let $\kappa \geq 15$, $t\geq 1$. There exists $0<\eta\ll 1$ such that
if $\mathcal{E}_{\kappa - 3} \leq \eta$. Then for $|\alpha|\leq \kappa-5$, there holds
\begin{equation}\nonumber
t\sum_{|\alpha|\leq \kappa-5}
\big(\delta^{(h-\frac12)} \|\nabla\partial_3\Gamma^{\alpha} p\|_{L^2(\Omega_\delta)}
+ \delta^{(h-\frac12)}  \|(\partial_t^2 - \Delta)\partial_3\Gamma^{\alpha} Y\|_{L^2(\Omega_\delta)}\big)
\lesssim \mathcal{E}_{\kappa - 3}+ \mathcal{F}_{\kappa -3}.
\end{equation}
\end{lem}
\begin{proof}
In view of Lemma \ref{Pr3}, we need to take care of $\Phi(\alpha,2)$ in Lemma \ref{Pr3}.
Recall $\Phi(\alpha,2)$
\begin{align*}
\Phi(\alpha,2)
=&\delta^{(h-\frac12)}\sum_{\beta + \gamma = \alpha}\big\| |\nabla^2\Gamma^\beta Y|
 \cdot| (\partial_t^2 - \Delta)\Gamma^\gamma Y|\big\|_{L^2(\Omega_\delta)} \\\nonumber
& + \delta^{(h-\frac12)}\sum_{k}\sum_{\beta + \gamma =\alpha}\big\| |\nabla\partial \Gamma^\beta Y| \cdot |(\omega_k\partial_t+\partial_k)\nabla\Gamma^\gamma Y| \big\|_{L^2(\Omega_\delta)}   \\ \nonumber
&+\delta^{(h-\frac12)}\sum_{\beta+\gamma+\iota=\alpha} \big\|
|\nabla\partial \Gamma^\beta Y|  \cdot |\partial \nabla\times (\Gamma ^\gamma Y, \Gamma^\iota Y)|
\big\|_{L^2(\Omega_\delta)}.
\end{align*}
The estimate of $\Phi(\alpha,2)$ is similar to $\Pi(\alpha,2)$ in the proof of Lemma \ref{SNB-1}.
Hence in the following, we will sketch the proof.

Recall the index notation $\alpha=(h,a),\ \beta=(l,b),\ \gamma=(m,c)$, $h=l+m$ and $|\alpha|\leq \kappa-5$.
By Lemma \ref{Sob2}, Lemma \ref{SN-3} and Lemma \ref{WE-2}, we compute
\begin{align*}
& \sum_{\beta + \gamma =\alpha,\ |\beta|\leq |\gamma|}
\delta^{h-\frac12}\big\| |\nabla\partial\Gamma^\beta Y|\cdot
|(\omega_k\partial_t+\partial_k)\nabla \Gamma^\gamma Y| \big\|_{L^2(\Omega_\delta)} + \\ \nonumber
& \sum_{\beta + \gamma =\alpha,\ |\beta|> |\gamma|}
\delta^{h-\frac12}\big\| |\nabla\partial\Gamma^\beta Y|\cdot
|(\omega_k\partial_t+\partial_k)\nabla \Gamma^\gamma Y| \big\|_{L^2(\Omega_\delta)}  \\ \nonumber
&\lesssim\sum_{\beta + \gamma =\alpha,\ |\beta|\leq |\gamma|}
\delta^{h-\frac12} \|\nabla\partial\Gamma^\beta Y\|_{L^\infty(\Omega_\delta)}
\|(\omega_k\partial_t+\partial_k)\nabla \Gamma^\gamma Y\|_{L^2(\Omega_\delta)} + \\ \nonumber
& \sum_{\beta + \gamma =\alpha,\ |\beta|> |\gamma|}
\delta^{h-\frac12} \| \nabla\partial\Gamma^\beta Y \|_{L^2(\Omega_\delta)}
\|(\omega_k\partial_t+\partial_k)\nabla \Gamma^\gamma Y \|_{L^\infty(\Omega_\delta)}  \\ \nonumber
&\lesssim t^{-1} (\mathcal{E}_{\kappa-3}+\mathcal{F}_{\kappa-3}).
\end{align*}
By Lemma \ref{Sob2} and Lemma \ref{WE-3}, we have
\begin{align*}
&\delta^{(h-\frac12)}\sum_{\beta + \gamma = \alpha}\big\| \nabla^2\Gamma^\beta Y|
 \cdot| (\partial_t^2 - \Delta)\Gamma^\gamma Y|\big\|_{L^2(\Omega_\delta)} \\\nonumber
&\lesssim t^{-1} (\mathcal{E}_{\kappa-3}+\mathcal{F}_{\kappa-3}).
\end{align*}
By Lemma \ref{Sob2}, Lemma \ref{KS-A} and Lemma \ref{WE-2}, there holds
\begin{align*}
&\delta^{(h-\frac12)}\sum_{\beta+\gamma+\iota=\alpha} \big\|
|\nabla\partial \Gamma^\beta Y|  \cdot |\partial \nabla\times (\Gamma ^\gamma Y, \Gamma^\iota Y)|
\big\|_{L^2(\Omega_\delta)} \\\nonumber
&\lesssim t^{-1} \mathcal{E}_{\kappa-3}^{\frac 12}(\mathcal{E}_{\kappa-3}+\mathcal{F}_{\kappa-3}).
\end{align*}
Collecting the argument and notice that $\mathcal{E}_{\kappa - 3} \ll 1$
will imply the lemma.
\end{proof}

\section{Higher-order energy estimate}\label{HEE}

This section is devoted to the estimate of the anisotropic higher-order generalized energy $\mathcal{E}_\kappa$.
The ghost weight method introduced by
Alinhac in \cite{Alinhac01a} plays an important role.

Let $\kappa \geq 15$, $|\alpha| \leq \kappa - 1$, $\sigma = t
- r$ and $q(\sigma) = \arctan\sigma$. Taking the $L^2$ inner product
of \eqref{Elasticity-D} with $\delta^{2(h-\frac 12)}e^{- q(\sigma)}\partial_t\Gamma^\alpha
Y$ and using integration by parts, we have
\begin{eqnarray}\nonumber
&&\frac{d}{dt}\int_{\Omega_\delta} \delta^{2(h-\frac 12)}e^{- q(\sigma)}\big(|\partial_t
  \Gamma^\alpha Y|^2 + |\nabla \Gamma^\alpha Y|^2\big)dy\\\nonumber
&&= - \delta^{2(h-\frac 12)}\sum_j\int_{\Omega_\delta} \frac{e^{- q(\sigma)}}{1 +
  \sigma^2}|(\omega_j\partial_t + \partial_j)
  \Gamma^\alpha Y|^2dy\\\nonumber
&&\quad -\ 2\delta^{2(h-\frac 12)}\int_{\Omega_\delta} e^{- q(\sigma)}\partial_t
  \Gamma^\alpha Y\cdot\nabla\Gamma^\alpha pdy\\\nonumber
&&\quad -\ 2\delta^{2(h-\frac 12)}\int_{\Omega_\delta} e^{- q(\sigma)}\partial_t
  \Gamma^\alpha Y\cdot\sum_{\beta + \gamma
  = \alpha}C_\alpha^\beta (\nabla\Gamma^\beta
  Y)^\top(\partial_t^2 - \Delta)\Gamma^\gamma Ydy,
\end{eqnarray}
which gives that
\begin{eqnarray}\label{ES-1}
&&\frac{d}{dt}\int_{\Omega_\delta} \delta^{2(h-\frac 12)} e^{- q(\sigma)}\big(|\partial_t
  \Gamma^\alpha Y|^2 + |\nabla \Gamma^\alpha Y|^2\big)dy\\\nonumber
&&\quad +\ \sum_j\int_{\Omega_\delta} \delta^{2(h-\frac 12)} \frac{e^{- q(\sigma)}}{1 +
  \sigma^2}|(\omega_j\partial_t + \partial_j)
  \Gamma^\alpha Y|^2dy\\\nonumber
&&\lesssim \delta^{h-\frac 12} \mathcal{E}_{\kappa}^{\frac{1}{2}}
  \Big(\|\nabla\Gamma^\alpha p\|_{L^2(\Omega_\delta)} + \sum_{\beta + \gamma
  = \alpha}\|(\nabla
  \Gamma^\beta Y)^{\top}(\partial_t^2 - \Delta)\Gamma^\gamma
  Y\|_{L^2(\Omega_\delta)}\Big).
\end{eqnarray}
Thus we are left to deal with the the last line of \eqref{ES-1}.
We will perform the estimate in two cases: $t\leq 1$ and $t\geq 1$.
In the first case, we should be careful of the double derivative $\partial_3^2$.
In the second case, we need to earn enough decay in time.

\textbf{Case of $t\leq 1:$}
In this case, we need to use the Sobolev imbedding in thin domain.
Since $\kappa\ge 15$, $|\alpha|\leq \kappa-1$. By Lemma \ref{SNB-1}, there holds
\begin{align*}
\delta^{h-\frac 12} \|\nabla\Gamma^\alpha p\|_{L^2(\Omega_\delta)}\lesssim \mathcal{E}_{\kappa}^{\frac12} (\mathcal{E}_{\kappa-3}^{\frac12}+\mathcal{F}_{\kappa-3}^{\frac12}).
\end{align*}
On the other hand, by Lemma \ref{Sob2} and Lemma \ref{SNB-1}, there holds
\begin{align*}
& \delta^{h-\frac 12}  \sum_{\beta + \gamma = \alpha}\|(\nabla
  \Gamma^\beta Y)^{\top}(\partial_t^2 - \Delta)\Gamma^\gamma
  Y\|_{L^2(\Omega_\delta)}\\
&\lesssim \delta^{h-\frac 12}  \sum_{\beta + \gamma
  = \alpha,\ |\beta|\geq |\gamma|}\|(\nabla
  \Gamma^\beta Y)^{\top}(\partial_t^2 - \Delta)\Gamma^\gamma Y\|_{L^2(\Omega_\delta)}\\\nonumber
&\quad  +\delta^{h-\frac 12}  \sum_{\beta + \gamma
  = \alpha,\ |\beta|\leq |\gamma|}\|(\nabla
  \Gamma^\beta Y)^{\top}(\partial_t^2 - \Delta)\Gamma^\gamma
  Y\|_{L^2(\Omega_\delta)}\\\nonumber
&\lesssim   \sum_{\ |\gamma|\leq  [\alpha/2]+3}
\mathcal{E}_{|\alpha|+1}^{\frac12}
  \delta^{m-\frac 12}\|(\partial_t^2 - \Delta)\Gamma^\gamma Y\|_{L^2(\Omega_\delta)}\\\nonumber
&\quad  + \sum_{|\gamma|\leq |\alpha|}
\mathcal{E}_{[\alpha/2]+4}^{\frac12}
  \delta^{m-\frac 12} \|(\partial_t^2 - \Delta)\Gamma^\gamma Y\|_{L^2(\Omega_\delta)} \\\nonumber
&\lesssim
\mathcal{E}_{\kappa}^{\frac12} (\mathcal{E}_{\kappa-3}+\mathcal{F}_{\kappa-3})
+\mathcal{E}_{\kappa}^{\frac12} (\mathcal{E}_{\kappa-3}^{\frac12}+\mathcal{F}_{\kappa-3}^{\frac12})
\lesssim
\mathcal{E}_{\kappa}^{\frac12} (\mathcal{E}_{\kappa-3}^{\frac12}+\mathcal{F}_{\kappa-3}^{\frac12}).
\end{align*}
\textbf{Case of $t\geq 1:$}
In this case, we need to use the ghost weight method and a refined organization of the pressure.

First of all, by Lemma \ref{Sob2} and
Lemma \ref{WE-3},
there holds
\begin{align}\label{K-2}
& \delta^{h-\frac 12}  \sum_{\beta + \gamma= \alpha,\ \gamma\neq \alpha}\|(\nabla
  \Gamma^\beta Y)^{\top}(\partial_t^2 - \Delta)\Gamma^\gamma
  Y\|_{L^2(\Omega_\delta)}\\\nonumber
&\lesssim   \sum_{\ |\gamma|\leq  [\alpha/2]+3}
\mathcal{E}_{|\alpha|+1}^{\frac12}
  \delta^{m-\frac 12}\|(\partial_t^2 - \Delta)\Gamma^\gamma Y\|_{L^2(\Omega_\delta)}\\\nonumber
&\quad  + \sum_{|\gamma|\leq |\alpha|-1}
\mathcal{E}_{[\alpha/2]+4}^{\frac12}
  \delta^{m-\frac 12} \|(\partial_t^2 - \Delta)\Gamma^\gamma Y\|_{L^2(\Omega_\delta)} \\\nonumber
&\lesssim t^{-1} \mathcal{E}_{\kappa}^{\frac 12}\mathcal{E}_{\kappa-3}^{\frac 12}
(\mathcal{E}_{\kappa-3}^{\frac 12}+\mathcal{F}_{\kappa-3}^{\frac 12}).
\end{align}
We are left to estimate
\begin{align}\label{K-1}
& \delta^{h-\frac 12} \|(\nabla Y)^{\top}(\partial_t^2-\Delta)\Gamma^\alpha Y\|_{L^2(\Omega_\delta)}
+\delta^{h-\frac 12} \|\nabla\Gamma^\alpha p\|_{L^2(\Omega_\delta)}\\[-4mm]\nonumber\\\nonumber
&\lesssim
\delta^{h-\frac 12} \|\nabla Y\|_{L^\infty(\Omega_\delta)} \|(\partial_t^2-\Delta)\Gamma^\alpha Y\|_{L^2(\Omega_\delta)}
+\delta^{h-\frac 12} \|\nabla\Gamma^\alpha p\|_{L^2(\Omega_\delta)}\\[-4mm]\nonumber\\\nonumber
&\lesssim
\delta^{h-\frac 12} \|(\partial_t^2-\Delta)\Gamma^\alpha Y\|_{L^2(\Omega_\delta)}
+\delta^{h-\frac 12} \|\nabla\Gamma^\alpha p\|_{L^2(\Omega_\delta)}.
\end{align}
In terms of Lemma \ref{SN-1}, the above \eqref{K-1} are bounded by
\begin{align*}
\delta^{h-\frac 12} \|(\partial_t^2-\Delta)\Gamma^\alpha Y\|_{L^2(\Omega_\delta)}
+\delta^{h-\frac 12} \|\nabla\Gamma^\alpha p\|_{L^2(\Omega_\delta)}
\lesssim \Pi(\alpha,2).
\end{align*}
Hence we need to take care of $\Pi_1$, $\Pi_2$, $\Pi_3$, $\Pi_4$ and $\Pi_5$
in Lemma \ref{SN-1}.
The estimate of $\Pi_1$ is already done in \eqref{K-2}
\begin{align}
\Pi_1 \lesssim t^{-1} \mathcal{E}_{\kappa}^{\frac 12}
(\mathcal{E}_{\kappa-3}^{\frac 12}+\mathcal{F}_{\kappa-3}^{\frac 12})
+t^{-\frac 32} \mathcal{E}_{\kappa}^{\frac 12}
(\mathcal{E}_{\kappa-3}^{\frac 12}+\mathcal{F}_{\kappa-3}^{\frac 12}).
\end{align}
The estimate of $\Pi_2$ and $\Pi_3$ are the same. We only present the estimate for the former one.
For the expression inside $\Pi_2$, they are organized as follows to show the null structure
\begin{align*}
&(\partial_i, \partial_j)
\partial_t \Gamma^\beta Y^i \partial_t\Gamma^\gamma Y^j
- (\partial_i, \partial_j) \partial_k\Gamma^\beta Y^i \partial_k \Gamma^\gamma Y^j\\
&=(\omega_k(\omega_k\partial_t+\partial_k))(\partial_i, \partial_j)
\Gamma^\beta Y^i \partial_t\Gamma^\gamma Y^j
- (\partial_i, \partial_j) \partial_k  \Gamma^\beta Y^i (\omega_k\partial_t+\partial_k)\Gamma^\gamma Y^j.
\end{align*}
Note that the null condition is present. Thus
\begin{eqnarray*}
&&\sum_{\beta + \gamma =\alpha,\ |\beta|\leq |\gamma|}\Pi_2  \\\nonumber
&&\lesssim \sum_{\beta + \gamma =\alpha,\ |\beta|\leq |\gamma|}
\delta^{h-\frac12}\| \nabla(- \Delta)^{-1} (-\partial_j,\partial_i)  \cdot
\big((\omega_k(\omega_k\partial_t+\partial_k))(\partial_i, \partial_j)
\Gamma^\beta Y^i \partial_t\Gamma^\gamma Y^j\big)\|_{L^2(\Omega_\delta)}  \\ \nonumber
&&\quad+\sum_{\beta + \gamma =\alpha, |\beta|\leq |\gamma|}
\delta^{h-\frac12}\| \nabla(- \Delta)^{-1} (-\partial_j,\partial_i)  \cdot
\big((\partial_i, \partial_j) \partial_k\Gamma^\beta Y^i (\omega_k\partial_t+\partial_k) \Gamma^\gamma Y^j \big)
\|_{L^2(\Omega_\delta)}  \\ \nonumber
&&\lesssim \delta^{h-\frac12}\sum_{k}\sum_{\beta + \gamma =\alpha,\ |\beta|\leq |\gamma|}
\big\|
|(\omega_k\partial_t+\partial_k)\nabla \Gamma^\beta Y| \cdot
|\partial_t\Gamma^\gamma Y| \big\|_{L^2(\Omega_\delta)}  \\ \nonumber
&&\quad +\delta^{h-\frac12}\sum_{k}\sum_{\beta + \gamma =\alpha,\ |\beta|\leq |\gamma|}
\big\| |\nabla^2\Gamma^\beta Y|\cdot |(\omega_k\partial_t+\partial_k) \Gamma^\gamma Y|
\big\|_{L^2(\Omega_\delta)}.
\end{eqnarray*}
Note the index $\alpha=(h,a),\ \beta=(l,b),\ \gamma=(m,c)$, $h=l+m$.
The above is further bounded by
\begin{align*}
& \delta^{h-\frac12}\sum_{\beta + \gamma =\alpha,\ |\beta|\leq |\gamma|}
\big\|
|(\omega_k\partial_t+\partial_k)\nabla \Gamma^\beta Y| \cdot
|\partial_t\Gamma^\gamma Y| \big\|_{L^2(\Omega_\delta)}  \\ \nonumber
&+\delta^{h-\frac12}\sum_{\beta + \gamma =\alpha,\ |\beta|\leq |\gamma|}
\big\| |\nabla^2\Gamma^\beta Y|\cdot |(\omega_k\partial_t+\partial_k) \Gamma^\gamma Y|
\big\|_{L^2(\Omega_\delta)}  \\ \nonumber
&\leq  \sum_{|\beta|\leq [\alpha/2]}
\delta^{l}\| \partial\nabla \Gamma^\beta Y\|_{L^\infty(\Omega_\delta(r\leq t/2))}
\mathcal{E}_{|\alpha|+1}^{\frac12}
 \\ \nonumber
&\quad+\sum_{|\beta|\leq [\alpha/2]}
\delta^{l}\| (\omega_k\partial_t+\partial_k)\nabla \Gamma^\beta Y\|_{L^\infty(\Omega_\delta(r\geq t/2))}
\mathcal{E}_{|\alpha|+1}^{\frac12}
 \\ \nonumber
&\quad+\sum_{\beta + \gamma =\alpha,\ |\beta|\leq |\gamma|}
\|\delta^{l}\nabla^2\Gamma^\beta Y\cdot
\delta^{m-\frac12}(\omega_k\partial_t+\partial_k) \Gamma^\gamma Y\|_{L^2(\Omega_\delta(r\geq t/2))}.
\end{align*}
For $|\beta|\leq [\alpha/2]$, by Lemma \ref{KS-A} and Lemma \ref{SN-3}, there holds
\begin{align*}
\delta^{l}\| \partial\nabla \Gamma^\beta Y\|_{L^\infty(\Omega_\delta(r\leq t/2))}
 \leq t^{-1} \mathcal{X}_{|\beta|+5}^{\frac12}
\lesssim t^{-1} (\mathcal{E}_{\kappa-3}^{\frac12}+\mathcal{F}_{\kappa-3}^{\frac12}) .
\end{align*}
By Lemma \ref{Sob2}, Lemma \ref{GoodDeri}, Lemma \ref{SN-1} and Lemma \ref{WE-2},
there holds
\begin{align*}
&\sum_{|\beta|\leq [\alpha/2]} \delta^{l}\| (\omega_k\partial_t+\partial_k)\nabla \Gamma^\beta Y\|_{L^\infty(\Omega_\delta(r\geq t/2))} \\
&\leq \sum_{|\beta|\leq [\alpha/2]+3} \delta^{l-\frac12}\| (\omega_k\partial_t+\partial_k)\nabla \Gamma^\beta Y\|_{L^2(\Omega_\delta(r\geq t/2))} \\
&\quad +t^{-1} \sum_{|\beta|\leq [\alpha/2]+2} \delta^{l-\frac12}\| \partial_t\nabla \Gamma^\beta Y\|_{L^2(\Omega_\delta(r\geq t/2))}\\
&\lesssim
t^{-1} \mathcal{E}_{\kappa-3}^{\frac12} .
\end{align*}
On the other hand,
if $r\geq t/2$, since $r\leq r_v+1$, $t\geq 1$, $\delta\leq 1$, then
\begin{align*}
t/2\leq r\leq r_v+1.
\end{align*}
By Lemma \ref{Sob2}, Lemma \ref{KS-A} and Lemma \ref{WE-2}, one has
\begin{align*}
&\sum_{|\beta|\leq [\alpha/2]}
\delta^{l}\| \langle t-r\rangle\nabla^2\Gamma^\beta Y\|_{L^\infty(\Omega_\delta(r\geq t/2))} \\\nonumber
&\lesssim \sum_{|\beta|\leq [\alpha/2]}
\delta^{l}\| (t-r)\nabla^2\Gamma^\beta Y\|_{L^\infty(\Omega_\delta(r\geq t/2))}
+\sum_{|\beta|\leq [\alpha/2]}
\delta^{l}\| \nabla^2\Gamma^\beta Y\|_{L^\infty(\Omega_\delta(r\geq t/2))}\\\nonumber
&\lesssim (r_v+1)^{-\frac12} (\mathcal{X}_{[\alpha/2]+5}^{\frac12}+\mathcal{E}_{[\alpha/2]+5}^{\frac12})\\
&\lesssim t^{-\frac12}(\mathcal{E}_{\kappa-3}^{\frac12}+\mathcal{F}_{\kappa-3}^{\frac12}).
\end{align*}

Hence we have,
\begin{align*}
&\sum_{\beta + \gamma =\alpha,\ |\beta|\leq |\gamma|}
\delta^{h-\frac12}\big\| |\nabla^2\Gamma^\beta Y|
\cdot |(\omega_k\partial_t+\partial_k) \Gamma^\gamma Y|\big \|_{L^2(r\geq t/2)}  \\\nonumber
&\lesssim\sum_{\beta + \gamma =\alpha,\ |\beta|\leq |\gamma|}
\delta^{l}\| \langle t-r\rangle\nabla^2\Gamma^\beta Y\|_{L^\infty(r\geq t/2)}
\delta^{m-\frac12}\big\|\frac{(\omega_k\partial_t+\partial_k) \Gamma^\gamma Y}
{\langle t-r\rangle}\big\|_{L^2(\Omega_\delta)} \\\nonumber
&\lesssim t^{-\frac12}(\mathcal{E}_{\kappa-3}^{\frac12}+\mathcal{F}_{\kappa-3}^{\frac12})
\cdot \sum_j\sum_{|\alpha|\leq \kappa-1}
\delta^{h-\frac12}\big\|\frac{(\omega_j\partial_t+\partial_j) \Gamma^\alpha Y}{
\langle t-r\rangle}\big\|_{L^2(\Omega_\delta)}.
\end{align*}

Next we estimate $\Pi_4$ and $\Pi_5$. As has been computed in Lemma \ref{SN-3},
there holds
\begin{align*}
\Pi_4+\Pi_5 &\lesssim
\sum_{|\alpha|/2\leq |\beta|<|\alpha|}
 \delta^{l-\frac12} \|
(\partial_t^2 - \Delta) \Gamma^\beta Y\|_{L^2(\Omega_\delta)}
\mathcal{E}_{[\alpha/2]+4}
\\ \nonumber
&\quad+\sum_{|\beta|\leq [\alpha/2]+3}
 \delta^{l-\frac12}\|
(\partial_t^2 - \Delta) \Gamma^\beta Y\|_{L^2(\Omega_\delta)}
\mathcal{E}_{|\alpha|+1}^{\frac12} \mathcal{E}_{[\alpha/2]+4}^{\frac12}\\ \nonumber
&\quad+t^{-1}\mathcal{E}_{|\alpha|+1}^{\frac12} (\mathcal{E}_{[\alpha/2]+4}^{\frac12}
+ \mathcal{X}_{[\alpha/2]+4}^{\frac12}
+\mathcal{F}_{[\alpha/2]+4}^{\frac12} ) \mathcal{E}_{[\alpha/2]+4}^{\frac12}.
\end{align*}
Since $\kappa\geq 15$, $|\alpha|\leq \kappa-1$.
By Lemma \ref{WE-2} and Lemma \ref{WE-3}, the above are further bounded by
\begin{align*}
\Pi_4+\Pi_5
&\lesssim  t^{-1} \mathcal{E}_{\kappa}^{\frac 12}
(\mathcal{E}_{\kappa-3}^{\frac 12}+\mathcal{F}_{\kappa-3}^{\frac 12}).
\end{align*}

Finally we collect all the estimate to obtain
\begin{align*}
&\frac{d}{dt}\sum_{|\alpha|\leq \kappa-1}\int_{\Omega_\delta} \delta^{2(h-\frac 12)} e^{- q(\sigma)}\big(|\partial_t
  \Gamma^\alpha Y|^2 + |\nabla \Gamma^\alpha Y|^2\big)dy\\\nonumber
&\quad +\ \sum_j\sum_{|\alpha|\leq \kappa-1}\int_{\Omega_\delta} \delta^{2(h-\frac 12)} \frac{e^{- q(\sigma)}}{1 +
  \sigma^2}|(\omega_j\partial_t + \partial_j)
  \Gamma^\alpha Y|^2dy\\\nonumber
&\lesssim \langle t\rangle^{-1} \mathcal{E}_{\kappa}
( \mathcal{E}_{\kappa-3}^{\frac12}+\mathcal{F}_{\kappa-3}^{\frac12} )\\\nonumber
&\quad + \langle t\rangle^{-\frac12} \mathcal{E}_{\kappa}^{\frac12}
(\mathcal{E}_{\kappa-3}^{\frac12}+\mathcal{F}_{\kappa-3}^{\frac12})
\sum_j\sum_{|\alpha|\leq \kappa-1}\delta^{(h-\frac 12)}
  \big\|\frac{(\omega_j\partial_t + \partial_j)\Gamma^\alpha Y}{\langle t-r\rangle}\big\|_{L^2(\Omega_\delta)}.
\end{align*}
This would finally gives
\begin{align} \label{ES-2}
&\frac{d}{dt}\sum_{|\alpha|\leq \kappa-1}\int_{\Omega_\delta} \delta^{2(h-\frac 12)} e^{- q(\sigma)}\big(|\partial_t
  \Gamma^\alpha Y|^2 + |\nabla \Gamma^\alpha Y|^2\big)dy\\\nonumber
&\quad +\ \sum_j\sum_{|\alpha|\leq \kappa-1}\int_{\Omega_\delta} \delta^{2(h-\frac 12)} \frac{e^{- q(\sigma)}}{1 +
  \sigma^2}|(\omega_j\partial_t + \partial_j)
  \Gamma^\alpha Y|^2dy\\\nonumber
&\lesssim \langle t\rangle^{-1} \mathcal{E}_{\kappa}
( \mathcal{E}_{\kappa-3}^{\frac12}+\mathcal{F}_{\kappa-3}^{\frac12} ).
\end{align}
This gives the first differential inequality at the end of Section
\ref{Equations}.

\section{Lower-order energy estimate}\label{LEE}
\subsection{Energy estimate for $\mathcal{E}_{\kappa-3}$}
Before the lower order energy estimate, we state  that the energy is decided by
its curl part.
\begin{lem}
Let $\kappa \geq 15$. There exists $0< \eta\ll 1$ such that
if $\mathcal{E}_{\kappa-3}+\mathcal{F}_{\kappa-3} \leq \eta$.
There holds
\begin{align*}
\mathcal{E}_{\kappa-3}\sim \sum_{|\alpha|\leq \kappa-4}
\delta^{2(h-\frac12)}\| \partial(-\Delta)^{-\frac12}\nabla\times \Gamma^\alpha Y\|_{L^2(\Omega_\delta)}^2.
\end{align*}
\end{lem}
\begin{proof}
We begin with the following div-curl identity
\begin{align}\label{DC-1}
\delta^{2(h-\frac12)}\| \partial  \Gamma^\alpha Y\|_{L^2(\Omega_\delta)}^2
=\delta^{2(h-\frac12)}\| \partial(-\Delta)^{-\frac12}\nabla\cdot \Gamma^\alpha Y\|_{L^2(\Omega_\delta)}^2
+\delta^{2(h-\frac12)}\| \partial(-\Delta)^{-\frac12}\nabla\times \Gamma^\alpha Y\|_{L^2(\Omega_\delta)}^2.
\end{align}
Now let us calculate $\partial\nabla\cdot\Gamma^\alpha Y$
in the first term on the right hand side of \eqref{DC-1}.
By \eqref{Struc-2}, one has
\begin{align*}
\partial\nabla\cdot \Gamma^\alpha Y
 =& -\frac12 \sum_{\beta + \gamma =\alpha} C_{\alpha}^\beta
\partial\big(\partial_i\Gamma^\beta Y^i \partial_j\Gamma^\gamma Y^j -
  \partial_i\Gamma^\gamma Y^j\partial_j\Gamma^\beta Y^i\big)\\
  &-\sum_{\beta+\gamma+\iota=\alpha}C_\alpha  ^\beta C_{\alpha-\beta}  ^\gamma
 \partial \left|
\begin{matrix}
 \partial_1\Gamma^\beta Y^1 &\partial_2 \Gamma^\beta Y^1 & \partial_3 \Gamma^\beta Y^1\\
 \partial_1\Gamma ^\gamma Y^2& \partial_2 \Gamma ^\gamma Y^2& \partial_3\Gamma ^\gamma Y^2\\
\partial_1\Gamma ^\iota Y^3&\partial_2\Gamma^\iota Y^3&\partial_3\Gamma^\iota Y^3
\end{matrix} \right|.
\end{align*}
We compute
\begin{align*}
&-\partial\big(\partial_i\Gamma^\beta Y^i \partial_j\Gamma^\gamma Y^j -
  \partial_i\Gamma^\gamma Y^j\partial_j\Gamma^\beta Y^i\big) \\
&=\partial(\partial_i, \partial_j) \Gamma^\beta Y^i \cdot (-\partial_j,\partial_i)\Gamma^\gamma Y^j
+(\partial_i, \partial_j) \Gamma^\beta Y^i \cdot \partial(-\partial_j,\partial_i)\Gamma^\gamma Y^j \\
&=(\partial_i, \partial_j) \cdot\big(  \partial\Gamma^\beta Y^i (-\partial_j,\partial_i)\Gamma^\gamma Y^j\big)
+(-\partial_j,\partial_i) \cdot\big((\partial_i, \partial_j) \Gamma^\beta Y^i \partial\Gamma^\gamma Y^j\big),
\end{align*}
and
\begin{align*}
&\partial\left|
\begin{matrix}
 \partial_1\Gamma^\beta Y^1 &\partial_2 \Gamma^\beta Y^1 & \partial_3 \Gamma^\beta Y^1\\
 \partial_1\Gamma ^\gamma Y^2& \partial_2 \Gamma ^\gamma Y^2& \partial_3\Gamma ^\gamma Y^2\\
\partial_1\Gamma ^\iota Y^3&\partial_2\Gamma^\iota Y^3&\partial_3\Gamma^\iota Y^3
\end{matrix} \right| \\
&=\partial\big(\nabla \Gamma^\beta Y^1\cdot \nabla\times (\Gamma ^\gamma Y^2, \Gamma ^\iota Y^3) \big) \\
&=\partial\nabla \Gamma^\beta Y^1\cdot \nabla\times (\Gamma ^\gamma Y^2, \Gamma ^\iota Y^3)
+\nabla \Gamma^\beta Y^1\cdot \nabla\times (\partial\Gamma ^\gamma Y^2, \Gamma ^\iota Y^3)  \\
&\quad +\nabla \Gamma^\beta Y^1\cdot \nabla\times (\Gamma ^\gamma Y^2, \partial\Gamma ^\iota Y^3) \\
&=\nabla\cdot\big(  \partial \Gamma^\beta Y^1 \nabla\times (\Gamma ^\gamma Y^2, \Gamma ^\iota Y^3)\big)
-\nabla\cdot\big( \partial\Gamma ^\gamma Y^2\nabla\times (\Gamma^\beta Y^1, \Gamma ^\iota Y^3) \big) \\
&\quad +\nabla\cdot\big(\partial\Gamma ^\iota Y^3
\nabla\times (\Gamma^\beta Y^1,\Gamma ^\gamma Y^2 ) \big).
\end{align*}
Consequently, by Lemma \ref{Sob2}, one has
\begin{align*}
&\delta^{2(h-\frac12)} \| \partial(-\Delta)^{-\frac12}\nabla\cdot \Gamma^\alpha Y\|_{L^2(\Omega_\delta)}^2 \\
&\lesssim \delta^{2(h-\frac12)}\sum_{\beta + \gamma =\alpha}
\|(-\Delta)^{-\frac12}(\partial_i, \partial_j) \cdot\big(  \partial\Gamma^\beta Y^i (-\partial_j,\partial_i)\Gamma^\gamma Y^j\big) \|_{L^2(\Omega_\delta)}^2 \\
&\quad +\delta^{2(h-\frac12)}\sum_{\beta + \gamma =\alpha} \| (-\Delta)^{-\frac12}(-\partial_j,\partial_i) \cdot\big((\partial_i, \partial_j) \Gamma^\beta Y^i \partial\Gamma^\gamma Y^j\big)\|_{L^2(\Omega_\delta)}^2\\
&\quad +\delta^{2(h-\frac12)}\sum_{\beta+\gamma+\iota=\alpha} \| (-\Delta)^{-\frac12} \nabla\cdot\big(  \partial \Gamma^\beta Y^1 \nabla\times (\Gamma ^\gamma Y^2, \Gamma ^\iota Y^3)\big) \|_{L^2(\Omega_\delta)}^2\\
&\quad +\delta^{2(h-\frac12)}\sum_{\beta+\gamma+\iota=\alpha} \| (-\Delta)^{-\frac12}\nabla\cdot\big( \partial\Gamma ^\gamma Y^2\nabla\times (\Gamma^\beta Y^1, \Gamma ^\iota Y^3) \big) \|_{L^2(\Omega_\delta)}^2\\
&\quad +\delta^{2(h-\frac12)}\sum_{\beta+\gamma+\iota=\alpha} \| (-\Delta)^{-\frac12} \nabla\cdot\big(\partial\Gamma ^\iota Y^3
\nabla\times (\Gamma^\beta Y^1,\Gamma ^\gamma Y^2 ) \big)\|_{L^2(\Omega_\delta)}^2\\
&\lesssim \delta^{2(h-\frac12)}\sum_{\beta + \gamma =\alpha}
\big\|  |\partial\Gamma^\beta Y| |\nabla\Gamma^\gamma Y| \big\|_{L^2(\Omega_\delta)}^2
 +\delta^{2(h-\frac12)}\sum_{\beta+\gamma+\iota=\alpha}
\big\|  |\partial \Gamma^\beta Y| |\nabla\Gamma ^\gamma Y| |\nabla\Gamma ^\iota Y| \big\|_{L^2(\Omega_\delta)}^2\\
&\lesssim \mathcal{E}_{\kappa-3}^2+\mathcal{E}^3_{\kappa-3}.
\end{align*}
Inserting the above inequality into \eqref{DC-1}, we have
\begin{align*}
\delta^{2(h-\frac12)}\| \partial  \Gamma^\alpha Y\|_{L^2(\Omega_\delta)}^2
\lesssim \mathcal{E}_{\kappa-3}^2+\mathcal{E}^3_{\kappa-3}
+\delta^{2(h-\frac12)}\| \partial(-\Delta)^{-\frac12}\nabla\times \Gamma^\alpha Y\|_{L^2(\Omega_\delta)}^2.
\end{align*}
For $\kappa\geq 15$, summing over $|\alpha|\leq \kappa-4$ yields
\begin{align*}
\mathcal{E}_{\kappa-3}
\lesssim \mathcal{E}_{\kappa-3}^2+\mathcal{E}^3_{\kappa-3}
+\sum_{|\alpha|\leq \kappa-4} \delta^{2(h-\frac12)}\| \partial(-\Delta)^{-\frac12}\nabla\times \Gamma^\alpha Y\|_{L^2(\Omega_\delta)}^2.
\end{align*}
Recall the assumption that $\mathcal{E}_{\kappa-3}+\mathcal{F}_{\kappa-3} \leq \eta\ll 1$.
This yields
\begin{align*}
\mathcal{E}_{\kappa-3}
\lesssim \sum_{|\alpha|\leq \kappa-4} \delta^{2(h-\frac12)}\| \partial(-\Delta)^{-\frac12}\nabla\times \Gamma^\alpha Y\|_{L^2(\Omega_\delta)}^2.
\end{align*}
On the other hand, by Lemma \ref{OP1} and Lemma \ref{OP2}, it's easy to see the reverse bound also holds.
Thus the lemma is proved.
\end{proof}

Now we perform the lower order energy estimate.
Let $\kappa \geq 15$ and $|\alpha| \leq \kappa - 4$.
Applying $(-\Delta)^{-\frac12}\nabla\times $ onto \eqref{Elasticity-D} and then
taking the $L^2$ inner product of the resulting equations with $\delta^{2(h-\frac 12)}\partial_t(-\Delta)^{-\frac12}\nabla\times\Gamma^\alpha Y$ and using integration by parts, we have
\begin{eqnarray}\nonumber
&&\frac{d}{dt}\int_{\Omega_\delta} \delta^{2(h-\frac 12)}
\big(|\partial_t (-\Delta)^{-\frac12}\nabla\times \Gamma^\alpha Y|^2
+ |\nabla (-\Delta)^{-\frac12}\nabla\times \Gamma^\alpha Y|^2\big)dy\\\nonumber
&&=
-\ 2\delta^{2(h-\frac 12)}\int_{\Omega_\delta}\partial_t
  (-\Delta)^{-\frac12}\nabla\times \Gamma^\alpha Y
  \cdot(-\Delta)^{-\frac12}\nabla\times \sum_{\beta + \gamma  = \alpha}C_\alpha^\beta (\nabla\Gamma^\beta
  Y)^\top(\partial_t^2 - \Delta)\Gamma^\gamma Ydy\\\nonumber
&&\lesssim \delta^{2(h-\frac 12)}
\| \partial_t
  (-\Delta)^{-\frac12}\nabla\times \Gamma^\alpha Y \|_{L^2(\Omega_\delta)}
\sum_{\beta + \gamma  = \alpha}\| (-\Delta)^{-\frac12}\nabla\times(\nabla\Gamma^\beta
  Y)^\top(\partial_t^2 - \Delta)\Gamma^\gamma Y \|_{L^2(\Omega_\delta)} \\\nonumber
&&\lesssim \delta^{2(h-\frac 12)}
\| \partial_t \Gamma^\alpha Y \|_{L^2(\Omega_\delta)}
\sum_{\beta + \gamma  = \alpha}\|(\nabla\Gamma^\beta
  Y)^\top(\partial_t^2 - \Delta)\Gamma^\gamma Y \|_{L^2(\Omega_\delta)} \\\nonumber
&&\lesssim \delta^{(h-\frac 12)}\mathcal{E}_{\kappa-3}^{\frac12}
\sum_{\beta + \gamma  = \alpha}\|(\nabla\Gamma^\beta
  Y)^\top(\partial_t^2 - \Delta)\Gamma^\gamma Y \|_{L^2(\Omega_\delta)}.
\end{eqnarray}
If $t\leq 1$, by Lemma \ref{Sob2} and Lemma \ref{SNB-1}, there holds
\begin{align*}
& \delta^{h-\frac 12}  \sum_{\beta + \gamma = \alpha}\|(\nabla
  \Gamma^\beta Y)^{\top}(\partial_t^2 - \Delta)\Gamma^\gamma
  Y\|_{L^2(\Omega_\delta)}\\
&\lesssim \delta^{h-\frac 12}  \sum_{\beta + \gamma
  = \alpha,\ |\beta|\geq |\gamma|}\|(\nabla
  \Gamma^\beta Y)^{\top}(\partial_t^2 - \Delta)\Gamma^\gamma Y\|_{L^2(\Omega_\delta)}\\\nonumber
&\quad  +\delta^{h-\frac 12}  \sum_{\beta + \gamma
  = \alpha,\ |\beta|\leq |\gamma|}\|(\nabla
  \Gamma^\beta Y)^{\top}(\partial_t^2 - \Delta)\Gamma^\gamma
  Y\|_{L^2(\Omega_\delta)}\\\nonumber
&\lesssim   \sum_{\ |\gamma|\leq  [\alpha/2]+3}
\mathcal{E}_{|\alpha|+1}^{\frac12}
  \delta^{m-\frac 12}\|(\partial_t^2 - \Delta)\Gamma^\gamma Y\|_{L^2(\Omega_\delta)}\\\nonumber
&\quad  + \sum_{|\gamma|\leq |\alpha|}
\mathcal{E}_{[\alpha/2]+4}^{\frac12}
  \delta^{m-\frac 12} \|(\partial_t^2 - \Delta)\Gamma^\gamma Y\|_{L^2(\Omega_\delta)} \\\nonumber
&\lesssim
\mathcal{E}_{\kappa-3}^{\frac12} (\mathcal{E}_{\kappa-3}+\mathcal{F}_{\kappa-3}).
\end{align*}
Next, if $t\geq 1$, we deduce that
\begin{align}\label{L-2}
&\delta^{(h-\frac 12)}\sum_{\beta + \gamma  = \alpha}
\|(\nabla\Gamma^\beta Y)^\top(\partial_t^2 - \Delta)\Gamma^\gamma Y\|_{L^2(\Omega_\delta)} \\ \nonumber
&\leq
\sum_{\beta + \gamma  = \alpha}
\delta^{l}\|\nabla\Gamma^\beta Y\|_{L^\infty(\Omega_\delta)}
\delta^{(m-\frac 12)} \|(\partial_t^2 - \Delta)\Gamma^\gamma Y \|_{L^2(\Omega_\delta)}.
\end{align}
By Lemma \ref{K-S-O}, one has
\begin{align*}
\delta^{l} \|\nabla\Gamma^\beta Y\|_{L^\infty(r\leq t/2)}
&\lesssim t^{-\frac12} (\mathcal{X}_{|\beta|+4}^{\frac12}+
\mathcal{E}_{|\beta|+4}^{\frac12}).
\end{align*}
On the other hand,
if $r\geq t/2$, since $r\leq r_v+1$, $t\geq 1$, $\delta\leq 1$, then
\begin{align*}
t/2\leq r\leq r_v+1.
\end{align*}
By Lemma \ref{Sob2} and Lemma \ref{KS-A}, one has
\begin{align*}
\delta^{l} \|\nabla\Gamma^\beta Y\|_{L^\infty(r\geq t/2)}
&\lesssim (1+r_v)^{-\frac12} (1+r_v^{\frac12})
\|\nabla\Gamma^\beta Y\|_{L^\infty(r\geq t/2)} \\
&\lesssim (1+r_v)^{-\frac12}
\mathcal{E}_{|\beta|+4}^{\frac12}
\lesssim t^{-\frac12}
\mathcal{E}_{|\beta|+4}^{\frac12}.
\end{align*}
Hence by Lemma \ref{WE-2}, \eqref{L-2} is bounded by
\begin{align*}
&\sum_{\beta + \gamma  = \alpha}
\delta^{l}\|\nabla\Gamma^\beta Y\|_{L^\infty(\Omega_\delta)}
\delta^{(m-\frac 12)} \|(\partial_t^2 - \Delta)\Gamma^\gamma Y \|_{L^2(\Omega_\delta)} \\\nonumber
&\lesssim \sum_{\beta + \gamma  = \alpha}
t^{-\frac12}(\mathcal{E}_{|\beta|+4}^{\frac12}+\mathcal{X}_{|\beta|+4}^{\frac12})
\cdot \delta^{(m-\frac 12)} \|(\partial_t^2 - \Delta)\Gamma^\gamma Y \|_{L^2(\Omega_\delta)}\\\nonumber
&\lesssim
t^{-\frac 32} \mathcal{E}_{\kappa}^{\frac12} (\mathcal{E}_{\kappa-3}^{\frac12}+\mathcal{F}_{\kappa-3}^{\frac12}).
\end{align*}

Finally we summarize that
\begin{align}\label{LOEE-1}
\frac{d}{dt} \sum_{|\alpha|\leq \kappa-4}
\delta^{2(h-\frac12)}\| \partial(-\Delta)^{-\frac12}\nabla\times \Gamma^\alpha Y\|_{L^2(\Omega_\delta)}^2
\lesssim
\langle t\rangle^{-\frac32}  \mathcal{E}_{\kappa-3}^{\frac 12} \mathcal{E}_{\kappa}^{\frac 12}
(\mathcal{E}_{\kappa-3}^{\frac12}+\mathcal{F}_{\kappa-3}^{\frac12}).
\end{align}
Then we can replace all $\mathcal{E}_{\kappa - 3}$ appeared
throughout this paper by
$$\sum_{|\alpha|\leq \kappa-4}
\delta^{2(h-\frac12)}\| \partial(-\Delta)^{-\frac12}\nabla\times \Gamma^\alpha Y\|_{L^2(\Omega_\delta)}^2$$
without changing the final result. Then \eqref{LOEE-1} gives the
second differential inequality at the end of Section
\ref{Equations}.

\subsection{Estimate of the scaled Energy $\mathcal{F}_{\kappa-3}$ }
\begin{lem}
Let $\kappa \geq 15$. There exists $0<\eta\ll 1$ such that if $\mathcal{E}_{\kappa-3}+\mathcal{F}_{\kappa-3} \leq \eta$.
Then there holds
\begin{align*}
\mathcal{F}_{\kappa-3}\sim \sum_{|\alpha|\leq \kappa-5}
\delta^{2(h-\frac12)} \| \partial_3\partial(-\Delta)^{-\frac12}\nabla\times \Gamma^\alpha Y\|_{L^2(\Omega_\delta)}.
\end{align*}
\end{lem}
\begin{proof}
We begin by the following div-curl identity.
\begin{align}\label{DC-5}
&\delta^{2(h-\frac12)}\| \partial_3\partial  \Gamma^\alpha Y\|_{L^2(\Omega_\delta)}^2\\\nonumber
&=\delta^{2(h-\frac12)}\| \partial_3\partial(-\Delta)^{-\frac12}\nabla\cdot \Gamma^\alpha Y\|_{L^2(\Omega_\delta)}^2
+\delta^{2(h-\frac12)}\| \partial_3\partial(-\Delta)^{-\frac12}\nabla\times \Gamma^\alpha Y\|_{L^2(\Omega_\delta)}^2.
\end{align}
Now let us calculate $\partial_3\partial\nabla\cdot\Gamma^\alpha Y$
in \eqref{DC-5}.
By \eqref{Struc-2}, one has
\begin{align*}
\partial_3\partial\nabla\cdot \Gamma^\alpha Y
 =& -\frac12 \sum_{\beta + \gamma =\alpha} C_{\alpha}^\beta
\partial_3\partial\big(\partial_i\Gamma^\beta Y^i \partial_j\Gamma^\gamma Y^j -
  \partial_i\Gamma^\gamma Y^j\partial_j\Gamma^\beta Y^i\big)\\
  &-\sum_{\beta+\gamma+\iota=\alpha}C_\alpha  ^\beta C_{\alpha-\beta}  ^\gamma
  \partial_3\partial\left|
\begin{matrix}
 \partial_1\Gamma^\beta Y^1 &\partial_2 \Gamma^\beta Y^1 & \partial_3 \Gamma^\beta Y^1\\
 \partial_1\Gamma ^\gamma Y^2& \partial_2 \Gamma ^\gamma Y^2& \partial_3\Gamma ^\gamma Y^2\\
\partial_1\Gamma ^\iota Y^3&\partial_2\Gamma^\iota Y^3&\partial_3\Gamma^\iota Y^3
\end{matrix} \right|.
\end{align*}
Note that
\begin{align*}
&-\partial_3\partial\big(\partial_i\Gamma^\beta Y^i \partial_j\Gamma^\gamma Y^j -
  \partial_i\Gamma^\gamma Y^j\partial_j\Gamma^\beta Y^i\big) \\
&=(\partial_i, \partial_j) \cdot\big(  \partial_3\partial\Gamma^\beta Y^i (-\partial_j,\partial_i)\Gamma^\gamma Y^j\big)
+(\partial_i, \partial_j) \cdot\big(  \partial\Gamma^\beta Y^i (-\partial_j,\partial_i)\partial_3\Gamma^\gamma Y^j\big)
 \\
&\quad+(-\partial_j,\partial_i) \cdot\big(\partial_3(\partial_i, \partial_j) \Gamma^\beta Y^i \partial\Gamma^\gamma Y^j\big)
+(-\partial_j,\partial_i) \cdot\big((\partial_i, \partial_j) \Gamma^\beta Y^i \partial_3\partial\Gamma^\gamma Y^j\big)
\end{align*}
and
\begin{align*}
&\partial_3\partial\left|
\begin{matrix}
 \partial_1\Gamma^\beta Y^1 &\partial_2 \Gamma^\beta Y^1 & \partial_3 \Gamma^\beta Y^1\\
 \partial_1\Gamma ^\gamma Y^2& \partial_2 \Gamma ^\gamma Y^2& \partial_3\Gamma ^\gamma Y^2\\
\partial_1\Gamma ^\iota Y^3&\partial_2\Gamma^\iota Y^3&\partial_3\Gamma^\iota Y^3
\end{matrix} \right| \\
&=\nabla\cdot\big(  \partial_3\partial \Gamma^\beta Y^1 \nabla\times (\Gamma ^\gamma Y^2, \Gamma ^\iota Y^3)\big)
+\nabla\cdot\big(  \partial \Gamma^\beta Y^1 \partial_3\nabla\times (\Gamma ^\gamma Y^2, \Gamma ^\iota Y^3)\big) \\
&\quad-\nabla\cdot\big( \partial_3\partial\Gamma ^\gamma Y^2\nabla\times (\Gamma^\beta Y^1, \Gamma ^\iota Y^3) \big)
-\nabla\cdot\big( \partial\Gamma ^\gamma Y^2\partial_3\nabla\times (\Gamma^\beta Y^1, \Gamma ^\iota Y^3) \big) \\
&\quad +\nabla\cdot\big(\partial_3\partial\Gamma ^\iota Y^3
\nabla\times (\Gamma^\beta Y^1,\Gamma ^\gamma Y^2 ) \big)
+\nabla\cdot\big(\partial\Gamma ^\iota Y^3
\partial_3\nabla\times (\Gamma^\beta Y^1,\Gamma ^\gamma Y^2 ) \big).
\end{align*}
Consequently, by Lemma \ref{Sob2}, one has
\begin{align*}
&\delta^{2(h-\frac12)} \| \partial_3\partial(-\Delta)^{-\frac12}\nabla\cdot \Gamma^\alpha Y\|_{L^2(\Omega_\delta)}^2 \\
&\lesssim \delta^{2(h-\frac12)}\sum_{\beta + \gamma =\alpha}
\|  \partial_3\partial\Gamma^\beta Y \partial\Gamma^\gamma Y \|_{L^2(\Omega_\delta)}^2
 +\delta^{2(h-\frac12)}\sum_{\beta+\gamma+\iota=\alpha}
\|  \partial_3\partial \Gamma^\beta Y\partial\Gamma ^\gamma Y\partial\Gamma ^\iota Y \|_{L^2(\Omega_\delta)}^2\\
&\lesssim \mathcal{F}_{\kappa-3}\mathcal{E}_{\kappa-3}+\mathcal{F}_{\kappa-3}\mathcal{E}^2_{\kappa-3}.
\end{align*}
Inserting the above inequality into \eqref{DC-5}
and summing over $|\alpha|\leq \kappa-5$ for $\kappa\geq 15$, one has
\begin{align*}
\mathcal{F}_{\kappa-3}
\lesssim \mathcal{F}_{\kappa-3}(\mathcal{E}_{\kappa-3}+\mathcal{E}^2_{\kappa-3})
+\sum_{|\alpha|\leq \kappa-5} \delta^{2(h-\frac12)}\| \partial_3\partial(-\Delta)^{-\frac12}\nabla\times \Gamma^\alpha Y\|_{L^2(\Omega_\delta)}^2.
\end{align*}
Note the assumption that $\mathcal{E}_{\kappa-3}+\mathcal{F}_{\kappa-5} \leq \eta\ll 1$
 yields
\begin{align*}
\mathcal{F}_{\kappa-3}
\lesssim \sum_{|\alpha|\leq \kappa-5} \delta^{2(h-\frac12)}\| \partial_3\partial(-\Delta)^{-\frac12}\nabla\times \Gamma^\alpha Y\|_{L^2(\Omega_\delta)}^2.
\end{align*}
On the other hand, by Lemma \ref{OP1} and Lemma \ref{OP2}, the reverse bound also holds.
Thus the lemma is proved.
\end{proof}

Now let us perform the scaled energy estimate for $\mathcal{F}_{\kappa-3}$.
Let $\kappa \geq 15$ and $|\alpha| \leq \kappa - 5$.
Applying $\partial_3(-\Delta)^{-\frac12}\nabla\times $ onto \eqref{Elasticity-D} and then
taking the $L^2$ inner product
of the resulting equations with $\delta^{2(h-\frac 12)}\partial_3\partial_t(-\Delta)^{-\frac12}\nabla\times\Gamma^\alpha
Y$. Using integration by parts, we have
\begin{eqnarray}\nonumber
&&\frac{d}{dt}\int_{\Omega_\delta} \delta^{2(h-\frac 12)}
\big(|\partial_3\partial_t(-\Delta)^{-\frac12}\nabla\times\Gamma^\alpha
Y|^2 + |\partial_3\nabla(-\Delta)^{-\frac12}\nabla\times\Gamma^\alpha
Y|^2\big)dy \\ \nonumber
&&= -\ 2\delta^{2(h-\frac 12)}\int_{\Omega_\delta}
\partial_3\partial_t(-\Delta)^{-\frac12}\nabla\times\Gamma^\alpha Y
  \cdot\partial_3(-\Delta)^{-\frac12}\nabla\times\sum_{\beta + \gamma
  = \alpha}C_\alpha^\beta (\nabla\Gamma^\beta
  Y)^\top(\partial_t^2 - \Delta)\Gamma^\gamma Ydy\\
&&\lesssim \delta^{h-\frac 12} \mathcal{F}_{\kappa-3}^{\frac{1}{2}}
   \sum_{\beta + \gamma= \alpha}\|\partial_3\big[(\nabla
  \Gamma^\beta Y)^{\top}(\partial_t^2 - \Delta)\Gamma^\gamma
  Y\big]\|_{L^2(\Omega_\delta)}.
\end{eqnarray}
Now we focus our mind on
\begin{align*}
& \delta^{h-\frac 12}
   \sum_{\beta + \gamma= \alpha}\|\partial_3\big[(\nabla
  \Gamma^\beta Y)^{\top}(\partial_t^2 - \Delta)\Gamma^\gamma
  Y\big]\|_{L^2(\Omega_\delta)}\\
&\lesssim
\delta^{h-\frac 12}\sum_{\beta + \gamma= \alpha}\big\|
|\partial_3\nabla\Gamma^\beta Y|\cdot|(\partial_t^2 - \Delta)\Gamma^\gamma Y|\big\|_{L^2(\Omega_\delta)}\\
&\quad +\delta^{h-\frac 12}
   \sum_{\beta + \gamma= \alpha}\big\|
   |\nabla\Gamma^\beta Y|\cdot| \partial_3(\partial_t^2 - \Delta)\Gamma^\gamma Y| \big\|_{L^2(\Omega_\delta)}.
\end{align*}
We will perform the estimate in two cases: $t\leq 1$ and $t\geq 1$.

We first take care of the case $t\leq 1$. By Lemma \ref{Sob2} and Lemma \ref{SNB-2}, there holds
\begin{align}\label{H-3}
&\delta^{h-\frac 12}\sum_{\beta + \gamma= \alpha}\big\|
|\partial_3\nabla\Gamma^\beta Y|\cdot|(\partial_t^2 - \Delta)\Gamma^\gamma Y|\big\|_{L^2(\Omega_\delta)}\\\nonumber
&\leq\delta^{h-\frac 12}\sum_{\beta + \gamma= \alpha,\ |\beta|\leq |\gamma|}
\|\partial_3\nabla\Gamma^\beta Y\|_{L^\infty(\Omega_\delta)}
\|(\partial_t^2 - \Delta)\Gamma^\gamma Y\|_{L^2(\Omega_\delta)}\\\nonumber
&\quad+\delta^{h-\frac 12}\sum_{\beta + \gamma= \alpha,\ |\beta|\geq |\gamma|}
\|\partial_3\nabla\Gamma^\beta Y\|_{L^2}\|(\partial_t^2 - \Delta)\Gamma^\gamma Y\|_{L^\infty(\Omega_\delta)}\\\nonumber
&\lesssim (\mathcal{E}_{\kappa-3}^{\frac12}+\mathcal{F}_{\kappa-3}^{\frac12})
(\mathcal{E}_{\kappa-3}+\mathcal{F}_{\kappa-3}).
\end{align}
Similarly, by Lemma \ref{Sob2} and Lemma \ref{SNB-2}, we have
\begin{align}
\delta^{h-\frac 12}\sum_{\beta + \gamma= \alpha}\big\|
|\nabla\Gamma^\beta Y|\cdot|\partial_3(\partial_t^2 - \Delta)\Gamma^\gamma Y|\big\|_{L^2(\Omega_\delta)}
\lesssim \mathcal{E}_{\kappa-3}^{\frac12}
(\mathcal{E}_{\kappa-3}+\mathcal{F}_{\kappa-3}). \nonumber
\end{align}

Next we consider the case for $t\geq 1$. We first deal with
\begin{align}\label{H-4}
&\delta^{h-\frac 12}\sum_{\beta + \gamma= \alpha}\big\|
|\partial_3\nabla\Gamma^\beta Y|\cdot|(\partial_t^2 - \Delta)\Gamma^\gamma Y|\big\|_{L^2(\Omega_\delta)}.
\end{align}
For \eqref{H-4}, if $|\beta|\leq |\gamma|$, by Lemma \ref{KS-A} and Lemma \ref{WE-2}, 
one has
\begin{align*}
&\delta^{h-\frac 12}\sum_{\beta + \gamma= \alpha,\ |\beta|\leq |\gamma|}\big\|
|\partial_3\nabla\Gamma^\beta Y|\cdot|(\partial_t^2 - \Delta)\Gamma^\gamma Y|\big\|_{L^2(\Omega_\delta)}\\\nonumber
&\lesssim \sum_{\beta + \gamma= \alpha,\ |\beta|\leq |\gamma|}
\delta^{l} \|\partial_3\nabla\Gamma^\beta Y\|_{L^\infty(\Omega_\delta)}
\delta^{m-\frac 12} \|(\partial_t^2 - \Delta)\Gamma^\gamma Y\|_{L^2(\Omega_\delta)}\\\nonumber
&\lesssim \sum_{|\beta|\leq [\alpha/2]}
\delta^{l} \|\partial_3\nabla\Gamma^\beta Y\|_{L^\infty(\Omega_\delta)}
\cdot t^{-1}(\mathcal{E}_{\kappa-3}+\mathcal{F}_{\kappa-3}).
\end{align*}
By Lemma \ref{K-S-O}, one has
\begin{align*}
 \sum_{|\beta|\leq [\alpha/2]}
\delta^{l} \|\partial_3\nabla\Gamma^\beta Y\|_{L^\infty(r\leq t/2)}
&\lesssim t^{-\frac12} (\mathcal{X}_{[\alpha/2]+5}^{\frac12}+
\mathcal{E}_{[\alpha/2]+5}^{\frac12}
+\mathcal{F}_{[\alpha/2]+5}^{\frac12}) \\\nonumber
&\lesssim t^{-\frac12} (
\mathcal{E}_{\kappa-3}^{\frac12}
+\mathcal{F}_{\kappa-3}^{\frac12}).
\end{align*}
On the other hand,
if $r\geq t/2$, since $r\leq r_v+1$, $t\geq 1$, $\delta\leq 1$, then
\begin{align*}
t/2\leq r\leq r_v+1.
\end{align*}
By Lemma \ref{Sob2} and Lemma \ref{KS-A}, one has
\begin{align*}
 \sum_{|\beta|\leq [\alpha/2]}
\delta^{l} \|\partial_3\nabla\Gamma^\beta Y\|_{L^\infty(r\geq t/2)}
&\lesssim (1+r_v)^{-\frac12} (
\mathcal{E}_{[\alpha/2]+5}^{\frac12}
+\mathcal{F}_{[\alpha/2]+5}^{\frac12}) \\\nonumber
&\lesssim t^{-\frac12} (
\mathcal{E}_{\kappa-3}^{\frac12}
+\mathcal{F}_{\kappa-3}^{\frac12}).
\end{align*}
For \eqref{H-4},
if $|\beta|\geq |\gamma|$, by Lemma \ref{KS-A} and Lemma \ref{WE-2}, \eqref{H-3}
is bounded by
\begin{align*}
&\delta^{h-\frac 12}\sum_{\beta + \gamma= \alpha,\ |\beta|\geq |\gamma|}\big\|
|\partial_3\nabla\Gamma^\beta Y|\cdot|(\partial_t^2 - \Delta)\Gamma^\gamma Y|\big\|_{L^2(\Omega_\delta)}\\\nonumber
&\lesssim \sum_{\beta + \gamma= \alpha,\ |\beta|\geq |\gamma|}
\delta^{l-\frac12} \|(1-\varphi^t)\partial_3\nabla\Gamma^\beta Y\|_{L^2(\Omega_\delta)}
\delta^{m} \|(1-\varphi^t)(\partial_t^2 - \Delta)\Gamma^\gamma Y\|_{L^\infty(\Omega_\delta)}\\\nonumber
&\quad+ \sum_{\beta + \gamma= \alpha,\ |\beta|\geq |\gamma|}
\delta^{l-\frac12} \|\varphi^t\partial_3\nabla\Gamma^\beta Y\|_{L^2(\Omega_\delta)}
\delta^{m} \|\varphi^t(\partial_t^2 - \Delta)\Gamma^\gamma Y\|_{L^\infty(\Omega_\delta)}\\\nonumber
&\lesssim t^{-2}  (\mathcal{E}_{\kappa-3}+\mathcal{F}_{\kappa-3})
+ \sum_{\beta + \gamma= \alpha,\ |\beta|\geq |\gamma|}
(\mathcal{E}_{|\beta|+2}^{\frac12}+\mathcal{F}_{|\beta|+2}^{\frac12})
\delta^{m-\frac 12} \|\varphi^t(\partial_t^2 - \Delta)\Gamma^\gamma Y\|_{L^\infty(\Omega_\delta)}\\\nonumber
&\lesssim t^{-2}  (\mathcal{E}_{\kappa-3}+\mathcal{F}_{\kappa-3})
+ t^{-\frac12} (\mathcal{E}_{\kappa-3}^{\frac12}+\mathcal{F}_{\kappa-3}^{\frac12})
\sum_{|\gamma|\leq [\alpha/2]+3} \delta^{m-\frac 12} \|\varphi^t(\partial_t^2 - \Delta)\Gamma^\gamma Y\|_{L^2(\Omega_\delta)}\\\nonumber
&\lesssim (t^{-2}+t^{-\frac32})  (\mathcal{E}_{\kappa-3}+\mathcal{F}_{\kappa-3}).
\end{align*}
Next, we take care of
\begin{align}\label{H-5}
&\delta^{h-\frac 12}\sum_{\beta + \gamma= \alpha}\big\|
|\nabla\Gamma^\beta Y|\cdot|\partial_3(\partial_t^2 - \Delta)\Gamma^\gamma Y|\big\|_{L^2(\Omega_\delta)}.
\end{align}
By Lemma \ref{Sob2}, Lemma \ref{KS-A}, Lemma \ref{K-S-O} and Lemma \ref{WE-3}, \eqref{H-5}
is bounded by
\begin{align*}
&\delta^{h-\frac 12}\sum_{\beta + \gamma= \alpha}\big\|
|\nabla\Gamma^\beta Y|\cdot|\partial_3(\partial_t^2 - \Delta)\Gamma^\gamma Y|\big\|_{L^2(\Omega_\delta)}\\\nonumber
&\lesssim \sum_{\beta + \gamma= \alpha}
\delta^{l} \|\nabla\Gamma^\beta Y\|_{L^\infty}
\delta^{m-\frac 12} \|\partial_3(\partial_t^2 - \Delta)\Gamma^\gamma Y\|_{L^2(\Omega_\delta)}\\\nonumber
&\lesssim t^{-\frac12}
\sum_{\beta + \gamma= \alpha} (\mathcal{E}_{|\beta|+4}+\mathcal{X}_{|\beta|+4})^{\frac12}
\delta^{m-\frac 12} \|\partial_3(\partial_t^2 - \Delta)\Gamma^\gamma Y\|_{L^2(\Omega_\delta)}\\\nonumber
&\lesssim t^{-\frac32} \mathcal{E}_{\kappa}^{\frac 12}
(\mathcal{E}_{\kappa-3}^{\frac 12}+\mathcal{F}_{\kappa-3}^{\frac 12}).
\end{align*}
Finally summing up all the estimate gives that
\begin{align}\label{B-3}
\frac{d}{dt} \sum_{|\alpha|\leq \kappa-5}
\delta^{2(h-\frac12)}\|\partial_3\partial(-\Delta)^{-\frac12}\nabla\times\Gamma^\alpha
Y\|_{L^2(\Omega_\delta)}^2
\lesssim
\langle t\rangle^{-\frac32}  \mathcal{F}_{\kappa-3}^{\frac 12} \mathcal{E}_{\kappa}^{\frac 12}
(\mathcal{E}_{\kappa-3}^{\frac12}+\mathcal{F}_{\kappa-3}^{\frac12}).
\end{align}
This gives the last differential inequality at the end of Section
\ref{Equations}.

\section*{Conflict of interest}
The authors declare that they have no conflict of interest.

\section*{Acknowledgement}

The authors would like to thank Prof. Zhen Lei for many stimulating discussions.
Y. Cai was supported by Hong Kong RGC Grant GRF 16308820.
F. Wang was supported by National Natural Science Foundation of China (grant no. 11701477) and the Hong Kong Scholars Program No. XJ2018047.


\end{document}